\documentclass[aihp]{imsart}

\usepackage{amsthm,amsmath,amsfonts,amssymb}
\usepackage[colorlinks,citecolor=blue,urlcolor=blue]{hyperref}
\usepackage{graphicx}
\usepackage{cleveref}
\usepackage{subfigure}
\usepackage{romannum}
\usepackage{romanbar}
\usepackage{bbm}
\usepackage{enumitem}
\usepackage{todonotes}

\pdfoutput=1

\numberwithin{equation}{section}

\theoremstyle{plain}
\newtheorem{theorem}{Theorem}[section]
\newtheorem{Cor}[theorem]{Corollary}
\newtheorem{lemma}[theorem]{Lemma}
\newtheorem{proposition}[theorem]{Proposition}
\newtheorem{assumption}{Assumption}

\theoremstyle{remark}

\newtheorem{remark}[theorem]{Remark}
\newtheorem{example}[theorem]{Example}

\newcommand{\norm}[1]{\left\lvert #1 \right\rvert} 

\def\rmd{\mathrm{d}}

\newcommand{\Var}{\mathrm{Var}}
\newcommand{\Law}{\mathrm{Law}}

\newcommand{\1}{\mathbbm{1}}

\newcommand{\W}{\boldsymbol{\mathcal{W}}}
\newcommand{\wnorm}[1]{\left\| #1 \right\|_{\mathsf{w}}} 
\newcommand{\rn}[1]{\Romanbar{#1}}

\newcommand{\law}{\operatorname{Law}}
\def\J{\mathrm{Couplings}}

\usepackage[backend=biber,backref=true,style=numeric,firstinits=true,maxbibnames=9]{biblatex}
\begin{document}

\pagenumbering{arabic}

\begin{frontmatter}
\title{Nonlinear Hamiltonian Monte Carlo \&
 its Particle Approximation}
\runtitle{Nonlinear HMC}

\begin{aug}
\Author[A]{\fnms{Nawaf} \snm{Bou-Rabee}\ead[label=e1]{nawaf.bourabee@rutgers.edu}}
\and
\Author[B]{\fnms{Katharina} \snm{Schuh}\ead[label=e2]{mmarsden@stanford.edu}}
\address[A]{Department of Mathematical Sciences \\ Rutgers University Camden \\ 311 N 5th Street \\ Camden, NJ 08102 USA \\ 
\href{mailto:nawaf.bourabee@rutgers.edu}{nawaf.bourabee@rutgers.edu}} 
\vspace{0.1in} 
\address[B]{Institute of Analysis and Scientific Computing \\ Technische Universität Wien \\ Wiedner Hauptstraße 8–10 \\ 1040 Wien, Austria \\ 
\href{mailto:katharina.schuh@tuwien.ac.at}{katharina.schuh@tuwien.ac.at}} 
\runAuthor{N. Bou-Rabee \and K. Schuh}
\end{aug}

\begin{abstract} 
    We present a nonlinear (in the sense of McKean) generalization of Hamiltonian Monte Carlo (HMC) termed \emph{nonlinear HMC} (nHMC) capable of sampling from nonlinear probability measures of mean-field type. When the underlying confinement potential is $K$-strongly convex and $L$-gradient Lipschitz, and the underlying interaction potential is gradient Lipschitz, nHMC can produce an $\varepsilon$-accurate approximation of a $d$-dimensional nonlinear probability measure in $L^1$-Wasserstein distance using $O((L/K) \log(1/\varepsilon))$ steps.  Owing to a uniform-in-steps propagation of chaos phenomenon, and without further regularity assumptions, unadjusted HMC with randomized time integration for the  corresponding particle approximation can achieve $\varepsilon$-accuracy in $L^1$-Wasserstein distance  using $O( (L/K)^{5/3} (d/K)^{4/3} (1/\varepsilon)^{8/3} \log(1/\varepsilon) )$ gradient evaluations.  These mixing/complexity upper bounds are a specific case of more general results developed in the paper for a larger class of non-logconcave, nonlinear probability measures of mean-field type.   
\end{abstract}



\end{frontmatter}

\maketitle

\section{Introduction}


Within the machine learning community, there is growing interest in gradient-based Markov chain Monte Carlo (MCMC) methods for sampling from nonlinear probability measures of mean-field type \cite{MeMoNg18,chizat2018global,sirignano2020mean,de2020quantitative,RoVa2022, HuReSiSz21}, i.e.,  probability measures  on $\mathbb{R}^d$ of the form  $\bar{\mu}_*(\rmd x) \propto \exp(- U_{\bar{\mu}_*}(x))\rmd x$ where the potential $U_{\bar{\mu}_*}: \mathbb{R}^d \to \mathbb{R}$ is implicitly defined by \begin{align} \label{nonlinear_target}
 U_{\bar{\mu}_*}(x) \ = \ V(x)+\epsilon \int_{\mathbb{R}^d} W(x,u)\bar{\mu}_*(\rmd u) \;,
\end{align}
where $\epsilon>0$ measures the interaction strength and $V: \mathbb{R}^d \to \mathbb{R}$, $W: \mathbb{R}^{2d} \to \mathbb{R}$ are continuously differentiable functions.     Despite this interest, however, MCMC methods for sampling from $\bar{\mu}_*$, and corresponding complexity guarantees, are still scarce and underdeveloped.   To be sure, classical MCMC methods are not immediately applicable because while the functions $V$ and $W$ are typically known and evaluable, the direct evaluation of $U_{\bar{\mu}_*}$ is often not feasible when $\epsilon>0$.  
    

On the other hand, there have been substantial theoretical developments in both the understanding of nonlinear processes and their particle approximations. The latter phenomenon of approximating a nonlinear process on $\mathbb{R}^d$ by an $N$-particle system in $\mathbb{R}^d$ is commonly termed \emph{propagation of chaos}, which reflects the statistical independence that  occurs between particles (or `chaos propagation') in the infinite particle limit, i.e.,  $N \nearrow \infty$. This phenomenon can be traced back to Kac's works on the Boltzmann equation \cite{Ka60,kac1959probability, mckean1967exponential} and plays a crucial role in McKean's groundbreaking work on nonlinear diffusions  \cite{Mc66,mckean1967propagation}. Using probabilistic techniques, Sznitman \cite{Sz91} and M\'el\'eard \cite{Me96} provide quantitative finite-in-time  bounds between a nonlinear process and its corresponding particle approximation. Subsequently, quantitative uniform-in-time propagation of chaos bounds for a large class of Markov processes (including hybrid jump-diffusions) were developed by Mischler, Mouhot, and Wennberg via functional analytic techniques \cite{MiMoWe15}; see also the companion papers \cite{MiMo13, HaMi14} and related approaches tailored to the kinetic case \cite{monmarche2017long,guillin2021kinetic,guillin2021uniform,guillin2022uniform}. The scope of functional analytic approaches is quite broad encompassing e.g.~singular interactions 
\cite{jabin2018quantitative, guillin2021uniformB, rosenzweig2023modulated, de2023sharp, Serfaty2020, RoSe2023}. An elementary coupling approach to uniform-in-time propagation of chaos has also been developed for various $\bar{\mu}_*$-preserving McKean-Vlasov processes \cite{DuEbGuZi20}; for various extensions, see \cite{DuEbGuSc21, GuLeMo21, Sc22, monmarche2023elementary}, and in particular, for first steps towards MCMC for $\bar{\mu}_*$, see \cite[\S 4.2]{monmarche2023elementary}.   Nonlinear overdamped/kinetic Langevin  diffusions have also been studied in the context of non-convex learning \cite{HuReSiSz21,KaReTaYa2020}.
For a broad, two-part survey on nonlinear processes, propagation of chaos, and applications, see  \cite{ChDi22a,ChDi22b}.  


In view of the gap between MCMC and $\bar{\mu}_*$-preserving nonlinear processes, the contributions of this paper are threefold: (i) presents  a \emph{nonlinear Hamiltonian Monte Carlo} (nHMC) and proves convergence of nHMC to $\bar{\mu}_*$ (\Cref{thm:contr_nonlinear}); (ii) proves the existence of a \emph{uniform-in-steps} propagation of chaos phenomenon  which justifies a mean-field particle approximation to nHMC (\Cref{thm:propofchaos_uniform}); and (iii) quantifies the complexity of an \emph{unadjusted} HMC (uHMC) algorithm with randomized time integration for the underlying Hamiltonian flow of the particle approximation (\Cref{thm:propofchaos_uniform_unad}). In conjunction, (i)-(iii) indicate that the nonlinear measure $\bar{\mu}_*$ can be sampled from by an implementable  uHMC algorithm with a quantitative complexity upper bound.  At this point, it is important to emphasize that the particle approximation in (ii) introduces a large asymptotic bias in uHMC w.r.t.~the nonlinear measure $\bar{\mu}_*$, and therefore, a Metropolis-adjustment which only eliminates the asymptotic bias due to time discretization is not considered here.


Carrying out (i) is  technically interesting since the nHMC kernel incorporates the flow of a $\bar{\mu}_*$-preserving nonlinear Hamiltonian system per transition step, and as a consequence, the corresponding Markov chain is not in general time-homogeneous.  To be specific, consider a nonlinear probability measure \eqref{nonlinear_target} where $V$ and $W$ are gradient-Lipschitz and $\nabla V$ is \emph{strongly co-coercive} outside a Euclidean ball (a notion made precise in Assumption~\ref{ass_Vconv}); importantly, this assumption allows for non-convex $V$ (see Example~\ref{ex:multiwell} and Remark~\ref{rmk:ass_Vconv}). In this setting, by generalizing the coupling approach in \cite{BoEbZi2020},
\Cref{thm:contr_nonlinear}  proves that if the interaction and duration parameters sufficiently small (but independent of the dimension $d$), then the nHMC transition kernel is contractive in a particular $L^1$-Wasserstein distance equivalent to the standard one. As consequences, we obtain uniqueness of the invariant measure $\bar{\mu}_*$ of nHMC and convergence of the nHMC transition kernel $\bar{\pi}$ towards $\bar{\mu}_*$ in the standard $L^1$-Wasserstein distance, i.e., for any $\nu\in\mathcal{P}(\mathbb{R}^d)$,
\begin{align} \label{eq:intro_contr_nHMC}
\W^1(\bar{\mu}_*, \nu \bar{\pi}^m) \ \le \ \mathbf{A} e^{-cm}\W^1(\bar{\mu}_*, \nu) \;,  \qquad \text{for any } m\in\mathbb{N},
\end{align}
where the contraction rate $c$ and the constant $\mathbf{A}$ are both independent of the dimension $d$.  For notational brevity, we have suppressed the dependence on the underlying distribution in the nHMC transition kernel. 

This convergence result for nHMC motivates (ii), i.e., 
identifying a propagation of chaos phemonenon for nHMC. 
Intuitively, this phenomenon  states that the nHMC chain can be approximated arbitrarily well by an  \emph{exact} HMC (xHMC) chain for an $N$-particle system with  invariant  measure $\mu_*(dx) \propto \exp(-U(x)) dx$ on $\mathbb{R}^{N d}$ where $U: \mathbb{R}^{Nd} \to \mathbb{R}$ is 
 \[
 U(x) = \sum_{i=1}^N \Big(V(x^i)+\frac{\epsilon}{2N}\sum_{j=1}^N W(x^i,x^j)\Big) \;, \qquad x = (x^1, \dots, x^N ) \in \mathbb{R}^{Nd} \;.
\]  Here the modifier \emph{exact} signifies that the exact Hamiltonian flow of the $N$-particle system is used per transition step.  By coupling $N$ independent copies of nHMC with a single copy of xHMC applied to this $N$-particle system, and analogous to related but different   propagation of chaos arguments for nonlinear kinetic Langevin diffusions \cite{DuEbGuZi20, DuEbGuSc21, GuLeMo21, Sc22}, \Cref{thm:propofchaos_uniform} proves that for any $\nu\in\mathcal{P}(\mathbb{R}^d)$ and $m \in \mathbb{N}$, there exists a uniform-in-$m$ progagation of chaos phenomenon
\begin{align} \label{eq:intro_propofchaos}
\W_{\bar{\ell}_1^N}^1(\nu^{\otimes N}\pi^m,(\nu\bar{\pi}^m)^N) \ \le \ \mathbf{B} N^{-1/2} \;,  \qquad \text{for any } m\in\mathbb{N},
\end{align}
where $\W_{\bar{\ell}_1^N}^1$ is the $L^1$-Wasserstein distance with respect to  the averaged $\ell^1$-distance $\bar{\ell}_1^N(x,y)=N^{-1}\sum_{i=1}^N|x^i-y^i|$ and $\pi$ denotes the transition kernel of xHMC for the $N$-particle system. Importantly, the constant $\mathbf{B}$ is independent of both the number of transition steps $m$ and the number of particles $N$.

Since xHMC is rarely implementable on a computer, in (iii), we consider  uHMC with randomized time integration for the underlying Hamiltonian flow of the $N$-particle system.  This randomization leads to a better complexity without the need for  Hessian Lipschitz assumptions on $V$ and/or $W$ \cite{BouRabeeMarsden2022}.  Building on recent work on uHMC for an $N$-particle mean-field model \cite{BouRabeeSchuh2023} and uHMC with randomized time integration \cite{BouRabeeMarsden2022},  \Cref{thm:propofchaos_uniform_unad} proves that for any $\nu\in\mathcal{P}(\mathbb{R}^d)$,
\begin{align} \label{eq:intro_accuracy_uHMC}
\W_{\bar{\ell}_1^N}^1(\nu^{\otimes N}\pi_h^m,(\nu\bar{\pi}^m)^N) \ \le \ \mathbf{B} N^{-1/2}  +  \mathbf{C} h^{3/2}  \;,  \qquad \text{for any } m\in\mathbb{N},
\end{align}
where $\pi_h$ denotes the transition kernel of uHMC with step size $h>0$.
Crucially, for any accuracy $\varepsilon>0$,   the number of steps $m$, the step size $h$ and the number of particles $N$ needed to obtain $\varepsilon$-accuracy in $L^1$-Wasserstein distance w.r.t.~the nonlinear measure $\bar{\ell}_1^N$ can be read off of Theorem~\ref{thm:propofchaos_uniform_unad}; notably, this complexity upper bound holds for arbitrary initial distributions including cold starts.  For a promising new alternative approach to quantify the complexity of uHMC via functional inequalities, see the very recent work of Monmarch\'{e} on xHMC \cite{monmarche2022entropic} and the sequel paper for uHMC \cite{camrud2023second}.

\section*{Acknowledgements}
N.~Bou-Rabee has been partially supported by the National Science Foundation under Grant No.~DMS-2111224.

\section{Main results}
\subsection{Nonlinear HMC}
Let $V:\mathbb{R}^d\to\mathbb{R}$,  $W:\mathbb{R}^{2d}\to\mathbb{R}$ be continuously differentiable functions.
For $\mu\in\mathcal{P}(\mathbb{R}^d)$ define $U_{\mu}:\mathbb{R}^d\to\mathbb{R}$ by
\begin{align} \label{eq:U_mu}
	U_\mu(x)=V(x)+\epsilon\int_{\mathbb{R}^d}W(x,u)\mu(\rmd u),
\end{align}
where $\epsilon>0$ is a positive constant. We suppose that there exists a probability distribution $\bar{\mu}_*$ of the form 
\begin{align} \label{eq:invmeas_nonl}
	\bar{\mu}_*(\rmd x)=\frac{1}{Z}\exp(-U_{\bar{\mu}_*}(x))\rmd x \quad \text{with } Z=\int_{\mathbb{R}^d}\exp(-U_{\bar{\mu}_*}(x))\rmd x. 
\end{align}
We emphasize that this is equivalent to assuming that there exists a solution $U_{\bar{\mu}_*}$ to 
\begin{align}
    U_{\bar{\mu}_*}(x)=V(x)+\epsilon\int_{\mathbb{R}^d}W(x,u)Z^{-1}\exp(-U_{\bar{\mu}_*}(u))\rmd u.
\end{align}
\emph{Nonlinear HMC} (nHMC) is an MCMC method to sample from the nonlinear target measure $\bar{\mu}_*$ by generating a Markov chain on $\mathbb{R}^d$ using the flow of the following $\bar{\mu}_*$-preserving Hamiltonian dynamics \begin{align} \label{eq:hamdyn_nonl}
	\begin{cases}
		\frac{\rmd}{\rmd t} \bar{q}_t(x,v,\bar{\mu}) =\bar{p}_t(x,v,\bar{\mu})
		\\ \frac{\rmd}{\rmd t} \bar{p}_t(x,v,\bar{\mu}) =-\nabla U_{\bar{\mu}_t}(\bar{q}_t(x,v,\bar{\mu}))=-\nabla V(\bar{q}_t(x,v,\bar{\mu}))-\epsilon \int_{\mathbb{R}^d} \nabla_1 W(\bar{q}_t(x,v,\bar{\mu}),u) \bar{\mu}_t(\rmd u)
	\end{cases}
\end{align}
with initial condition $(\bar{q}_0,\bar{p}_0,\bar{\mu}_0)=(x,v,\bar{\mu})\in\mathbb{R}^{2d}\times \mathcal{P}(\mathbb{R}^d)$ where $\bar{\mu}_t :=\Law(\bar{q}_t(\tilde{x},\tilde{v},\bar{\mu}))$ with $(\tilde{x},\tilde{v})\sim \bar{\mu}\otimes\mathcal{N}(0,I_d)$. Compared to  classical Hamiltonian dynamics, this nonlinear Hamiltonian dynamics depends both on the current state $(\bar q_t, \bar p_t)$ and the law of $\bar q_t$, i.e.,  $\bar{\mu}_t$.  The nHMC transition step with complete velocity refreshment is defined by
\begin{align} \label{eq:transstep_nonl}
	\bar{\mathbf{X}}^{\bar{\mu}}(x) =\bar{q}_T(x,\xi,\bar{\mu}),
\end{align}
where $T>0$ is a duration parameter and $\xi\sim\mathcal{N}(0,I_d)$, i.e., the law of the initial velocity $\xi$ is $d$-dimensional standard normal. The corresponding transition kernel of nHMC is given by 
\begin{align} \label{eq:transkern_nonl}
	\bar{\pi}_{\bar{\mu}}(x,A) =\mathbb{P}[\bar{q}_T(x,\xi,\bar{\mu})\in A]
\end{align}
for any measurable set $A\subseteq \mathbb{R}^d$.
Note that the dependence of each transition step on the terminal law from the previous step implies that the nHMC chain is time-inhomogeneous, which differs from classical MCMC methods.
If there exists a probability measure $\bar{\mu}$ such that $\bar{\mu}=\bar{\mu}\bar{\pi}_{\bar{\mu}}$, then the nHMC chain initialized from $\bar{\mu}$  becomes time-homogeneous. In fact, the following result holds.

\begin{proposition} \label{prop:inv_meas}
If $\bar{\mu}_*$ given in \eqref{eq:invmeas_nonl} is well-defined, then $\bar{\pi}_{\bar{\mu}_*}$ leaves the measure $\bar{\mu}_*$ invariant, i.e. $\bar{\mu}_*\bar{\pi}_{\bar{\mu}_*}=\bar{\mu}_*$.
\end{proposition}

\begin{proof}
    Consider the Hamiltonian dynamics
    \begin{align} \label{eq:hamdyn_fixedmeas}
\begin{cases}
    \frac{\rmd }{\rmd t}{q}_t(x,v)={p}_t(x,v)
    \\ \frac{\rmd}{\rmd t} {p}_t(x,v)=-\nabla U_{\bar{\mu}_*}({q}_t(x,v)),
\end{cases}
\end{align}
where $U_{\bar{\mu}_*}=V(x)+\epsilon \int_{\mathbb{R}^d}W(x,u)\bar{\mu}_*(\rmd u)$ with $\bar{\mu}_*$ given in \eqref{eq:invmeas_nonl}. The corresponding Hamiltonian flow preserves the Hamiltonian $H_{\bar{\mu}_*}(x,v)=U_{\bar{\mu}_*}(x)+|v|^2/2$ and volume, see \cite{MaRa94}. Hence, $\bar{\mu}_*$ is an invariant probability measure for the transition kernel $\pi(x,A)=\mathbb{P}[q_T(x,\xi)\in A]$ for any measurable set $A\subseteq \mathbb{R}^d$ where $\xi\sim\mathcal{N}(0, I_d)$; see e.g. \cite{Ne11,BoSa18}. In particular, if $\Law(x)=\bar{\mu}_*$, then $\Law(q_t(x,\xi))=\bar{\mu}_*$ for all $t\in[0,T]$ and the dynamics \eqref{eq:hamdyn_fixedmeas} coincides with \eqref{eq:hamdyn_nonl} with initial distribution $\bar{\mu}_*$. Hence, for this initial distribution, $\bar{\pi}_{\bar{\mu}_*}=\pi$ and $\bar{\mu}_*\bar{\pi}_{\bar{\mu}_*}=\bar{\mu}_*$.
\end{proof}

\subsection{Unadjusted HMC for the Particle Approximation}

A key idea in this work is that the nonlinear Hamiltonian flow satisfying \eqref{eq:hamdyn_nonl} can be approximated arbitrarily well by the classical Hamiltonian flow $(q_t(x,v), p_t(x,v))$ of a mean-field $N$-particle system on $\mathbb{R}^{N d}$ satisfying 
\begin{align} \label{eq:hamdyn_exact}
\begin{cases}
    \frac{\rmd }{\rmd t}{q}^i_t(x,v)={p}_t^i(x,v)
    \\ \frac{\rmd}{\rmd t} {p}^i_t(x,v)=-\nabla_i U({q}_t(x,v))
\end{cases}  \text{for  } i\in\{1,\ldots,N\}
\end{align}
with initial condition $ (x,v)\in\mathbb{R}^{Nd}$ and mean-field potential $U:\mathbb{R}^{Nd}\to\mathbb{R}$ defined by 
\begin{align} \label{eq:U_meanf}
	U(x)=\sum_{i=1}^N \Big(V(x^i)+\frac{\epsilon}{2N}\sum_{j=1}^N W(x^i, x^j)\Big) \qquad \text{where  } x=(x^1, \ldots, x^n) \;. 
\end{align}
Here $\nabla_i U \equiv \partial U/\partial x^i$.  Assuming $\int_{\mathbb{R}^{Nd}}\exp(-U(x))\rmd x <\infty$, a key property of \eqref{eq:hamdyn_exact} is that it leaves invariant the mean-field probability measure on $\mathbb{R}^{N d}$ given by 
\begin{align} \label{eq:invmeas_meanf}
	\mu_*(\rmd x)\propto \exp\Big(-\sum_{i=1}^N \Big(V(x^i)+\frac{\epsilon}{2N}\sum_{j=1}^N W(x^i,x^j)\Big) \Big)\rmd x.
\end{align} 
A transition step of \emph{exact HMC} (xHMC) is defined by
\begin{align} \label{eq:transstep_exact}
	\mathbf{X}(x)=\{q_T^i(x,\xi)\}_{i=1}^N \quad \text{with } \xi\sim \mathcal{N}(0,I_{Nd}) \;.
\end{align}
  The corresponding transition kernel is denoted by $\pi$ and satisfies $\mu_* \pi = \mu_*$. Since the exact flow of \eqref{eq:hamdyn_exact} is often unavailable, xHMC is not in general implementable, and therefore, we consider a time discretization of the exact Hamiltonian flow of the $N$-particle system $(q_t(x,v), v_t(x,v))$.

\smallskip

 Among time integrators for \eqref{eq:hamdyn_exact}, a randomized time integrator is selected since it substantially improves the complexity of the corresponding uHMC algorithm  without assuming a Hessian Lipschitz condition on either $V$ and/or $W$ \cite{BouRabeeMarsden2022}.  The integrator is akin to the randomized midpoint method for the kinetic Langevin diffusion \cite{shen2019randomized,Cao_2021_IBC,ErgodicityRMMHYB,leimkuhler2023contractionA,leimkuhler2023contractionB}.  The flow $(Q_t(x,v), P_t(x,v))$ of the randomized time integrator satisfies
\begin{align} \label{eq:hamdyn_meanf_num}
	\begin{cases}
		\frac{\rmd}{\rmd t} Q_t^i(x,v)= P_{ t }^i(x,v)
		\\ \frac{\rmd}{\rmd t} P_t^i(x,v)=-\nabla_i U \big( Q_t^{\star}(x,v) \big) 
	\end{cases} \text{for } i\in\{1,\ldots,N\}
\end{align}
where $h>0$ is a time step size; $\{ \mathcal{U}_k \}_{k \in \mathbb{N}_0}$ is an i.i.d.~sequence of standard uniform random variables; $Q_t^{\star}(x,v)$ is a random evaluation point defined by $Q_t^{\star}(x,v):=Q_{\lfloor t \rfloor}(x,v) + h \, \mathcal{U}_{\lfloor t \rfloor/h}   \, P_{\lfloor t \rfloor}(x,v)$;   and we  introduced the shorthand
\begin{align*}
	\lfloor t \rfloor=\max\{s \in h\mathbb{Z} : s \le t\} \quad \text{and} \quad \lceil t \rceil=\min\{s \in h\mathbb{Z} : s \ge t\}.
\end{align*}
Replacing the exact flow in \eqref{eq:transstep_exact} with this approximate flow yields \emph{unadjusted HMC} (uHMC) with transition step
\begin{align} \label{eq:transstep_meanf}
	\mathbf{X}_h(x)=\{Q_T^i(x,\xi)\}_{i=1}^N \quad \text{with } \xi\sim \mathcal{N}(0,I_{Nd}) \;.
\end{align}
 The uHMC transition kernel is denoted by $\pi_h$.
In this work, we do not consider the corresponding Metropolis-adjusted HMC algorithm for sampling from the mean-field probability measure $\mu_*$, because the asymptotic bias due to the particle approximation cannot be eliminated by Metropolis-adjustment. Since we avoid Metropolis-adjustment, time step or duration adaptivity can be easily included in \eqref{eq:transstep_meanf} --- as in \cite{HoGe2014,BoSa2017,kleppe2022, hoffman2022tuning, whalley2022randomized}.
For promising alternative strategies to the particle approximation considered here, see \cite{GOBET201871,gomes2020mean}.

\begin{table}
    \centering
    \begin{tabular}{c|c|c|c|c}
        method & short form & transition step & defined in & transition kernel \\  \hline 
        nonlinear HMC & nHMC & $\bar{\mathbf{X}}^{\bar{\mu}}(x)$ & \eqref{eq:transstep_nonl} & $\bar{\pi}_{\bar{\mu}}$ \\
        exact HMC & xHMC &  $ \mathbf{X}(x) $ & \eqref{eq:transstep_exact} & $\pi$ \\
        unadjusted HMC &  uHMC &  $\mathbf{X}_h(x)$ & \eqref{eq:transstep_meanf} & $\pi_h$ \\
        \hline
    \end{tabular}
    \caption{Frequently Used Notation}
    \label{tab:my_label}
\end{table}

\subsection{Assumptions}

The main results of this paper rely on the following assumptions on $V$ and $W$.  We stress that $\nabla_1 W \equiv \partial W/\partial x^1$ represents the gradient of $W$ in its first component.
\begin{assumption} \label{ass} Let $V:\mathbb{R}^d\to\mathbb{R}$ and $W:\mathbb{R}^{2d}\to\mathbb{R}$ be continuously differentiable functions satisfying:
\begin{enumerate}[label=\textbf{(\alph*)}]
\setlength{\itemsep}{0pt}
\item\label{ass_Vmin} $V$ has a global minimum at $0$, $V(0)=0$ and $V(x)\geq 0$ for all $x\in\mathbb{R}^d$. 
\item\label{ass_Vlip} $V$ is $L_1$-gradient Lipschitz, i.e., there exists $L_1 > 0$ such that $|\nabla V(x) - \nabla V(y)| \le L_1 |x-y|$ for all $x, y \in \mathbb{R}^d$.
\item\label{ass_Vconv} $V$ is asymptotically strongly co-coercive, i.e., there exist $\mathcal{R} \ge 0 $,  $K > 0$ and $L_2 > 0$ such that
\begin{align*}
\langle x-y, \nabla V(x)-\nabla V(y) \rangle \ \geq \ K \norm{x-y}^2 + \frac{1}{L_2} \norm{\nabla V(x)-\nabla V(y)}^2 \;, \quad \text{for all $x,y\in\mathbb{R}^d$ with $|x-y|\geq \mathcal{R}$.}
\end{align*}
\item\label{ass_Wlip} $W$ is symmetric, i.e., $W(x,y)=W(y,x)$ and $\tilde{L}$-gradient Lipschitz, i.e., there exists $\tilde{L} \ge 0$ such that \begin{align*}
|\nabla_1 W(x,y) - \nabla_1 W(\tilde{x},\tilde{y})| \ \le \ \tilde{L} (|x - \tilde{x}| + |y-\tilde{y}|) \quad \text{for all $x,y,\tilde{x},\tilde{y} \in\mathbb{R}^d$  }  \;.  
\end{align*}
\end{enumerate}
\end{assumption}
\noindent
\emph{ 
Note that if \ref{ass_Vlip} and \ref{ass_Vconv} hold for $L_1$ and $L_2$, then these conditions also necessarily hold for $\max(L_1,L_2)$.
Therefore, for the sake of clarity/readability, all of the  conditions and results will  often be stated in terms of $L:=\max(L_1,L_2)$. } %

\medskip

Referring to \Cref{ass}, \ref{ass_Vmin} can always be obtained by adjusting $V$. By \ref{ass_Vmin}, \ref{ass_Vlip} and \ref{ass_Wlip}, note that 
\begin{align*}
|\nabla V(x)| \ = \ |\nabla V(x)-\nabla V(0)| \ \leq \  L_1 |x|~~\text{and}
~~|\nabla_1 W(x,y)| \ \leq \  \tilde{L}(|x|+|y|)+\mathbf{W}_0  \quad \text{for all $x,y\in\mathbb{R}^d$} \;,
\end{align*}
with $\mathbf{W}_0=|\nabla_1 W(0,0)|$.

\begin{remark}[Notion of Asymptotic Strong Co-coercivity of $\nabla V$] \label{rmk:ass_Vconv}
As far as we can tell, \Cref{ass}~\ref{ass_Vconv} is new and therefore deserves some additional comments.  It allows for multi-well potentials (see Example~\ref{ex:multiwell}) and can be viewed as a strengthening of the asymptotic strong convexity condition appearing in previous works proving Wasserstein mixing time upper bounds in non-convex settings; see, e.g.,  \cite{cheng2018sharp,BoEbZi2020,BouRabeeSchuh2023}.   If $V$ is $\Lambda$-gradient Lipschitz and (globally) $m$-strongly convex,  then \ref{ass_Vconv} holds with $\mathcal{R}=0$, $K=m \Lambda/(m+\Lambda)$ and $L_2 = m+\Lambda$ \cite[Theorem 2.1.12]{nesterov2018lectures}.  Therefore, it is natural to assume this condition also holds outside a Euclidean ball.   
Moreover, by ~\ref{ass_Vlip} and \ref{ass_Vconv}, note that
\begin{align} \label{eq:V_asymp_stronglyconv}
\langle x-y, \nabla V(x)-\nabla V(y) \rangle  \ \geq \  K \norm{x-y}^2 + \frac{1}{L_2} \norm{\nabla V(x) - \nabla V(y)}^2 - \hat{C} \qquad \text{ for all } x,y\in\mathbb{R}^d
\end{align}
with $\hat{C}=\frac{L_1 L_2 + K L_2 + L_1^2 }{L_2} \mathcal{R}^2$.  Repeating this computation  with \ref{ass_Vlip} and \ref{ass_Vconv} in terms of $L$, one finds that $\hat{C} = (2 L + K) \mathcal{R}^2$.  On the other hand, if $V$ is  $\Lambda$-gradient Lipschitz but only asymptotically $m$-strongly convex, i.e., there exist $m>0$ and $\Upsilon \ge 0$ such that \[
\langle x-y, \nabla V(x)-\nabla V(y) \rangle  \ \geq \ m \norm{x-y}^2 - \Upsilon,  \quad \text{for all $x,y \in \mathbb{R}^d$} \;, 
\] then \ref{ass_Vconv} holds with $\mathcal{R}= \sqrt{2 \Upsilon/m}$, $K=m/4$ and $L_2 = (4 \Lambda^2)/m$.  Conversely, note that \ref{ass_Vconv} implies $V$ is asymptotically strongly convex, and by the Cauchy-Schwarz inequality, $L_2$-gradient Lipschitz for $\norm{x-y} \ge \mathcal{R}$; hence, $K \le L_2$. In addition to being natural, another advantage of \ref{ass_Vconv} is that it yields theoretical results that interpolate between  these cases.     
\end{remark}


\begin{example}
\label{ex:multiwell}
A concrete example of a multi-well function $V: \mathbb{R}^d \to \mathbb{R}$ that satisfies \Cref{ass}~\ref{ass_Vconv} is defined by 
$V(x) = (1/2) \norm{x}^2 + e^{-(a/2) \norm{x}^2}$ where $a \ge 0$.  In particular, it is easily seen that $\nabla^2 V(0)$ is not positive definite if $a \ge 1$, and therefore, $V$ is not globally strongly convex.  Nonetheless, $V$ satisfies  \begin{align*}
& \langle \nabla V(x) - \nabla V(y) , x - y \rangle - \frac{1}{2} \norm{x-y}^2 - \frac{1}{2} \norm{\nabla V(x) - \nabla V(y)}^2 \ = \ - \frac{1}{2} \norm{ \nabla V(x) - \nabla V(y) - (x-y)}^2 \\
& \qquad \ = \  - \frac{a^2}{2} \norm{  e^{-\frac{a}{2} \norm{x}^2} x  - e^{-\frac{a}{2} \norm{y}^2} y  }^2 \ \ge \  - a^2  e^{-\frac{a}{2}\norm{x}^2} \norm{ x }^2  - a^2  e^{-\frac{a}{2} \norm{y}^2} \norm{y  }^2 \ \ge \ - \frac{4 a}{e} \;.
\end{align*} Therefore, for all $x,y \in \mathbb{R}^d$ such that $\norm{x - y} \ge \mathcal{R} = 4 \sqrt{a/e}$, \ref{ass_Vconv} holds with $K = 1/4$ and $L_2 = 2$.  \end{example}

\begin{remark}
For the main results that follow, it is possible to replace the gradients $\nabla V$ and $\nabla_1 W$ in \eqref{eq:hamdyn_nonl}, \eqref{eq:hamdyn_exact}, and \eqref{eq:hamdyn_meanf_num} by more general drifts $b:\mathbb{R}^d\to\mathbb{R}^d$ and $\tilde{b}:\mathbb{R}^{2d}\to\mathbb{R}^d$, satisfying assumptions analogous to \Cref{ass}, but not necessarily derivable from a potential.  However, for forces not derivable from a potential, a possibly nonlinear formula for the density of the corresponding invariant measures may no longer be available.
\end{remark}

\subsection{Couplings \& Metrics}
To prove contractivity of nHMC, our strategy is to use a coupling of $ \bar{\pi}_{\bar{\mu}}(\bar{x},\cdot)$ and $\bar{\pi}_{\bar{\mu}'}(\bar{x}',\cdot)$ 
 based on coupling their underlying random initial velocities  following \cite[\S 2.3]{BoEbZi2020}; in particular, the coupling transition step is given by 
\begin{align} \label{eq:coupl_transstep}
\bar{\mathbf{X}}^{\bar{\mu}}(\bar{x},\bar{x}') \ = \ \bar{q}_T(\bar{x},\xi,\bar{\mu}) \qquad \text{and} \qquad \bar{\mathbf{X}}^{\bar{\mu}'}(\bar{x},\bar{x}') \ = \ \bar{q}_T(\bar{x}',\eta,\bar{\mu}'),
\end{align}
where $\xi$ and $\eta$ are defined on the same probability space such that $\xi\sim \mathcal{N}(0,I_d)$ and $\eta$ is defined in cases by: if $|\bar{x}-\bar{x}'|\geq \tilde{R}$ for a constant $\tilde{R}$ which is specified later, then $\eta \ = \  \xi$; and otherwise,  
\begin{align*}
\eta \ = \  \begin{cases} \xi+\gamma \bar{z} & \text{if } \mathcal{U}\leq \dfrac{\varphi_{0,1}(e\cdot\xi+\gamma|\bar{z}|)}{\varphi_{0,1}(e\cdot\xi)} , \\
\xi-2(e\cdot\xi)e & \text{else,} \end{cases}
\end{align*}
where $\mathcal{U}\sim \mathrm{Unif}[0,1]$ is independent of $\xi$, $\varphi_{0,1}$ denotes the density of the standard normal distribution, $\bar{z}=\bar{x}-\bar{x}'$, and $e=\bar{z}/|\bar{z}|$ if $|\bar{z}|\neq 0$ and otherwise $e$ is an arbitrary unit vector.
Here the coupling parameter $\gamma>0$ is given by
\begin{align}
\gamma \ := \ \min(T^{-1},\tilde{R}^{-1}/4). \label{eq:gamma}
\end{align}
By \cite[\S 2.3.2]{BoEbZi2020}, the law of $\eta$ is $\mathcal{N}(0,I_d)$, and hence, \eqref{eq:coupl_transstep} is indeed a coupling of $ \bar{\pi}_{\bar{\mu}}(\bar{x},\cdot)$ and $\bar{\pi}_{\bar{\mu}'}(\bar{x}',\cdot)$.


Contractivity of the nHMC kernel is proved in terms of an $L^1$-Wasserstein distance w.r.t.~a particular distance  which involves a metric-preserving function; cf. \cite[\S 2.5.2]{BoEbZi2020}.  In particular, we consider the distance $\rho:\mathbb{R}^d\times\mathbb{R}^d\to\mathbb{R}_+$ given by
\begin{align} \label{eq:rho}
\rho(x,y) \ = \ f(|x-y|),
\end{align}
where $f:\mathbb{R}_+\to\mathbb{R}_+$ is the following metric-preserving function  
\begin{align} \label{eq:f}
f(r) \ &= \  \int_0^r \exp(-\min(R_1,s)/T)\rmd s \quad \text{where} \\
\tilde{R} \ &:= \ \sqrt{(1/6)(2L+K)/K}\mathcal{R} 
\qquad  \text{and} \qquad R_1 \ := \ (5/4)(\tilde{R}+2T).  \label{eq:R1}
\end{align}
Note that $\mathcal{R}$, $K$ and $L$ are given in \Cref{ass}; and the  distance $\rho$ is equivalent to the Euclidean distance  on $\mathbb{R}^d$ since
\begin{align} \label{eq:distance_equiv}
rf'(R_1) \ \leq \  f(r) \ \leq \  r.
\end{align} To prove the existence of a propagation of chaos phenomenon, we use the distance $\rho_N:\mathbb{R}^{Nd}\times\mathbb{R}^{Nd}\to[0,\infty)$ given by
\begin{align} \label{eq:rho^N}
\rho_N(x,y) \ = \ \frac{1}{N}\sum_{i=1}^N f(|x^i-y^i|),
\end{align}
where $f$ is the metric-preserving function given in \eqref{eq:f}.
By \eqref{eq:distance_equiv}, $\rho_N$ is equivalent to the averaged $\ell^1$-distance
\begin{align} \label{eq:ell^1}
\bar{\ell}_1^N(x,y) \ = \ \frac{1}{N}\sum_{i=1}^N |x^i-y^i|.
\end{align}

\subsection{Convergence of Nonlinear HMC}

\label{sec:convnhmc}

Remarkably, the nHMC transition kernel is a contraction in a particular Wasserstein distance provided both the duration parameter $T$ and the interaction parameter $\epsilon$  are suitably small as indicated in the next theorem.  To precisely state this result, for a metric $\mathsf{d}$  on $\mathbb{R}^d$, and
for probability measures $\nu,\eta \in \mathcal{P}(\mathbb{R}^d)$,
define the $L^1$-Wasserstein distance w.r.t.~$\mathsf{d}$ by
\[
\W_{\mathsf{d}}^1(\nu,\eta) \ := \ \inf \Big\{ \mathbb{E}\left[\mathsf{d}(X,Y)  \right] ~:~ \law(X, Y) \in \J(\nu,\eta) \Big\} \;.
\] When $\mathsf{d}$ is the standard Euclidean metric,  the corresponding $L^1$-Wasserstein distance is simply written as $\W^1$.  

\begin{theorem}[Contraction for nHMC] \label{thm:contr_nonlinear}
Let $\bar{\mu}$ and $\bar{\mu}'$ be two probability measures on $\mathbb{R}^d$ with finite second moment.
Suppose \Cref{ass} holds. Assume $T \in (0,\infty)$ and $\epsilon\in[0,\infty)$ satisfy
\begin{align}
LT^2 \ &\le \ \frac{3}{5}\min\Big(\frac{1}{4},\frac{3}{256\cdot 5 L\tilde{R}^2}\Big)  \label{eq:cond_T}
\\  \epsilon \tilde{L} \ &\le \  \frac{5}{64} \frac{K}{\sqrt{(7/6)+3/(2KT^2)}}\exp\left(-\frac{5\tilde{R}}{2T}-5\right) .\label{eq:cond_eps} 
\end{align} 
Let $\rho$ be the distance defined in \eqref{eq:rho}.  Then
\begin{align} \label{eq:contraction}
& \mathbb{E}_{x\sim \bar{\mu},x'\sim \bar{\mu}'} [\rho(\bar{\mathbf{X}}^{\bar{\mu}}(x,x'),\bar{\mathbf{X}}'^{\bar{\mu}'}(x,x'))] \ \leq \  (1-c) \, \mathbb{E}_{x\sim \bar{\mu},x'\sim \bar{\mu}'}[\rho(x,x')]  \quad \text{where}  \\
 \label{eq:c}
& c \ = \ (KT^2/156)\exp(-5\tilde{R}/4T).
\end{align}
\end{theorem}

\begin{proof}
The proof is postponed to \Cref{sec:convnhmc_proofs}.
\end{proof}

As straightforward consequences of \Cref{thm:contr_nonlinear}, we obtain uniqueness of an invariant  measure and exponential $\W^1$-convergence.  To simplify notation, we write $\mu \bar{\pi}=\mu \bar{\pi}_\mu$ for $\mu\in\mathcal{P}(\mathbb{R}^d)$, and denote the distribution of the nHMC chain after $m\in\mathbb{N}$ steps by $\mu \bar{\pi}^m$.
For alternative methods to prove existence/uniqueness of an invariant measure under possibly more mild conditions, we highlight \cite{hairer2011asymptotic,butkovsky2014ergodic,eberle2019quantitative,bao2022existence}.  
Without the weak interaction condition in \eqref{eq:cond_eps},  we stress that nHMC may admit multiple invariant measures and exhibit phase transitions, as in \cite{dawson1983critical, HeTu10,gomes2018mean, gomes2020mean, pavliotis2023}.


\begin{Cor}[Uniqueness of invariant nonlinear measure] \label{cor:invmeas}
In the situation of \Cref{thm:contr_nonlinear}, it holds for $m\in\mathbb{N}$, \begin{equation} \label{eq:W_contraction}
\begin{aligned}
 \W^1_{\rho}(\bar{\mu}\bar{\pi}^m ,\bar{\mu}'\bar{\pi}^m)  \  \leq \ e^{-cm}\W^1_{\rho}(\bar{\mu}, \bar{\mu}') \qquad \text{and} \qquad
 \W^1(\bar{\mu}\bar{\pi}^m ,\bar{\mu}'\bar{\pi}^m) \  \leq \  \mathbf{A} e^{-cm}\W^1(\bar{\mu}, \bar{\mu}'),
\end{aligned}
\end{equation}
where $c$ is given in \eqref{eq:c} and 
\begin{align} \label{eq:M_1}
\mathbf{A} \ = \ f'(R_1)^{-1} \ = \ \exp((5/4)(\tilde{R}/T+2)).
\end{align}	
Moreover, $\bar{\mu}_*$ is the unique invariant (nonlinear) measure for $\bar{\pi}$, i.e., $\bar{\mu}_*\bar{\pi} \ = \ \bar{\mu}_*$.
\end{Cor}

\begin{proof}
	The proof is postponed to \Cref{sec:convnhmc_proofs}.
\end{proof}



When $V$ is $K$-strongly convex, i.e., $\mathcal{R}=0$ in \ref{ass_Vconv}, contractivity of the nHMC kernel holds w.r.t.~the standard $\W^1$ distance under improved conditions, which yield the optimal convergence rate for exact (or ideal) HMC; cf.~ \cite[Thm.~1.3]{chen2022optimal}.

\begin{Cor}[Contractivity of nHMC under global strong convexity]\label{cor:strong_cvx}
    Let $\bar{\mu}$ and $\bar{\mu}'$ be two probability measures on $\mathbb{R}^d$ with finite second moment. Suppose \Cref{ass} holds with $\mathcal{R}=0$ in \ref{ass_Vconv}. Let $T\in(0,\infty)$ and $\epsilon\in[0,\infty)$ satisfy
    \begin{align}
        & LT^2  \ \le \  3/20 \;, \quad \text{and} \quad 
      \epsilon \tilde{L} \ \le \  (K/15) (7/6+3/(2KT^2))^{-1/2} \;. 
      \label{eq:condTepsi_strongconv}
    \end{align}
    Then 
    \begin{align}
        &\W^1(\bar{\mu}\bar{\pi}^m, \bar{\mu}'\bar{\pi}^m)
        \ \le \ e^{-cm} \W^1(\bar{\mu},\bar{\mu}') \qquad \text{where} \qquad c=(1/8) K T^2.  \label{eq:contr_strongconvW1}
    \end{align}
        Moreover, there exists a unique invariant (nonlinear) measure $\bar{\mu}_*$ for $\bar{\pi}$, i.e., $\bar{\mu}_*\bar{\pi}=\bar{\mu}_*$.
\end{Cor}

\begin{proof}
	The proof is postponed to \Cref{sec:convnhmc_proofs}.
\end{proof}

By adapting \cite[Theorem 3.1]{BouRabeeSchuh2023}, it is  possible to extend \Cref{thm:contr_nonlinear} to an unadjusted version of nHMC  where the exact nonlinear Hamiltonian flow is time discretized. However, since this numerical flow is not in general implementable,
 we instead consider a particle approximation of the nonlinear process, as in  \cite[\S 4.2]{monmarche2023elementary} and \cite{malrieu2003convergence,cattiaux2008probabilistic,dos2022simulation,kumar2022well}.

\subsection{Propagation of Chaos Phenomenon} \label{sec:propofchaos}

By using a variant of Sznitman's synchronous coupling \cite{Sz91}, we prove propagation of chaos for a single xHMC step.

\begin{proposition}[Propagation of chaos for a single xHMC step] \label{thm:propofchaos_onestep}
Let $\bar{\mu}$ be a probability measure on $\mathbb{R}^d$ with finite second moment. Suppose that \Cref{ass} holds. Assume $(L+2\epsilon\tilde{L}) T^2\le 1$. Let $(\bar{\mathbf{X}}^{\bar{\mu},i})_{i=1}^N$ be $N$ independent copies of one nHMC step given by \eqref{eq:hamdyn_nonl} with initial distribution $\bar{\mu}$. Let  $(\mathbf{X}^i)_{i=1}^N$ be one xHMC step applied to the corresponding mean-field particle system given by \eqref{eq:hamdyn_meanf_num} with initial distribution $\bar{\mu}^{\otimes N}$. Then
\begin{align*}
\mathbb{E}_{x\sim \bar{\mu}^{\otimes N}}\Big[\frac{1}{N}\sum_{i=1}^N|\bar{\mathbf{X}}^{\bar{\mu},i}(x^i)-\mathbf{X}^i(x)|\Big] \ \le \ N^{-1/2} \mathbf{B}
\end{align*}
with $\mathbf{B}=4T^2\epsilon \tilde{L}\mathbf{B}_1^{1/2}$, where $\mathbf{B}_1$ is the uniform second moment bound of nHMC which is given in \Cref{lem:secondmoment} and which depends on the second moment of the initial distribution $\bar{\mu}$, $K$, $\epsilon$, $L$, $\mathcal{R}$, $T$ and $\mathbf{W}_0$.
\end{proposition}

\begin{proof}
The proof is postponed to \Cref{sec:propofchaos_proofs}.
\end{proof}

In turn, by means of a particle-wise coupling \cite{Eb2016A,BouRabeeSchuh2023} for xHMC and
\Cref{thm:propofchaos_onestep}, we prove a uniform-in-steps propagation of chaos bound for xHMC. 

\begin{theorem} [Uniform-in-steps propagation of chaos of xHMC]\label{thm:propofchaos_uniform}
Let $\mu$ and $\bar{\mu}$ be two probability measures on $\mathbb{R}^d$ with finite second moment.
Suppose that \Cref{ass} holds, and $T\in(0,\infty)$, $\epsilon\in[0,\infty)$ satisfy \eqref{eq:cond_T} and \eqref{eq:cond_eps}. Let $\rho_N$ be the distance given in \eqref{eq:rho^N}. Then
\begin{equation}\label{eq:unif_propofchaos_exact}
\begin{aligned} 
& \W_{\rho_N}^1(\mu^{\otimes N}\pi^m , ( \bar{\mu}\bar{\pi}^m)^{\otimes N}) \ \leq \ e^{-cm}\W^1_{\rho}(\mu, \bar{\mu})+c^{-1}\mathbf{B}N^{-1/2},
\\ & \W_{\bar{\ell}_1^N}^1(\mu^{\otimes N}\pi^m , ( \bar{\mu}\bar{\pi}^m)^{\otimes N}) \ \leq \ \mathbf{A} e^{-cm}\W_{1}(\mu, \bar{\mu})+\mathbf{A}c^{-1}\mathbf{B}N^{-1/2},
\end{aligned}
\end{equation}
where  c is given in \eqref{eq:c}, $\mathbf{A}$ is given in \eqref{eq:M_1} and $\mathbf{B}$ is given in \Cref{thm:propofchaos_onestep}. Moreover, 
\begin{equation} \label{eq:unif_propofchaos_exact2}
\begin{aligned} 
& \W_{\rho_N}^1(\bar{\mu}_*^{\otimes N} ,  \mu^{\otimes N}\pi^m) \ \leq \ e^{-cm}\W^1_{\rho}(\bar{\mu}_*, {\mu})+c^{-1}\mathbf{B}N^{-1/2},
\\ & \W_{\bar{\ell}_1^N}^1(\bar{\mu}_*^{\otimes N},\mu^{\otimes N}\pi^m ) \ \leq \  \mathbf{A} e^{-cm}\W_{1}(\bar{\mu}_*, {\mu})+\mathbf{A}c^{-1}\mathbf{B}N^{-1/2}.
\end{aligned}
\end{equation}
\end{theorem}

\begin{proof}
The proof is postponed to \Cref{sec:propofchaos_proofs}.
\end{proof}

  Theorem~\ref{thm:propofchaos_uniform} combined with discretization error bounds provided in \Cref{sec:strong_accuracy} yields a complexity upper bound for uHMC to sample  $\bar{\mu}_*$. 

\begin{theorem} [Complexity of uHMC]\label{thm:propofchaos_uniform_unad}
Let $\mu$ and $\bar{\mu}$ be two probability measures on $\mathbb{R}^d$ with finite second moment.
Suppose \Cref{ass}  holds. Let $T\in(0,\infty)$ satisfy
\begin{align}
	&LT^2 \ \le \ \frac{5}{3}\min\Big(\frac{1}{9},\frac{3}{36^2\cdot 5 L\hat{R}^2}\Big). \label{eq:cond_Th1}
\end{align}
Suppose $\epsilon\in[0,\infty)$ satisfies \eqref{eq:Cepsi}. 
Let $\rho_N$ be the distance defined in \eqref{eq:rho^N}.
Let $h \ge 0$ satisfy $T/h \in \mathbb{Z}$ if $h>0$.  Then
\begin{equation}\label{eq:unif_propofchaos_unadj}
\begin{aligned} 
& \W_{\rho_N}^1(\mu^{\otimes N}{\pi_h}^m , ( \bar{\mu}\bar{\pi}^m)^{\otimes N}) \ \leq \ e^{-cm}\W^1_{\rho}(\mu, \bar{\mu})+c^{-1}\mathbf{B}N^{-1/2}+h^{3/2} \mathbf{C},
\\ & \W_{\bar{\ell}_1^N}^1(\mu^{\otimes N}{\pi_h}^m , ( \bar{\mu}\bar{\pi}^m)^{\otimes N}) \ \leq \ \mathbf{A} e^{-cm}\W^1(\mu, \bar{\mu})+\mathbf{A}c^{-1}\mathbf{B}N^{-1/2}+h^{3/2} \mathbf{C},
\end{aligned}
\end{equation}
where $c$ and $\mathbf{A}$ are given in \eqref{eq:c_meanf} and \eqref{eq:M_1}; $\mathbf{B}$ is given in \Cref{thm:propofchaos_onestep}; and $\mathbf{C} = c^{-1} L^{3/4}\mathbf{B}_3(T\sqrt{d}+\sqrt{\mathbf{B}_2(\mu)}+(\epsilon/K)\mathbf{W}_0)$ where $\mathbf{B}_2(\mu)$ is the uniform second moment bound of xHMC for the mean-field particle model with initial distribution $\mu$ and where $\mathbf{B}_3$ is some numerical constant.
 Moreover, 
\begin{equation} \label{eq:unif_propofchaos_exact2_unadj}
\begin{aligned} 
& \W_{\rho_N}^1(\bar{\mu}_*^{\otimes N} , \mu^{\otimes N}{\pi_h}^m ) \ \leq \  e^{-cm}\W^1_{\rho}(\bar{\mu}_*, {\mu})+c^{-1}\mathbf{B}N^{-1/2}+h^{3/2} \mathbf{C},
\\ & \W_{\bar{\ell}_1^N}^1(\bar{\mu}_*^{\otimes N}, \mu^{\otimes N}{\pi_h}^m ) \ \leq \  \mathbf{A} e^{-cm}\W^1(\bar{\mu}_*, {\mu})+\mathbf{A}c^{-1}\mathbf{B}N^{-1/2}+h^{3/2} \mathbf{C}.
\end{aligned}
\end{equation}
\end{theorem}

\begin{proof}
The proof is postponed to \Cref{sec:propofchaos_proofs}.
\end{proof}

\Cref{thm:propofchaos_uniform_unad} is reminiscent of complexity upper bounds for preconditioned HMC in infinite dimension, in the sense that the perturbed Gaussian measure is analogous to the nonlinear measure \cite{BePiSaSt11,BoEb21,pidstrigach2022convergence}. 

\begin{remark}[Complexity for Strongly Convex Confinement Potentials]
In the specific case of globally strongly convex $V$, uniform-in-steps propagation of chaos of xHMC holds with the same $c$ and under the same assumptions on $T$ and $\epsilon$ as in \Cref{cor:strong_cvx} for nHMC.  Choosing $T \propto \sqrt{L}$ in order to saturate the condition on $T$ in \Cref{cor:strong_cvx}, the number of steps required for  nHMC to attain $\varepsilon$ accuracy in $\W^1$-distance is of order $L/K \log(\Delta(0)/\varepsilon)$ where $\Delta(0) = \W^1(\bar{\mu}_*, {\mu})$. 
To ensure $\varepsilon$ accuracy in $\W^1_{\bar{\ell}_1^N}$-distance of uHMC w.r.t.~$\bar{\mu}_*^{\otimes N}$, Theorem~\ref{thm:propofchaos_uniform_unad} indicates that it is sufficient to choose $h \propto (\mathbf{C}/\varepsilon)^{-2/3}$, $N \propto (\mathbf{B}/(c\varepsilon))^2$ and $m \propto \log(\mathbf{A}\Delta(0)/\varepsilon)/c$. 
Therefore, a sufficient number of gradient evaluations to obtain a $\W^1_{\bar{\ell}_1^N}$-distance smaller than  $\varepsilon$ satisfies
\begin{align*}
    N \times m \times T/h & \ \propto \   \frac{(KT^2)^2\left(\mathbf{B}_4(\mu)+\frac{d}{K}+\frac{\epsilon^2\mathbf{W}_0^2}{K^2}\right)}{(KT^2)^2\varepsilon^2} \times \frac{TL^{1/2}(\mathbf{B}_4(\mu)^{1/3}+(\frac{d}{K})^{1/3}+(\frac{\epsilon\mathbf{W}_0}{K})^{2/3})}{(KT^2\varepsilon)^{2/3}} \times \frac{\log(\Delta(0)/\varepsilon)}{KT^2}
    \\ & \ \propto \  \frac{(L/K)^{5/3}}{\varepsilon^{8/3}}\times \Big(\mathbf{B}_4(\mu)^{4/3} + \Big(\frac{d}{K}\Big)^{4/3} + \Big(\frac{\epsilon\mathbf{W}_0}{K}\Big)^{8/3} \Big)  \times \log\left( \frac{\Delta(0)}{\varepsilon} \right),
\end{align*}
where $\mathbf{B}_4(\mu)=\int_{\mathbb{R}^d}|x|^2\mu(\rmd x)$.  This computational cost reflects the expense of resolving the asymptotic bias due to the particle approximation and motivates alternative approximations to nHMC \cite{GOBET201871,gomes2020mean}; however, these alternatives may require higher regularity conditions than assumed in this paper; cf.~\Cref{ass}. 
\end{remark}

\subsection{Simple Numerical Verification}

This part gives a simple numerical illustration of the theoretical results. Consider the nonlinear target measure $\bar{\mu}_*$ in \eqref{nonlinear_target} with $V(x) = x^2 (1-\epsilon)/2$ and   $W(x,y) = \epsilon ( (x-y)^2 - 1)/2$ for $x,y\in\mathbb{R}$ and $\epsilon >0$.
In this example, the nonlinear probability measure $	\bar{\mu}_*$ in \eqref{nonlinear_target} reduces to  $\mathcal{N}(0,1)$ since
\begin{align*}
& V(x)+\int_\mathbb{R}W(x,y)\varphi_{0,1}(y)\rmd y =  \frac{x^2}{2}+\frac{1}{2}\int_\mathbb{R}\epsilon(-2xy+y^2-1)\varphi_{0,1}(y)\rmd y =\frac{x^2}{2},
\end{align*}
where $\varphi_{0,1}$ is once again the density of the standard normal distribution.   The corresponding Hamiltonian dynamics of the mean-field $N$-particle system is given by \begin{equation*}
    \frac{d}{dt} q_t^i = p_t^i \;, \quad \frac{d}{dt} p_t^i = - q_t^i + \frac{\epsilon}{N} \sum_{j=1}^N q_t^j  \;.
\end{equation*}

By transforming to internal variables, this Hamiltonian system decouples into a collection of $N$ harmonic oscillators.  Explicitly, this transformation is given by
\[
Q^1 = \frac{1}{N} (q^1 + \cdots + q^N) \;, \quad
Q^{i+1} = q^{i+1} - q^{i} \;,
\] and the corresponding natural frequencies are defined by \[
\omega^1 = \sqrt{1-\epsilon} \;, \quad \omega^{i+1} = 1 \;, \quad \text{for } i \in \{ 1, \dots, N-1 \} \;. \]  Moreover, this transformation to internal variables can be inverted in $O(N)$ calculations: (i) set $\tilde q^1 = 0$ and recursively solve $\tilde q^i - \tilde q^{i-1} = Q^{i}$ for $i \in \{ 2, \dots, N \}$;  then, (ii) set $q^i = \tilde q^i + Q^1 - (1/N) \sum_{j=1}^N \tilde q^j$. Using this transformation to internal variables,  the computational cost per gradient evaluation of the $N$-particle system  is  linear in $N$.

We applied this fast implementation to estimate the asymptotic bias of uHMC with time integrator randomization for interaction strength $\epsilon=1/4$.  In particular, for accuracy $\varepsilon = 2^{-k}$ where $k \in \{1, 2, 3, 4, 5 \}$, we ran uHMC for $10^8$ steps with duration parameter $T=1$,  time step size  $h=\varepsilon^{2/3}$, and number of particles $N = \varepsilon^{-2}$.  With this choice of parameters, we expect that the asymptotic bias decays at least linearly with $\varepsilon$; see, in particular, the upper bound in  \eqref{eq:unif_propofchaos_exact2_unadj} as $m \nearrow \infty$.
An estimate of the asymptotic bias of uHMC is shown in  \Cref{fig1} as a function of $k$.   This estimate is obtained by computing  the relative error between a kernel density estimate of the marginal in the first component of uHMC and a standard normal density.  The dashed curve is parallel to the graph of $2^{-k} (=\varepsilon)$ versus $k$.  Since the asymptotic bias  seems to decrease linearly with accuracy $\varepsilon$, this figure is consistent with \eqref{eq:unif_propofchaos_exact2_unadj} with the hyperparameter choices $h =\varepsilon^{2/3}$ and  $N = \varepsilon^{-2}$.  We also tested a slightly larger time step size of $h= \varepsilon^{1/2}$, and obtained a similar finding.  This result suggests that the asymptotic bias of uHMC with randomized time integration is better than the strong accuracy of the underlying randomized time integrator, under perhaps higher regularity assumptions than assumed in this paper.  This  is not surprising since the orders of strong/weak accuracy do not always coincide; for a refined case study of this phenomenon, see \cite{alfonsi2014pathwise}.

\begin{figure}[ht]
\begin{center}
\includegraphics[scale=0.4]{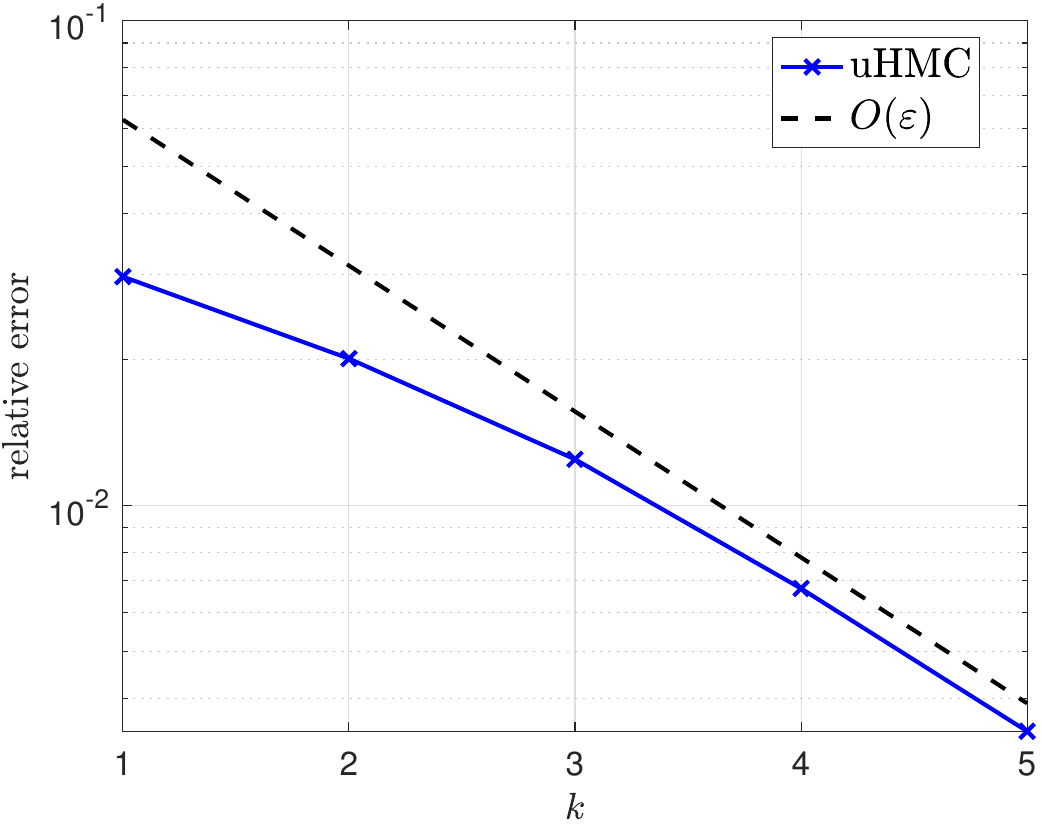}
\includegraphics[scale=0.4]{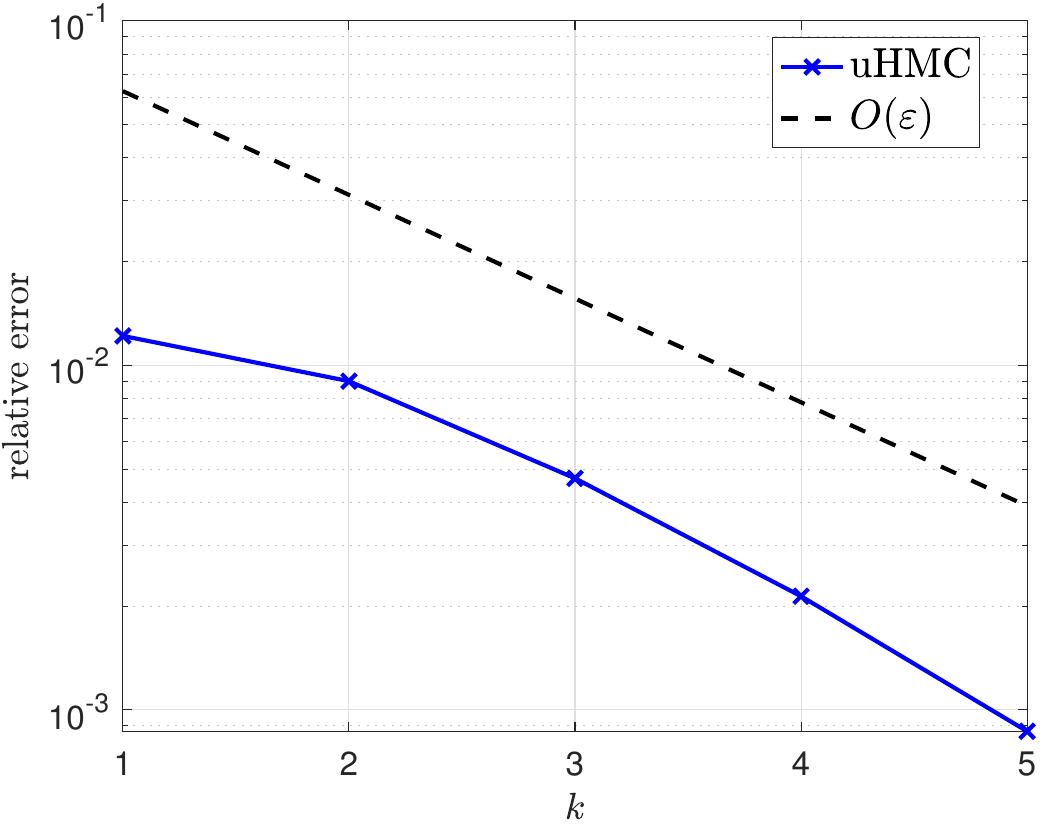}
\end{center}
\caption{\footnotesize{Asymptotic bias versus $k$ where $\varepsilon=2^{-k}$,  $N=\varepsilon^{-2}$, $h=\varepsilon^{1/2}$ (left image), and $h=\varepsilon^{2/3}$ (right image). }}\label{fig1}
\end{figure}


\section{Proofs of Main Results} \label{sec:proof}

\subsection{Proofs Related to Convergence of nHMC}

\label{sec:convnhmc_proofs}

\begin{proof}[Proof of \Cref{thm:contr_nonlinear}]
The proof closely tracks the proof of \cite[Theorem 3.1]{BouRabeeSchuh2023}. 
 We write $R=|\bar{\mathbf{X}}^{\bar{\mu}}(\bar{x},\bar{x}')-\bar{\mathbf{X}}'^{\bar{\mu}'}(\bar{x},\bar{x}')|$ and $r=|\bar{x}-\bar{x}'|$.
First we consider the case, when the synchronous coupling is applied, i.e., $r=|\bar{x}-\bar{x}'|\ge \tilde{R}$. By concavity of $f$, by \eqref{eq:R1}, by \Cref{lem:convexarea} and since $\sqrt{1-a}\leq 1-(a/2)$ for $a\in[0,1)$,
\begin{align}
\mathbb{E}[f(R)-f(r)]\leq f'(r)\mathbb{E}[R-r]\leq f'(r)(-1/8) K T^2r+f'(r)\epsilon\tilde{L}T^2\sqrt{\frac{7}{6}+\frac{3}{2KT^2}}\max_{s\leq T}\mathbb{E}[|\bar{z}_s|]. \label{eq:contr_largedist}
\end{align}
where  $\mathbb{E}[|\bar{z}_s|]=\mathbb{E}_{x\sim \bar{\mu}_0,x'\sim \bar{\mu}_0', \xi, \eta\sim\mathcal{N}(0,I_d)}[|\bar{q}_s(x,\xi, \bar{\mu}_0)-\bar{q}_s(x',\eta, \bar{\mu}_0')|]$.

If $r=|\bar{x}-\bar{x}'|<\tilde{R}$, then the initial velocities satisfy $\xi-\eta=\bar{w}=-\gamma\bar{z}$ with maximal possible probability and otherwise a reflection coupling is applied. We split the expectation $\mathbb{E}[f(R)-f(r)]$ according to these disjoint probabilities, i.e.,
\begin{align*}
\mathbb{E}[f(R)-f(r)]&=\mathbb{E}[f(R)-f(r);\{\bar{w}=-\gamma \bar{z}\}]+\mathbb{E}[f(R\wedge R_1)-f(r);\{\bar{w}\neq -\gamma \bar{z}\}]\\ & +\mathbb{E}[f(R)-f(R\wedge R_1); \{\bar{w}\neq-\gamma \bar{z}\}]=\rn{1}+\rn{2}+\rn{3}.
\end{align*}
As in \cite[Equation (6.9)]{BouRabeeSchuh2023}, it holds by \eqref{eq:gamma}
\begin{align} \label{eq:TV_bound}
	\mathbb{P}[\bar{w}\neq -\gamma \bar{z}]=\|\mathcal{N}(0,I_d)-\mathcal{N}(\gamma\bar{z},I_d)\|_{TV}\le \frac{2}{\sqrt{2\pi}}\int_0^{\gamma|\bar{z}|/2}\exp\left(-\frac{x^2}{2}\right)\rmd x \le \frac{\gamma|\bar{z}|}{\sqrt{2\pi}}<1/10.
\end{align}
To bound $\rn{1}$, note that on the set $\{\bar{w}=-\gamma\bar{z}\}$ by \eqref{eq:lem_nonlin5}, \eqref{eq:cond_T} and \eqref{eq:gamma},
\begin{align*}
R &\leq (1-\gamma T)|\bar{z}|+(L+\epsilon\tilde{L})T^2|\bar{z}|+\epsilon\tilde{L}T\max_{s\leq T}\mathbb{E}[|\bar{z}_s|]
\\ & \leq (1-\gamma T)|\bar{z}|+(\gamma T/4)|\bar{z}|+\epsilon\tilde{L}T\max_{s\leq T}\mathbb{E}[|\bar{z}_s|]=(1-(3/4)\gamma T)r+\epsilon\tilde{L}T\max_{s\leq T}\mathbb{E}[|\bar{z}_s|].
\end{align*}
Hence, by concavity of $f$ and \eqref{eq:TV_bound},
\begin{align*}
\rn{1}&\leq -f'(r)(3/4)\gamma Tr \mathbb{P}[\bar{w}= -\gamma \bar{z}]+f'(r)\epsilon\tilde{L}T\max_{s\leq T}\mathbb{E}[|\bar{z}_s|]
\\ & \leq -f'(r)(27/40)\gamma Tr +f'(r)\epsilon\tilde{L}T\max_{s\leq T}\mathbb{E}[|\bar{z}_s|].
\end{align*}
We bound $\rn{2}$ as in \cite[Equation (6.11)]{BouRabeeSchuh2023}, i.e., for $r,s\le R_1$, $f(s)-f(r)=\int_r^s e^{-t/T}\rmd t =T(e^{-r/T}-e^{-s/T})\leq T f'(r)$ and hence by \eqref{eq:TV_bound},
\begin{align}
\rn{2}\leq Tf'(r) \mathbb{P}[\bar{w}\neq -\gamma\bar{z}] \le (2/5)\gamma T r f'(r). \label{eq:boundII}
\end{align}
For $\rn{3}$, note that if $\bar{w}\neq -\gamma\bar{z}$, then $\bar{w}=2(\bar{e}\cdot\xi)\bar{e}$ with $\bar{e}=\bar{z}/|\bar{z}|$. Hence, $|\bar{z}+T\bar{w}|=|r+2T\bar{e}\cdot\xi|$, and by \eqref{eq:lem_nonlin6} and \eqref{eq:cond_T},
\begin{align*}
R\leq 5/4 \max(|r+2T\bar{e}\cdot\xi|,r)+\epsilon\tilde{L}T^2\max_{s\leq T}\mathbb{E}[|\bar{z}_s|].
\end{align*}
Since $(5/4)r-R_1\leq (5/4)\tilde{R}-R\leq 0$,
\begin{align*}
\mathbb{E}[(R-R_1)^+;\{\bar{w}\neq -\gamma\bar{z}\}] & \le \mathbb{E}\left[\Big(\frac{5}{4}\max(|r+2T\bar{e}\cdot\xi|,r)+\epsilon\tilde{L}T\max_{s\leq T}\mathbb{E}[|\bar{z}_s|]-R_1\Big)^+;\{\bar{w}\neq -\gamma\bar{z}\}\right]
\\ & \le  \mathbb{E}\left[\Big(\frac{5}{4}|r+2T\bar{e}\cdot\xi|-R_1\Big)^+;\{\bar{w}\neq -\gamma\bar{z}\}\right] +\epsilon\tilde{L}T^2\max_{s\leq T}\mathbb{E}\left[|\bar{z}_s|\right].
\end{align*}
For the first term, it holds 
as in \cite[Equation (6.14)]{BouRabeeSchuh2023},
\begin{align}
\mathbb{E}&\left[\Big(\frac{5}{4}|r+2T\bar{e}\cdot\xi|-R_1\Big)^+;\{\bar{w}\neq -\gamma\bar{z}\}\right] \nonumber
\\ & =\int_{-\gamma r/2}^\infty ((5/4)|r+2Tu|-R_1)^+\frac{1}{\sqrt{2\pi}}\Big(\exp(-u^2/2)-\exp(-(u+\gamma r)^2/2)\Big) \rmd u \nonumber
\\ & =\int_{-\gamma r/2}^{\gamma r/2}\Big(\frac{5}{4}|r+2Tu|-R_1)^+\frac{1}{\sqrt{2\pi}}\exp(-u^2/2)\rmd u \nonumber
\\ & + \int_{\gamma r/2}^\infty \Big(\Big(\frac{5}{4}|r+2Tu|-R_1\Big)^+-\Big(\frac{5}{4}|r+2T(u-\gamma r)|-R_1\Big)^+\Big)\frac{1}{\sqrt{2\pi}}\exp(-u^2/2) \rmd u \nonumber
\\ & \le \int_{\gamma r/2}^\infty \Big(\frac{5}{4} 2\gamma r T\Big) \frac{1}{\sqrt{2\pi}}\exp(-u^2/2)\rmd u \label{eq:boundIII}
\le \frac{5}{4}\gamma T r,
\end{align}
where \eqref{eq:gamma} and \eqref{eq:R1} is applied in the second last step.
Applying concavity of $f$,
\begin{align*}
\rn{3} &\le f'(R_1)\mathbb{E}[(R-R_1)^+;\{\bar{w}\neq -\gamma\bar{z}\}]
 \le f'(R_1) \Big( \frac{5}{4}\gamma T r +\epsilon\tilde{L}T^2\max_{s\leq T}\mathbb{E}[|\bar{z}_s|] \Big)
\\ & \le f'(r) \Big( \frac{5}{48}\gamma T r +\frac{1}{12}\epsilon\tilde{L}T^2\max_{s\leq T}\mathbb{E}[|\bar{z}_s|] \Big), 
\end{align*}
where the last step holds since $\exp(T^{-1}(R_1-\tilde{R}))\ge 12$ by \eqref{eq:R1}.
Combining the bounds for $\rn{1}$, $\rn{2}$ and $\rn{3}$ yields
\begin{align*}
\mathbb{E}[f(R)-f(r)]\le -f'(r)\frac{41}{240}\gamma T r+f'(r)\frac{13}{12}\epsilon\tilde{L}T^2\max_{s\leq T}\mathbb{E}[|\bar{z}_s|].
\end{align*}
Combining the two cases $r\ge \tilde{R}$ and $r < \tilde{R}$ yields
\begin{align*}
\mathbb{E}[f(R)-f(r)]\le -\min\Big(\frac{41}{240}\gamma T,\frac{1}{8} K T^2\Big)f'(r) r+\max\Big(\frac{13}{12}\epsilon\tilde{L}T^2,\epsilon\tilde{L}T^2\sqrt{\frac{7}{6}+\frac{3}{2KT^2}}\Big)f'(r)\max_{s\leq T}\mathbb{E}[|\bar{z}_s|],
\end{align*}
To bound the last term note that by \eqref{eq:gamma}
\begin{align*}
	\mathbb{E}[|\bar{w}| \1_{\{r < \tilde{R}	\} \cap \{\bar{w}\neq -\gamma \bar{z}\}}] &=\1_{\{r < \tilde{R}	\} } \int_{-\gamma r/2}^\infty \frac{2|u|}{\sqrt{2\pi}} (\exp(-u^2/2)-\exp(-(u+\gamma r)^2/2))\rmd u
	\\ & =\1_{\{r < \tilde{R}	\}} \Big(\int_{-\gamma r/2}^\infty \frac{2|u|}{\sqrt{2\pi}} \exp(-u^2/2)\rmd u- \int_{\gamma r/2}^\infty \frac{2|u-\gamma r|}{\sqrt{2\pi}} \exp(-u^2/2)\rmd u \Big)
	\\ &\le \1_{\{r < \tilde{R}	\} } \Big(\int_{-\gamma r/2}^{\gamma r/2} \frac{2|u|}{\sqrt{2\pi}} \exp(-u^2/2)\rmd u+\int_{\gamma r/2}^\infty \frac{2\gamma r}{\sqrt{2\pi}} \exp(-u^2/2)\rmd u \Big)
	\\ & \le \1_{\{r < \tilde{R}	\}} ((\gamma r)^2+ \gamma r)\le (5/4)\gamma r.
\end{align*}
Therefore, by \eqref{eq:cond_T} and \eqref{eq:lem_nonlin9} it holds,
\begin{align} \label{eq:expec_z_s}
\max_{s\leq T}\mathbb{E}[|\bar{z}_s|] & \le \frac{5}{4} \mathbb{E}[|\bar{z}|+|T\bar{w}|\1_{\{r < \tilde{R} \} \cap \{\bar{w}\neq -\gamma \bar{z}\}}] = \frac{5}{4} \mathbb{E}[|\bar{z}|]+\frac{5}{4}T\mathbb{E}[|\bar{w}|\1_{\{r < \tilde{R} \} \cap \{\bar{w}\neq -\gamma \bar{z}\}}] \le \frac{45}{16} \mathbb{E}[r],
\end{align}
where the last step holds by \eqref{eq:gamma} and the previous estimate.
By \eqref{eq:gamma} and \eqref{eq:cond_T} the minimum is attained at $(1/8)K T^2$ and the maximum is attained at $\epsilon\tilde{L}T^2\sqrt{\frac{7}{6}+\frac{3}{2KT^2}}$.
Further,
\begin{align} \label{eq:estimate_f}
\inf_r \frac{r f'(r)}{f(r)}\ge \frac{5}{4}\Big(\frac{\tilde{R}}{T}+2\Big)\exp\Big(-\frac{5\tilde{R}}{4T}\Big)\exp\Big(-\frac{5}{2}\Big) \quad \text{and} \quad r\le \exp\Big(\frac{5\tilde{R}}{4 T}\Big)\exp\Big(\frac{5}{2}\Big)f(r).
\end{align}
Hence,
\begin{align*}
\mathbb{E}[f(R)-f(r)]\leq -\frac{1}{78}K T^2 \exp\Big(-\frac{5\tilde{R}}{4T}\Big)\mathbb{E}[r],
\end{align*}
where we used \eqref{eq:cond_eps}. 
\end{proof}

\begin{proof}[Proof of \Cref{cor:invmeas}]
	By \eqref{eq:contraction} and \eqref{eq:rho}, it holds
	\begin{align*}
		\W^1_{\rho}(\bar{\mu}\bar{\pi},\bar{\mu}'\bar{\pi})\le (1-c)\mathbb{E}_{x\sim \bar{\mu},x'\sim \bar{\mu}'}[f(|x-x'|)].
	\end{align*}
	Taking the infimum over all couplings $\omega\in\Pi(\bar{\mu},\bar{\mu}')$, we obtain
	\begin{align*}
		\W^1_{\rho}(\bar{\mu}\bar{\pi},\bar{\mu}'\bar{\pi})\le(1-c)\W^1_{\rho}(\bar{\mu},\bar{\mu}').
	\end{align*}	
	Hence, the first bound in \eqref{eq:W_contraction} holds for all $m\in\mathbb{N}$, since $(1-c)^m\leq e^{-cm}$,
	and the second bound holds by \eqref{eq:distance_equiv} with $\mathbf{A}$ given in \eqref{eq:M_1}.
	The existence of a unique invariant probability measure $\bar{\mu}_*$ holds by Banach fixed-point theorem (cf.\cite[Theorem 3.9]{Eb20}) for the operator $\mathcal{T}:(\mathcal{P}(\mathbb{R}^d), \W^1)\to(\mathcal{P}(\mathbb{R}^d), \W^1), \bar{\mu}\to\bar{\mu}\bar{\pi}$ as for sufficiently large $m$, $\mathcal{T}^m$ is a contraction mapping and hence $\mathcal{T}$ has a unique fixed point which is $\bar{\mu}_*$ by \Cref{prop:inv_meas}.
\end{proof}

\begin{proof}[Proof of \Cref{cor:strong_cvx}] The result is an immediate consequence of \Cref{lem:convexarea}. We note that \eqref{eq:condTepsi_strongconv} implies \eqref{eq:cond_t}.
For $\mathcal{R}=0$, we consider the synchronous coupling. Writing $R=|\bar{\mathbf{X}}^{\bar{\mu}}(\bar{x},\bar{x}')-\bar{\mathbf{X}}'^{\bar{\mu}'}(\bar{x},\bar{x}')|$ and $r=|\bar{x}-\bar{x}'|$, we obtain by \eqref{eq:convexarea} with $\hat{C}=0$
\begin{align*}
    \mathbb{E}[R] &\le \mathbb{E}[\sqrt{1-(5/12)KT^2}r]+ \epsilon\tilde{L} T \sqrt{7/6+3/(2KT^2)} \max_{s\le T}\mathbb{E}[|\bar{z}_s|]. 
\end{align*}
Using $\sqrt{1-a}\le 1-(a/2)$ for $a\in[0,1)$ and \eqref{eq:lem_nonlin9} and \eqref{eq:condTepsi_strongconv}
it holds
\begin{align*}
    \mathbb{E}[R] \le  \mathbb{E}[(1-(5/24)KT^2)r]+ \epsilon\tilde{L} T \sqrt{7/6+3/(2KT^2)} (5/4)\mathbb{E}[r]\le \mathbb{E}[(1-(1/8)KT^2)r].
\end{align*}
Then, by following the proof of \Cref{cor:invmeas}, \eqref{eq:contr_strongconvW1} holds.
Analogously to the proof of \Cref{cor:invmeas}, the existence of a unique invariant probability measure $\bar{\mu}_*$ 
holds by the Banach fixed-point theorem.
    \end{proof}

\subsection{Proofs Related to Propagation of Chaos Phenomenon}

\label{sec:propofchaos_proofs}

\begin{proof}[Proof of \Cref{thm:propofchaos_onestep}]
Using a synchronous coupling of the transition steps of xHMC
$(\mathbf{X}^i)_{i=1}^N$ for the particle system on $\mathbb{R}^{N d}$
and $N$ copies of nHMC  $(\bar{\mathbf{X}}^{\bar{\mu},i})_{i=1}^N$, we prove propagation of chaos phenomenon for a single xHMC step. By \eqref{eq:hamdyn_nonl} and \eqref{eq:hamdyn_meanf_num}, the difference process $(\{z_t^i,w_t^i\}_{i=1}^N)_{t\ge 0}=(\{x_t^i-\bar{x}_t^i,v_t^i-\bar{v}_t^i\}_{i=1}^N)_{t\ge 0}$ of the Hamiltonian dynamics for the particle system and $N$ independent copies of the distribution-dependent Hamiltonian dynamics satisfies
\begin{align} \label{eq:diff_proc}
\begin{cases}
 \frac{\rmd}{\rmd t} z_t^i \ = \ w_t^i
\\ \frac{\rmd}{\rmd t} w_t^i  
 \ = \ -(\nabla V(x_t^i)-\nabla V(\bar{x}_t^i))-\epsilon(N^{-1}\sum_{j=1}^N \nabla_1 W(x_t^i,x_t^j)-\mathbb{E}_{u\sim\bar{\mu}_t}[\nabla_1 W(\bar{x}_t^i,u)]),
\end{cases} 
\end{align} where $\bar{\mu}_t^x=\Law(\bar{x}_t^i)$, $w_0^i=0$, and $x_0^i$ and $\bar{x}_0^i$ are i.i.d. random variables with law $\mu$ for all $i \in \{ 1, \ldots ,N \}$. For brevity, let $\mathbb{E}[\nabla_1 W(\bar{x}_t^i,\bar{x}_t)]=\mathbb{E}_{u\sim\bar{\mu}_t}[\nabla_1 W(\bar{x}_t^i,u)]$.
Then,
\begin{align*}
\sum_{i=1}^N z_s^i & =\int_0^s\int_0^r\sum_{i=1}^N\Big(-\nabla V(x_u^i)+\nabla V(\bar{x}_u^i)
 -\epsilon \Big(N^{-1}\sum_{j=1}^N \nabla_1 W(x_u^i,x_u^j)-\mathbb{E}[\nabla_1 W(\bar{x}_u^i,\bar{x}_u)] \Big)\Big)\rmd u\rmd r.
\end{align*}
Then by \Cref{ass}~\ref{ass_Vlip},
\begin{align*}
\sum_{i=1}^N |z_s^i| &\leq \int_0^s\int_0^r \sum_{i=1}^N \Big( L|z_u^i|+\epsilon\Big|N^{-1}\sum_{j=1}^N \nabla_1 W(\bar{x}_u^i,\bar{x}_u^j)-\mathbb{E}[\nabla_1 W(\bar{x}_u^i,\bar{x}_u)]\Big|\Big)\rmd u\rmd r
\\ & \quad + \int_0^s\int_0^r \sum_{i=1}^N\epsilon\Big|N^{-1}\sum_{j=1}^N \nabla_1 W(\bar{x}_u^i,\bar{x}_u^j)-\nabla_1 W(x_u^i,x_u^j) \Big|\rmd u\rmd r.
\end{align*}
Taking expectation,
\begin{align*}
\max_{s\leq T} \mathbb{E} \Big[\sum_{i=1}^N |z_s^i|\Big]&\leq \frac{(L+2\epsilon\tilde{L}) T^2}{2}\max_{s\le T} \mathbb{E}\Big[\sum_{i=1}^N|z_s^i|\Big]
\\ & + \frac{\epsilon T^2}{2}\max_{s\le T}\mathbb{E}\Big[ \sum_{i=1}^N \Big|N^{-1}\sum_{j=1}^N \nabla_1 W(\bar{x}_s^i,\bar{x}_s^j)-\mathbb{E}[\nabla_1 W(\bar{x}_s^i,\bar{x}_s)]\Big|\Big].
\end{align*}
Since $(L+2\epsilon\tilde{L})T^2\leq 1$,
\begin{align} \label{eq:propofchaos_est1}
\max_{s\leq T} \mathbb{E} \Big[\sum_{i=1}^N |z_s^i|\Big]\leq \epsilon T^2\max_{s\le T}\mathbb{E}\Big[\sum_{i=1}^N \Big|N^{-1}\sum_{j=1}^N \nabla_1 W(\bar{x}_s^i,\bar{x}_s^j)-\mathbb{E}[\nabla_1 W(\bar{x}_s^i,\bar{x}_s)]\Big|\Big].
\end{align}

To bound $\mathbb{E}[  |N^{-1}\sum_{j=1}^N \nabla_1 W(\bar{x}_s^i,\bar{x}_s^j)-\mathbb{E}[\nabla_1 W(\bar{x}_s^i,\bar{x}_s)]|]$, note that given $\bar{x}_t^i$, the random variable $\bar{x}_t^j$ are i.i.d. with law $\bar{\mu}_t$. Hence, it holds
\begin{align*}
\mathbb{E}[\nabla_1 W(\bar{x}_t^i,\bar{x}_t^j)|\bar{x}_t^i]=\int_{\mathbb{R}^d} \nabla_1 W(\bar{x}_t^i, u) \bar{\mu}_t(\rmd u).
\end{align*}
Therefore, it holds
\begin{align*}
& \mathbb{E}\Big[|\int_{\mathbb{R}^d} \nabla_1 W(\bar{x}_t^i, u) \bar{\mu}_t(\rmd u)-\frac{1}{N}\sum_{j=1}^N\nabla_1 W(\bar{x}_t^i,\bar{x}_t^j)|^2\Big|\bar{x}_t^i\Big]=\frac{N-1}{N^2}\Var_{\bar{\mu}_t}(\nabla_1 W(\bar{x}_t^i,\cdot))
\\  & \qquad + \frac{2}{N^2}\sum_{j=1, j\neq i}^N \mathbb{E}\Big[|\int_{\mathbb{R}^d} \nabla_1 W(\bar{x}_t^i, u) \bar{\mu}_t(\rmd u)-\frac{1}{N}\sum_{j=1}^N\nabla_1 W(\bar{x}_t^i,\bar{x}_t^j)|
\\ & \qquad \qquad\qquad\qquad\cdot|\int_{\mathbb{R}^d} \nabla_1 W(\bar{x}_t^i, u) \bar{\mu}_t(\rmd u)-\frac{1}{N}\sum_{j=1}^N\nabla_1 W(\bar{x}_t^i,\bar{x}_t^i)|\Big|\bar{x}_t^i\Big]
\\ & \qquad + \frac{1}{N^2} \mathbb{E}\Big[|\int_{\mathbb{R}^d} \nabla_1 W(\bar{x}_t^i, u) \bar{\mu}_t(\rmd u)-\frac{1}{N}\sum_{j=1}^N\nabla_1 W(\bar{x}_t^i,\bar{x}_t^i)|^2\Big|\bar{x}_t^i\Big].
\end{align*}
By \Cref{ass}\ref{ass_Wlip}, Cauchy-Schwarz inequality and Young's inequality
\begin{align*}
& \mathbb{E}\Big[|\int_{\mathbb{R}^d} \nabla_1 W(\bar{x}_t^i, x) \bar{\mu}_t(\rmd x)-\frac{1}{N}\sum_{j=1}^N\nabla_1 W(\bar{x}_t^i,\bar{x}_t^j)|^2\Big]
\\ & \qquad \leq \frac{4 \tilde{L}^2}{N}\int_{\mathbb{R}^d}  |x|^2 \bar{\mu}_t(\rmd x)+\frac{8 \tilde{L}^2}{N}\int_{\mathbb{R}^d}  |x|^2 \bar{\mu}_t(\rmd x)
+\frac{4 \tilde{L}^2}{N^2}\int_{\mathbb{R}^d}  |x|^2 \bar{\mu}_t(\rmd x).
\end{align*}
Then, by Cauchy-Schwarz inequality
\begin{align*}
\mathbb{E}\left[  |N^{-1}\sum_{j=1}^N \nabla_1 W(x_t^i,x_t^j)-\mathbb{E}[\nabla_1 W(\bar{x}_t^i,\bar{x}_t)]|\right]\leq \frac{4\tilde{L}}{N^{1/2}}\Big(\int_{\mathbb{R}^d}  |x|^2 \bar{\mu}_t(\rmd x)\Big)^{1/2}\leq \frac{4\tilde{L}}{N^{1/2}}\mathbf{B}_1^{1/2},
\end{align*}
where $\mathbf{B}_1$ is the uniform second moment bound. Hence, by \eqref{eq:propofchaos_est1}
\begin{align*}
\max_{s\leq T} \mathbb{E} \Big[N^{-1}\sum_{i=1}^N |z_s^i|\Big]\leq \epsilon T^2 4\tilde{L} N^{-1/2} \mathbf{B}_1^{1/2}
\end{align*} as required. \end{proof}

\begin{proof}[Proof of \Cref{thm:propofchaos_uniform}]
\Cref{thm:contr_meanf} in \Cref{appendix} implies that
\begin{align} \label{eq:contrac_meanfield}
\W_{\rho_N}^1(\mu^{\otimes N}\pi,\bar{\mu}^{\otimes N}\pi)\leq (1-c)\W_{\rho_N}^1(\mu^{\otimes N},\bar{\mu}^{\otimes N})
\end{align}
with $c$ given in \eqref{eq:c}. 
Then, 
\begin{align*}
\W_{\rho_N}^1&((\bar{\mu}\bar{\pi}^m)^{\otimes N},\mu^{\otimes N}\pi^m) 
  \leq \sum_{i=0}^{m-1}  
\W_{\rho_N}^1((\bar{\mu}\bar{\pi}^{m-i})^{\otimes N}\pi^i,(\bar{\mu}\bar{\pi}^{m-i-1})^{\otimes N}\pi^{i+1})
 +\W_{\rho_N}^1(\bar{\mu}^{\otimes N}\pi^m,\mu^{\otimes N}\pi^m)
\\
& \leq \sum_{i=0}^{m-1} (1-c)^{i}
\W_{\rho_N}^1((\bar{\mu}\bar{\pi}^{m-i})^{\otimes N},(\bar{\mu}\bar{\pi}^{m-i-1})^{\otimes N}\pi)+(1-c)^m\W_{\rho_N}^1(\bar{\mu}^{\otimes N},\mu^{\otimes N})
\\ & \leq  \sum_{i=0}^{m-1} (1-c)^{i} \mathbf{B}N^{-1/2}+e^{-cm}\W_{\rho_N}^1(\bar{\mu}^{\otimes N},\mu^{\otimes N})\leq c^{-1}\mathbf{B}N^{-1/2}+e^{-cm}\W_{\rho_N}^1(\bar{\mu}^{\otimes N},\mu^{\otimes N}).
\end{align*}
In the second to last step, we used \Cref{thm:propofchaos_onestep}.
In the last step, we take the limit of the geometric series, since $\mathbf{B}$ does not depend on $i$, but on the second moment of $\bar{\mu}$ by \Cref{lem:secondmoment}. The second part of \eqref{eq:unif_propofchaos_exact} holds using \eqref{eq:distance_equiv}.

To prove \eqref{eq:unif_propofchaos_exact2}, we either set ${\mu}=\bar{\mu}_*$ and use that by \Cref{thm:contr_nonlinear} there exists an unique invariant measure $\bar{\mu}_*$ for the transition step $\bar{\pi}$, i.e., $\bar{\mu}_*=\bar{\mu}_*\bar{\pi}$, or we can prove it directly.
Applying the `triangle inequality trick' \cite[Remark 6.3]{MaStTr10} with \eqref{eq:contrac_meanfield}, $\bar{\mu}_*=\bar{\mu}_*\bar{\pi}$, and ${\mu}_*^N={\mu}_*^N{\pi}$,
\begin{align*}
& \W_{\rho_N}^1(\bar{\mu}_*^{\otimes N},\mu_*^N)=\W_{\rho_N}^1((\bar{\mu}_*\bar{\pi})^{\otimes N},\mu_*^N\pi) \le \W_{\rho_N}^1((\bar{\mu}_*\bar{\pi})^{\otimes N},(\bar{\mu}_*)^{\otimes N}\pi)+ \W_{\rho_N}^1((\bar{\mu}_*)^{\otimes N}\pi,\mu_*^N\pi)
\\ & \qquad \overset{\eqref{eq:contrac_meanfield}}{\le}  \W_{\rho_N}^1((\bar{\mu}_*\bar{\pi})^{\otimes N},(\bar{\mu}_*)^{\otimes N}\pi)+ (1-c)\W_{\rho_N}^1((\bar{\mu}_*)^{\otimes N},\mu_*^N).
\end{align*}
Simplifying this inequality yields,
\begin{align*}
\W_{\rho_N}^1(\bar{\mu}_*^{\otimes N},\mu_*^N)\leq c^{-1} \W_{\rho_N}^1((\bar{\mu}_*\bar{\pi})^{\otimes N},(\bar{\mu}_*)^{\otimes N}\pi).
\end{align*}
By \Cref{thm:propofchaos_onestep} and \eqref{eq:distance_equiv},
\begin{align*}
\W_{\rho_N}^1(\bar{\mu}_*^{\otimes N},\mu_*^N)\leq c^{-1} \mathbf{B}N^{-1/2}.
\end{align*}
One more application of \eqref{eq:contrac_meanfield} then yields
\begin{align*}
& \W_{\rho_N}^1(\mu^{\otimes N}\pi^m,(\bar{\mu}_*)^{\otimes N})  \leq \W_{\rho_N}^1(\mu^{\otimes N}\pi^m,\mu_*^N )+\W_{\rho_N}^1(\mu_*^N,(\bar{\mu}_*)^{\otimes N}) \\
& \qquad \leq  e^{-cm}\W_{\rho_N}^1(\mu^{\otimes N},(\bar{\mu}_*)^{\otimes N}) +c^{-1}\mathbf{B}N^{-1/2}.
\end{align*}
By \eqref{eq:distance_equiv}, we obtain the second bound in \eqref{eq:unif_propofchaos_exact2}. \end{proof}

\begin{proof}[Proof of \Cref{thm:propofchaos_uniform_unad}]
	The proof follows by combining the result of the previous theorem and the strong accuracy result of uHMC developed in \Cref{sec:strong_accuracy}. To obtain \eqref{eq:unif_propofchaos_unadj}, we note that by the triangle inequality and \eqref{eq:unif_propofchaos_exact},
	\begin{align}
		\W_{\rho_N}^1((\bar{\mu}\bar{\pi}^m)^{\otimes N},\mu^{\otimes N}{\pi_h}^m) & \le \W_{\rho_N}^1((\bar{\mu}\bar{\pi}^m)^{\otimes N},\mu^{\otimes N}{\pi}^m)+ \W_{\rho_N}^1(\mu^{\otimes N}{\pi}^m,\mu^{\otimes N}{\pi_h}^m) \nonumber
		\\ & \le e^{-cm}\W^1_{\rho}(\bar{\mu},\mu)+c^{-1}\mathbf{B}N^{-1/2}+ \W_{\rho_N}^1(\mu^{\otimes N}{\pi}^m,\mu^{\otimes N}{\pi_h}^m) \label{eq:proof_prop_unad1}.
	\end{align}
	Hence, it remains to bound the last term, 
	\begin{align}
		\W_{\rho_N}^1(\mu^{\otimes N}{\pi}^m,\mu^{\otimes N}{\pi_h}^m)&\le \sum_{k=1}^m \W_{\rho_N}^1(\mu^{\otimes N}{\pi}^k{\pi_h}^{m-k},\mu^{\otimes N}{\pi}^{k-1}{\pi_h}^{m-k+1})  \nonumber
		\\ & \le  \sum_{k=1}^m \W_{\rho_N}^1(\mu^{\otimes N}{\pi}^k,\mu^{\otimes N}{\pi}^{k-1}{\pi_h})(1-{c})^{m-k}, \label{eq:proof_prop_unad2}
	\end{align}
	where the last step follows by Theorem~\ref{thm:contr_uHMC} with $c$ given by \eqref{eq:c_meanf}.
	We note that similarly as in \Cref{lem:secondmoment} the second moment of the Markov chain generated by the exact Hamiltonian dynamics for the mean-field particle model is uniformly bounded, i.e., there exists $\mathbf{B}_2<\infty$ such that $\sup_{m\in\mathbb{N}}\mathbb{E}[N^{-1}\sum_{i=1}^N|\mathbf{X}_m^i|^2]\le \mathbf{B}_2$ provided the initial distribution has finite second moment; see  \Cref{lem:secondmoment_meanf} in \Cref{appendix}. Here, we denote by $\mathbf{B}_2(\mu)$ the uniform-in-steps moment bound corresponding to the initial distribution $\mu^{\otimes N}$.
	By Theorem~\ref{thm:strong_accuracy}, for all $k\in\mathbb{N}$,
	\begin{align*}
		\W_{\rho_N}^1(\mu^{\otimes N}{\pi}^k,\mu^{\otimes N}{\pi}^{k-1}{\pi_h})&\le h^{3/2}\mathbf{B}_3 L^{3/4}\left(T\sqrt{d}+ \left(\mathbb{E}_{\{x^i\}_{i=1}^N\sim (\mu^{\otimes N}\pi^{k-1}) }\Big[\frac{1}{N}\sum_{i=1}^N|x^i|^2\Big] \right)^{1/2}+ (\epsilon/K)\mathbf{W}_0\right)
\\ & \le h^{3/2}L^{3/4} \mathbf{B}_3(T\sqrt{d}+\sqrt{\mathbf{B}_2(\mu)}+(\epsilon/K) \mathbf{W}_0),
	\end{align*}
	where $\mathbf{B}_3$ is some numerical constant. 
	Inserting this bound in \eqref{eq:proof_prop_unad2} yields
	\begin{align*}
	& \W_{\rho_N}^1(\mu^{\otimes N}{\pi}^m,\mu^{\otimes N}{\pi_h}^m) \le \sum_{k=1}^m (1-{c})^{m-k} h^{3/2}L^{3/4} \mathbf{B}_3(T\sqrt{d}+\sqrt{\mathbf{B}_2(\mu)}+(\epsilon/K)\mathbf{W}_0) 
 \\ & \qquad  \le \frac{h^{3/2}L^{3/4}}{c}  \mathbf{B}_3(T\sqrt{d}+\sqrt{\mathbf{B}_2(\mu)}+(\epsilon/K)\mathbf{W}_0),
	\end{align*}
	and combined with \eqref{eq:proof_prop_unad1} the first bound of \eqref{eq:unif_propofchaos_unadj} is obtained with $\mathbf{C}={c}^{-1} L^{3/4} \mathbf{B}_3(T\sqrt{d}+\sqrt{\mathbf{B}_2(\mu)}+(\epsilon/K)\mathbf{W}_0)$. The second bound holds by \eqref{eq:distance_equiv} with $\mathbf{A}$ given in \eqref{eq:M_1}. Equation \eqref{eq:unif_propofchaos_exact2_unadj} holds for $\bar{\mu}=\bar{\mu}_*$.
\end{proof}

\section{Estimates for the Nonlinear Hamiltonian flow}\label{sec:estimates} 

In this section,  bounds are developed for the nonlinear Hamiltonian flow  that are crucial in the proofs of the main results.

\subsection{Deviation from Free Nonlinear Flow}\label{sec:deviation_freedyn}
   Throughout this subsection, we fix  $t\in(0,\infty)$ and  $\epsilon\in[0,\infty)$ such that \begin{align}\label{eq:t_cond}
(L+2\epsilon\tilde{L})t^2\leq 1.
\end{align}
Denote by $(\bar{x}_s,\bar{v}_s)_{s\geq 0}$ the solution to \eqref{eq:hamdyn_nonl} with initial condition $(x,v,\bar{\mu})\in\mathbb{R}^{2d}\times\mathcal{P}(\mathbb{R}^d)$.
Write $\mathbb{E}[|\bar{x}_t|]=\mathbb{E}_{x\sim\bar{\mu}_t}[|x|]$ and $\mathbb{E}[|\bar{v}_t|]=\mathbb{E}_{v\sim\bar{\varphi}_t}[|v|]$, where $\bar{\mu}_t$ and $\bar{\varphi}_t$ are the marginal distribution in the first and second variable of the corresponding law of $(\bar{x}_t,\bar{v}_t)$ at time $t \geq 0$, respectively.
It is also notationally convenient to introduce 
\begin{align} \label{eq:nablaW}
	\mathbf{W}_0 \ := \ |\nabla_1 W(0,0)|.
\end{align}

\begin{lemma}\label{lem:estimate1}
Let $\bar{x},\bar{v}\in\mathbb{R}^d$. Let $t\in[0,\infty)$ satisfy \eqref{eq:t_cond}. Then, for $(\bar{x}_0,\bar{v}_0)=(\bar{x},\bar{v})$,
\begin{align}
 \max_{s\leq t} |\bar{x}_s|&\leq (1+(L+\epsilon\tilde{L})t^2)\max(|\bar{x}|,|\bar{x}+t\bar{v}|)+t^2\epsilon\tilde{L}\max_{s\leq t}\mathbb{E}[|\bar{x}_s|] +\epsilon t^2\mathbf{W}_0, \label{eq:lem_nonlin1}
\\ \max_{s\leq t} |\bar{v}_s|&\leq |\bar{v}|+(L+\epsilon\tilde{L})t(1+(L+\epsilon\tilde{L})t^2)\max(|\bar{x}|,|\bar{x}+t\bar{v}|) 
\nonumber
\\ & +\epsilon\tilde{L}t(1+(L+\epsilon\tilde{L})t^2)\max_{s\leq t}\mathbb{E}[|\bar{x}_s|] +\epsilon t(1+(L+\epsilon\tilde{L})t^2)\mathbf{W}_0. \label{eq:lem_nonlin2}
\end{align}
Moreover, 
\begin{align}
 \max_{s\leq t}\mathbb{E}[ |\bar{x}_s|]&\leq (1+(L+2\epsilon\tilde{L})t^2)\mathbb{E}[\max(|\bar{x}|,|\bar{x}+t\bar{v}|)] +\epsilon t^2\mathbf{W}_0, \label{eq:lem_nonlin3}
 \\ \max_{s\leq t}\mathbb{E}[ |\bar{v}_s|]&\leq \mathbb{E}[|\bar{v}|]+ (L+2\epsilon\tilde{L})(1+(L+2\epsilon\tilde{L})t^2)\mathbb{E}[\max(|\bar{x}|,|\bar{x}+t\bar{v}|)] \nonumber
 \\ & +\epsilon t(1+(L+\epsilon\tilde{L})t^2)\mathbf{W}_0.  \label{eq:lem_nonlin4}
\end{align}
\end{lemma}

\begin{proof}
The proof is postponed to \Cref{sec:proof_estimates}.
\end{proof}

Let $(\bar{x}_s,\bar{v}_s)_{s\geq 0}$ and $(\bar{x}'_s,\bar{v}_s')_{s\geq 0}$ be two processes driven by the Hamiltonian dynamics \eqref{eq:hamdyn_nonl} with initial condition $(x,v,\bar{\mu})\in\mathbb{R}^{2d}\times\mathcal{P}(\mathbb{R}^d)$ and $(y,u,\bar{\mu}')\in\mathbb{R}^{2d}\times\mathcal{P}(\mathbb{R}^d)$.
By \eqref{eq:hamdyn_nonl},  $(\bar{z}_s,\bar{w}_s)=(\bar{x}_s-\bar{x}_s',\bar{v}_s-\bar{v}_s')$ satisfies
\begin{equation} \label{eq:nonlin_diff}
\begin{cases}
& \frac{\rmd }{\rmd t} \bar{z}_s=\bar{w}_s
\\ & \frac{\rmd }{\rmd t} \bar{w}_s=-(\nabla V(\bar{x}_t)-\nabla V(\bar{x}_t'))-\epsilon \Big(\int_{\mathbb{R}^d} \nabla_x W(\bar{x}_t,\tilde{x}) \bar{\mu}_t(\rmd \tilde{x})- \int_{\mathbb{R}^d}\nabla_x W(\bar{x}'_t,\tilde{x}')\bar{\mu}_t'(\rmd \tilde{x}')\Big)
\\ & \bar{\mu}_t=\Law(\bar{x}_t) \qquad \bar{\mu}_t'=\Law(\bar{x}'_t).
\end{cases}
\end{equation}
We write $\mathbb{E}[|\bar{z}_s|]=\mathbb{E}_{x\sim \bar{\mu}_t,x'\sim \bar{\mu}_t'}[|x-x'|]$ and  $\mathbb{E}[|\bar{w}_s|]=\mathbb{E}_{v\sim \bar{\varphi}_t,v'\sim \bar{\varphi}_t'}[|v-v'|]$, where $\bar{\mu}_t$, $\bar{\mu}_t'$ $\bar{\varphi}_t$ and $\bar{\varphi}_t'$ are the marginal distributions in the position and velocity component, respectively.

\begin{lemma} \label{lem:estimate2}
Let $\bar{x},\bar{x}',\bar{v},\bar{v}'\in\mathbb{R}^d$. 
Let $t\in[0,\infty)$ satisfy \eqref{eq:t_cond}. Then, for $(\bar{z},\bar{w})=(\bar{x}-\bar{x}',\bar{v}-\bar{v}')$,
\begin{align}
&\max_{s\leq t}|\bar{z}_s-\bar{z}-s\bar{w}|\leq (L+\epsilon\tilde{L})t^2\max(|\bar{z}+t\bar{w}|,|\bar{z}|)+\epsilon\tilde{L}t^2\max_{s\leq t} \mathbb{E}[|\bar{z}_s|] , \label{eq:lem_nonlin5}
\\ &\max_{s\leq t}|\bar{z}_s|\leq (1+(L+\epsilon\tilde{L})t^2)\max(|\bar{z}+t\bar{w}|,|\bar{z}|)+\epsilon\tilde{L}t^2\max_{s\leq t} \mathbb{E}[|\bar{z}_s|], \label{eq:lem_nonlin6}
\\ & \max_{s\leq t}|\bar{w}_s-\bar{w}|\leq (L+\epsilon\tilde{L})t(1+(L+\epsilon\tilde{L})t^2)\max(|\bar{z}+t\bar{w}|,|\bar{z}|) 
 +\epsilon\tilde{L}(1+(L+\epsilon\tilde{L})t^2)\max_{s\leq t} \mathbb{E}[|\bar{z}_s|] , \label{eq:lem_nonlin7}
\\ & \max_{s\leq t}|\bar{w}_s|\leq |\bar{w}|+ (L+\epsilon\tilde{L})t(1+(L+\epsilon\tilde{L})t^2)\max(|\bar{z}+t\bar{w}|,|\bar{z}|) 
+\epsilon\tilde{L}t(1+(L+\epsilon\tilde{L})t^2)\max_{s\leq t} \mathbb{E}[|\bar{z}_s|] . \label{eq:lem_nonlin8}
\end{align}
Moreover,
\begin{align}
& \max_{s\leq t }\mathbb{E}[|\bar{z}_s|]\leq (1+(L+2\epsilon\tilde{L})t^2)\mathbb{E}[\max(|\bar{z}+t\bar{w}|,|\bar{z}|)], \label{eq:lem_nonlin9}
\\ & \max_{s\leq t }\mathbb{E}[|\bar{w}_s|]\leq (L+2\epsilon\tilde{L})t(1+(L+2\epsilon\tilde{L})t^2)\mathbb{E}[\max(|\bar{z}+t\bar{w}|,|\bar{z}|)]+\mathbb{E}[|\bar{w}|]. \label{eq:lem_nonlin10}
\end{align}
\end{lemma}

\begin{proof}
The proof is postponed to \Cref{sec:proof_estimates}.
\end{proof}

\subsection{Bounds in Region of Strong Convexity for Nonlinear Flow} \label{sec:strongconv}

\begin{lemma} \label{lem:convexarea}
Suppose that \Cref{ass}\ref{ass_Vmin}, \ref{ass_Vlip}, \ref{ass_Vconv} and \ref{ass_Wlip} hold.  Let $T\in(0,\infty)$ and $\epsilon\in[0,\infty)$  satisfy \begin{align} \label{eq:cond_t}
(L+2\epsilon\tilde{L})T^2 \ &\leq \ 1/4, \\
    \epsilon \tilde{L} \ &\le \  K /3 \;. \label{eq:C_eps}
\end{align}
Let $\bar{\mu}$ and $\bar{\mu}'$ be probability measures on $\mathbb{R}^d$ and let $\bar{x},\bar{x}',\bar{v},\bar{v}'\in\mathbb{R}^d$.
Let $(\bar{x}_t,\bar{v}_t)_{t\geq 0}$ and $(\bar{x}_t',\bar{v}_t')_{t\geq 0}$ be two solutions of the distribution-dependent Hamiltonian dynamics given by \eqref{eq:hamdyn_nonl} with initial condition $(\bar{x},\bar{v},\bar{\mu})$ and $(\bar{x}',\bar{v}',\bar{\mu}')$.
If $\bar{v}=\bar{v}'$, then
\begin{align} \label{eq:convexarea}
|\bar{x}_T-\bar{x}_T'|^2 \ \leq \  (1-(5/12)KT^2 )|\bar{x}-\bar{x}'|^2+\epsilon^2\tilde{L}^2T^4\left(\frac{7}{6}+\frac{3}{2KT^2}\right) \max_{s\leq T}\mathbb{E}[|\bar{x}_s-\bar{x}_s'|]^2+ \hat{C}T^2.
\end{align}

\end{lemma}

\begin{proof}
The proof is postponed to \Cref{sec:proof_estimates}.
\end{proof}

\subsection{Uniform-in-steps Second Moment Bound for Nonlinear Flow} \label{sec:secondmoment}

Next, we prove a uniform-in-steps second moment bound for the position component of the distribution-dependent Hamiltonian dynamics, which in turn, implies a corresponding bound for the Markov chain generated by nHMC.

\begin{lemma} \label{lem:secondmoment}
	Suppose $\mathbb{E}[|\bar{x}_0|^2]<\infty$. Suppose \Cref{ass} holds. Let $T\in(0,\infty)$ and $\epsilon\in[0,\infty)$  satisfy \eqref{eq:cond_t} and \eqref{eq:C_eps}. 
	Then there exists a finite constant $\mathbf{B}_1$ depending on $K$, $\epsilon$, $L$, $\mathcal{R}$, $T$, $\mathbb{E}[|\bar{x}_0|^2]$, $\mathbf{W}_0$ and $d$ such that both the second moment of the distribution-dependent Hamiltonian dynamics for $t\leq T$ is uniformly bounded, i.e.,
	\begin{align*}
	\sup_{0 \le t \le T}\mathbb{E}[|\bar{x}_t|^2]\leq \mathbf{B}_1<\infty,
	\end{align*}
	and the second moment of the Markov chain generated by nHMC is uniformly bounded, i.e.,
	\begin{align*}
	\sup_{m\in\mathbb{N}}\mathbb{E}[|\bar{\mathbf{X}}_m|^2]\leq \mathbf{B}_1<\infty.
	\end{align*}
 More precisely, the constant $\mathbf{B}_1$ is bounded by 
 \begin{align} \label{eq:C_1}
     \mathbf{B}_1\le \mathbb{E}[|\bar{x}_0|^2] +  (1280/13K)(\mathcal{R}^2(2L+K)+11d+6(\epsilon \mathbf{W}_0)^2T^2 +(45/2)(\epsilon^2/K) \mathbf{W}_0^2).
 \end{align}
\end{lemma}

\subsection{Proofs of estimates for the nonlinear Hamiltonian flow }\label{sec:proof_estimates}

To prove the aforementioned estimates for the nonlinear Hamiltonian flow,  note that by \Cref{ass}\ref{ass_Vlip} and \ref{ass_Wlip}, for all $\bar{x},\bar{x}'\in\mathbb{R}^d$ and $\eta,\eta'\in\mathcal{P}(\mathbb{R}^d)$, 
\begin{align}
& |\nabla U_{\eta}(\bar{x})|\leq L |\bar{x}|+\epsilon\tilde{L}(|\bar{x}|+\mathbb{E}_{x\sim \eta} [|x|])+\epsilon \mathbf{W}_0, \quad \text{and,} \label{eq:lipschitz1}
\\ & |\nabla U_{\eta}(\bar{x})-\nabla U_{\eta'}(\bar{x}')|\leq (L+\epsilon\tilde{L}) |\bar{x}-\bar{x}'|+\epsilon\tilde{L}\mathbb{E}_{{x}\sim \eta,{x}'\sim \eta'} [|x-{x}'|] \label{eq:lipschitz2}
\end{align} 
where $\mathbf{W}_0$ is the constant defined in \eqref{eq:nablaW}.

\begin{proof}[Proof of \Cref{lem:estimate1}]
Fix $\bar{x},\bar{v}\in\mathbb{R}^d$ and the probability measure $\bar{\nu}_0=\bar{\mu}_0\otimes\mathcal{N}(0,I_d)$ on $\mathbb{R}^{2d}$. Let $s\leq t$. For all $0\leq u\leq t$, set $\bar{\mu}_u(x)=\int_{\mathbb{R}^d}\bar{\nu}_u(x,\rmd y)$ and $\mathbb{E}[|\bar{x}_u|]=\mathbb{E}_{x\sim \bar{\mu}_u}[|x|]$. Then by \eqref{eq:hamdyn_nonl},
\begin{align*}
\bar{x}_s-\bar{x}-s\bar{v}=\int_0^s \int_0^{\lfloor r\rfloor} \Big(-\nabla U_{\bar{\mu}_u}(x_{u})\Big) \rmd u \ \rmd r. 
\end{align*}
By \eqref{eq:lipschitz1},
\begin{align*}
|\bar{x}_s-\bar{x}-s\bar{v}|\leq \frac{(L+\epsilon\tilde{L})t^2}{2}\max_{r\leq t}(|\bar{x}_r-\bar{x}-r\bar{v}|+|\bar{x}+r\bar{v}|)+\frac{\epsilon\tilde{L}t^2}{2}\max_{r\leq t} \mathbb{E}[|\bar{x}_r|]+\frac{t^2}{2}\epsilon\mathbf{W}_0.
\end{align*}
By invoking condition \eqref{eq:t_cond}, we obtain
\begin{align*}
\max_{s\leq t}|\bar{x}_s-\bar{x}-s\bar{v}|&\leq (L+\epsilon\tilde{L})t^2\max_{s\leq t} |\bar{x}+s\bar{v}|+\epsilon\tilde{L}t^2\max_{s\leq t}\mathbb{E}[|\bar{x}_s|]+\epsilon t^2\mathbf{W}_0
\\ & =(L+\epsilon\tilde{L})t^2\max(|\bar{x}|,|\bar{x}+t\bar{v}|)+\epsilon\tilde{L}t^2\max_{s\leq t}\mathbb{E}[|\bar{x}_s|]+\epsilon t^2\mathbf{W}_0.
\end{align*}
By applying the triangle inequality we obtain \eqref{eq:lem_nonlin1}. By \eqref{eq:hamdyn_nonl} and \eqref{eq:lipschitz1},
\begin{align*} 
|\bar{v}_s-\bar{v}|\leq \int_0^s \max_{u\leq t}|\nabla U_ {\bar{\mu}_u}(\bar{x}_u)|\rmd r\leq (L+\epsilon\tilde{L})t \max_{u\leq t}|\bar{x}_u|+\epsilon\tilde{L}t\max_{u\leq t} \mathbb{E}[|\bar{x}_u|]+\epsilon t\mathbf{W}_0.
\end{align*}
By inserting \eqref{eq:lem_nonlin1} in the previous expression, we obtain
\begin{align*}
|\bar{v}_s-\bar{v}|&\leq (L+\epsilon\tilde{L})t(1+(L+\epsilon\tilde{L})t^2)\max(|\bar{x}|,|\bar{x}+t\bar{v}|)
 +\epsilon\tilde{L}t(1+(L+2\epsilon\tilde{L})t^2)\max_{u\leq t}\mathbb{E}[|\bar{x}_u|]
 \\ & + \epsilon t(1+(L+\epsilon\tilde{L})t^2)\mathbf{W}_0.
\end{align*}
Applying the triangle inequality we obtain \eqref{eq:lem_nonlin2}. To obtain \eqref{eq:lem_nonlin3} and \eqref{eq:lem_nonlin4}, we note that by \eqref{eq:hamdyn_nonl},
\begin{align*}
\mathbb{E}[|\bar{x}_s-\bar{x}-s\bar{v}|]&\leq \int_0^s\int_0^r \mathbb{E}[|\nabla U_{\bar{\mu}_u}(\bar{x}_{u})|]\rmd u \rmd r 
 \leq \frac{(L+2\epsilon\tilde{L})t^2}{2}\max_{r\leq t}\mathbb{E}[|\bar{x}_r|]+ \frac{\epsilon t^2}{2}\mathbf{W}_0.
\end{align*}
Similarly, we obtain the estimate \eqref{eq:lem_nonlin1},
\begin{align*}
\max_{s\leq t} \mathbb{E}[|\bar{x}_s-\bar{x}-s\bar{v}|] &\leq (L+2\epsilon\tilde{L})t^2\max_{r\leq t}\mathbb{E}[|\bar{x}+r\bar{v}|]+\epsilon t^2\mathbf{W}_0
\\ & \leq (L+2\epsilon\tilde{L})t^2\mathbb{E}[\max(|\bar{x}+t\bar{v}|,|\bar{x}|)]+\epsilon t^2\mathbf{W}_0.
\end{align*}
By applying the triangle inequality, \eqref{eq:lem_nonlin3} is obtained. By \eqref{eq:hamdyn_nonl} and \eqref{eq:lem_nonlin3},
\begin{align*}
\mathbb{E}[|\bar{v}_s-\bar{v}|]\leq (L+2\epsilon\tilde{L})t \max_{r\leq t} \mathbb{E}[|\bar{x}_r|]&\leq (L+2\epsilon\tilde{L})t(1+(L+2\epsilon\tilde{L})t^2)\mathbb{E}[\max(|\bar{x}|,|\bar{x}+t\bar{v}|)]
\\ & +\epsilon t(1+(L+\epsilon\tilde{L})t^2)\mathbf{W}_0
\end{align*}
and \eqref{eq:lem_nonlin4} is obtained by the triangle inequality.
\end{proof}

\begin{proof}[Proof of \Cref{lem:estimate2}]
By \eqref{eq:nonlin_diff} and \eqref{eq:lipschitz2},
\begin{align*}
|\bar{z}_s-\bar{z}-s\bar{w}|
 &\leq \int_0^s \int_0^r \max_{v\leq t}|-\nabla U_{\bar{\mu}_v}(\bar{x}_v)+\nabla U_{\bar{\mu}_v'}(\bar{x}'_v)|\rmd u \rmd r
\\ & \leq \frac{(L+\epsilon\tilde{L})t^2}{2}\max_{r\leq t}|\bar{z}_r|+\frac{\epsilon\tilde{L}t^2}{2}\max_{r\leq t}\mathbb{E}[|\bar{z}_r|].
\end{align*}
Similar to the proof of \Cref{lem:estimate1}, it holds
\begin{align*}
\max_{s\leq t} |\bar{z}_s-\bar{z}-s\bar{w}|\leq (L+\epsilon\tilde{L})t^2\max(|\bar{z}+t\bar{w}|,|\bar{z}|)+\epsilon\tilde{L}t^2\max_{s\leq t}\mathbb{E}[|\bar{z}_s|],
\end{align*}
which yields \eqref{eq:lem_nonlin5} and after applying the triangle inequality \eqref{eq:lem_nonlin6}. By \eqref{eq:nonlin_diff} and \eqref{eq:lipschitz2},
\begin{align*}
|\bar{w}_s-\bar{w}|\leq \int_0^s\max_{v\leq t}|-\nabla U_{\bar{\mu}_v}(\bar{x}_v)+\nabla U_{\bar{\mu}_v'}(\bar{x}'_v)|\rmd u \rmd r\leq (L+\epsilon\tilde{L})t \max_{r\leq t}|\bar{z}_r|+\epsilon\tilde{L}t\max_{r\leq t}\mathbb{E}[|\bar{z}_r|].
\end{align*}
By \eqref{eq:lem_nonlin6},
\begin{align*}
\max_{s\leq t} |\bar{w}_s-\bar{w}|& \leq (L+\epsilon\tilde{L})t(1+(L+2\epsilon\tilde{L})t^2)\max(|\bar{z}|,|\bar{z}+t\bar{w}|)
 +\epsilon\tilde{L}t(1+(L+\epsilon\tilde{L})t^2\max_{s\leq t}\mathbb{E}[|\bar{z}_s|], 
\end{align*}
and hence \eqref{eq:lem_nonlin7} and \eqref{eq:lem_nonlin8} hold. Equation \eqref{eq:lem_nonlin9} and \eqref{eq:lem_nonlin10} hold analogously by taking the expectation.
\end{proof}

\begin{proof}[Proof of \Cref{lem:convexarea}]
It is notationally convenient to define \begin{align*}
\bar{z}_t & := \bar{x}_t - \bar{x}'_t, \quad \bar{w}_t := \bar{v}_t - \bar{v}'_t, \\
\beta_t &:= \nabla V(\bar{x}_t) - \nabla V(\bar{x}'_t),  \quad \Omega_t :=\int_{\mathbb{R}^d} \left( \nabla_1 W(\bar{x}_t, y) - \nabla_1 W(\bar{x}'_t,y') \right) \rmd \bar{\mu}_t(y) \rmd \bar{\mu}'(y') \;.
\end{align*}
Therefore, $\nabla U_{\bar{\mu}_t}(\bar{x}_t) - \nabla U_{\bar{\mu}'_t}(\bar{x}'_t) = \beta_t + \epsilon \Omega_t$.   Let $\bar{a}_t := \norm{\bar{z}_t}^2$ and  $\bar{b}_t := 2 \langle \bar{z}_t, \bar{w}_t \rangle$.  Note that $\rho_t := \langle \bar{z}_t, \beta_t \rangle$ satisfies \begin{equation} \label{ieq:rhot}
K \bar{a}_t + \frac{1}{L} \norm{\beta_t}^2 - \hat{C} \overset{\ref{ass_Vconv}}{\le}  \rho_t \overset{\ref{ass_Vlip}}{\le}  L \bar{a}_t
\end{equation}
From \eqref{eq:hamdyn_nonl}, note that \begin{align} 
		\begin{cases}		
		\frac{d}{dt} \bar{z}_t \, = \bar{w}_t ,
		\\ \frac{d}{dt} \bar{w}_t \, = \, -  \beta_t - \epsilon \Omega_t \;, 
		\\ \bar{\mu}_t=\Law(\bar{x}_t)\, \quad \bar{\mu}'=\Law(\bar{x}_t') \; ,
		\end{cases} \label{eq:hamdyn_nonl_diff}
\end{align}
and hence, \begin{equation}
    \begin{aligned}
        \frac{d}{dt} \bar{a}_t := \bar{b}_t, \quad \frac{d}{dt} \bar{b}_t = 2 |\bar{w}_t|^2 - 2 \langle \bar{z}_t, \beta_t \rangle - 2 \epsilon \langle \bar{z}_t, \Omega_t \rangle \;. 
    \end{aligned}
\end{equation}
By variation of parameters, \begin{equation} \label{vop_bar_at} 
    \bar{a}_T = c_T \bar{a}_0 + \int_0^T s_{T-r} \left( 2 \norm{\bar{w}_r}^2 - 2 \rho_r + \delta_r \right) dr
\end{equation} 
where $\delta_t := K \bar{a}_t  - 2 \epsilon \langle \bar{z}_t, \Omega_t \rangle$, $c_t:=\cos(\sqrt{ K} t )$, and $s_t := \frac{\sin( \sqrt{ K} t )}{ \sqrt{ K}} $.  
To upper bound  $\delta_t$  in \eqref{vop_bar_at},  
\begin{align}    \label{eq:delta} 
& \delta_t  \  \le \  K \bar{a}_t  + 2 \epsilon \norm{\bar{z}_t}\norm{\Omega_t} \ \overset{\ref{ass_Wlip}}{\le}  K \bar{a}_t + 2 \epsilon \norm{\bar{z}_t}\tilde{L}(\norm{\bar{z}_t}+ \mathbb{E}[\norm{\bar{z}_t}]) \ \overset{\eqref{eq:C_eps}}{\le}   2K \bar{a}_t + \frac{3\epsilon^2\tilde{L}^2}{K} \max_{r\le T} \mathbb{E}[\norm{\bar{z}_r}]^2 \  \qquad \text{for }t\le T, 
\end{align}
where $\mathbb{E}[\norm{\bar{z}_r}]=\int_{\mathbb{R}^{2d}} |y-y'| \rmd \bar{\mu}_r(y)\rmd \bar{\mu}_r'(y')$ for any $r\in[0,T]$.

To upper bound the terms involving $|\bar{w}_t|^2$ in \eqref{vop_bar_at},  use \eqref{eq:hamdyn_nonl_diff} and note that $\bar{w}_0=0$, since the initial velocities in the two copies are synchronized.  Therefore, \begin{align} \label{eq:Wt2}
 |\bar{w}_t|^2 \, = \,   \norm{  \int_0^{t}   \beta_s + \epsilon \Omega_s  ds }^2 \ \, \le \,  t  \int_0^{t} \norm{\beta_s + \epsilon \Omega_s  }^2 ds
		\, \le \,   t  \int_0^{t} \left(8\norm{\beta_s}^2 + \epsilon^2 \frac{8}{7}\norm{ \Omega_s  }^2\right) ds 
\end{align}
where we used the Cauchy-Schwarz inequality.  By \eqref{eq:cond_t}, note that $s_{t-r}$ is monotonically decreasing with $r$,  \begin{equation} \label{eq:sinus}
0 \le s_{t-r} \le s_{t-s} \;, \quad \text{for $s \le r \le t \le T \le \frac{\pi/2}{\sqrt{ K}} $} \;. \end{equation} 
Therefore, combining \eqref{eq:Wt2} and \eqref{eq:sinus}, and by Young's inequality and Fubini's Theorem,   
\begin{align}
& \int_0^t s_{t-r} |\bar{w}_r|^2 dr  \, \overset{\eqref{eq:Wt2}}{\le} \,
 2 \int_0^t \int_0^r  r s_{t-r} \left(4\norm{\beta_s}^2 + \epsilon^2 \frac{4}{7} \norm{ \Omega_s  }^2\right)  ds dr \nonumber
 \\ & \qquad \, \overset{\eqref{eq:sinus}}{\le} \,
 2 \int_0^t \int_0^r  r s_{t-s}  \left(4\norm{\beta_s}^2 + \epsilon^2 \frac{4}{7} \norm{ \Omega_s  }^2\right) ds dr
 \le \,  2 \int_0^t \int_s^t  r s_{t-s}  \left(4\norm{\beta_s}^2 + \epsilon^2 \frac{4}{7} \norm{ \Omega_s  }^2\right)  dr ds  \, 
 \nonumber \\
& \qquad \
 \le \,  t^2  \int_0^t   s_{t-s} \left( 4 \norm{\beta_s}^2 + \epsilon^2 \frac{4}{7} \norm{ \Omega_s }^2\right)  ds \;. \label{ieq:Wt2}
\end{align} 
To upper bound the mean-field interaction force,   
\begin{align}  & \norm{\Omega_t}^2    \ = \ \norm{ \int_{\mathbb{R}^{2d}} \nabla_1 W(\bar{x}_t,y) - \nabla_1 W(\bar{x}'_t,y) \rmd \bar{\mu}_t(y) \rmd \bar{\mu}_t'(y') }^2  \nonumber \\
& \quad \le \left( \  \int_{\mathbb{R}^{2d}} \norm{ \nabla_1 W(\bar{x}_t,y) - \nabla_1 W(\bar{x}'_t,y) } \rmd \bar{\mu}_t(y) \rmd \bar{\mu}_t'(y')  \right)^2     \nonumber  \\
& \quad \overset{\ref{ass_Wlip}}{\le} \left( \tilde{L} ( \norm{ \bar{z}_t} + \mathbb{E}[\norm{\bar{z}_t}] ) \right)^2      \nonumber \\
& \quad \le 2 \tilde{L}^2 \bar{a}_t + 2 \tilde{L}^2 \max_{s\le t}\mathbb{E}[\norm{\bar{z}_s}]^2 .
\label{ieq:Omt}
\end{align}
We note that by \eqref{eq:lem_nonlin6}, \eqref{eq:cond_t} and Young's inequality, it holds
\begin{align}
\epsilon \tilde{L}T^2 \max_{t\le T} \bar{a}_t &\le \epsilon \tilde{L}t^2\left( \frac{5}{4} \norm{\bar{z}_0} + \epsilon\tilde{L}T^2 \max_{s\le T}\mathbb{E}[\norm{\bar{z}_s}]\right)^2 \le \epsilon \tilde{L}t^2 \left(  \frac{7}{5}\left(\frac{5}{4}\right)^2 \bar{a}_0 + \frac{7}{2} (\epsilon\tilde{L}T^2)^2 \max_{s\le T}\mathbb{E}[\norm{\bar{z}_s}]^2 \right)  \nonumber
\\ &{\le}  \frac{7}{64} \bar{a}_0 + \frac{7\epsilon \tilde{L}T^2}{2\cdot 20^2} \max_{s\le T}\mathbb{E}[\norm{\bar{z}_s}]^2 \label{eq:bar_a2}.
\end{align}
where we used that by \eqref{eq:cond_t} and \eqref{eq:C_eps}
\begin{align} \label{eq:cond_t_conseq}
\epsilon\tilde{L}t^2\leq (1/5)(K+2\epsilon\tilde{L})\leq (1/5)(L+2\epsilon\tilde{L})\leq (1/20)\qquad \text{for } t\le T.
\end{align}

 Inserting these bounds into the second term of \eqref{vop_bar_at} and simplifying yields 
\begin{align}
&  \int_0^T s_{T-r} \left( 2 \norm{\bar{w}_r}^2 - 2 \rho_r + \delta_r \right) dr  \nonumber  \\
&  \quad \overset{\eqref{eq:delta}}{\le}  \int_0^T s_{T-r} \left( 2 \norm{\bar{w}_r}^2  - 2 \rho_r + 2 K \bar{a}_r + \frac{3\epsilon^2\tilde{L}^2}{K} \max_{s\le T} \mathbb{E}[\norm{\bar{z}_s}]^2  \right) dr \nonumber \\
& \quad \overset{\eqref{ieq:Wt2}}{\le} \int_0^T s_{T-r} \left( T^2 \left(8\norm{\beta_r}^2 + \epsilon^2 \frac{8}{7}  \norm{ \Omega_r }^2\right)    - 2 \rho_r + 2K  \bar{a}_r + \frac{3\epsilon^2\tilde{L}^2}{K} \max_{s\le T} \mathbb{E}[\norm{\bar{z}_s}]^2   \right)  dr \nonumber \\
& \quad \overset{\eqref{ieq:Omt}}{\le} \int_0^T s_{T-r} \left( T^2 \left(8 \norm{\beta_r}^2 +  \frac{16}{7} \epsilon^2\tilde{L}^2 (\bar{a}_r + \max_{s\le r}\mathbb{E}[\norm{\bar{z}_s}]^2)\right)    - 2 \rho_r + 2K  \bar{a}_r  + \frac{3\epsilon^2\tilde{L}^2}{K} \max_{s\le T} \mathbb{E}[\norm{\bar{z}_s}]^2   \right)  dr \nonumber \\
& \quad \overset{\eqref{ieq:rhot}}{\le}  \int_0^T s_{T-r} \left( T^2 \left(8 \norm{\beta_r}^2 +  \frac{16}{7}  \epsilon^2\tilde{L}^2 (\bar{a}_r +  \max_{s\le r}\mathbb{E}[\norm{\bar{z}_s}]^2)\right)  +2 \hat{C} -\frac{2}{L} \norm{\beta_r}^2  + \frac{3\epsilon^2\tilde{L}^2}{K} \max_{s\le T} \mathbb{E}[\norm{\bar{z}_s}]^2   \right)  dr \nonumber \\
& \quad \overset{\eqref{eq:cond_t}}{\le}  \int_0^T s_{T-r} \left( T^2 \left( \frac{16}{7} \epsilon^2\tilde{L}^2 \bar{a}_r + \frac{16}{7} \epsilon^2 \tilde{L}^2 \max_{s\le r}\mathbb{E}[\norm{\bar{z}_s}]^2\right)  + 2\hat{C}  + \frac{3\epsilon^2\tilde{L}^2}{K} \max_{s\le T} \mathbb{E}[\norm{\bar{z}_s}]^2   \right)  dr \nonumber \\
& \quad \overset{\eqref{eq:bar_a2}, \eqref{eq:cond_t_conseq} }{\le}    \int_0^T s_{T-r} \left( \frac{\epsilon \tilde{L}}{4} \bar{a}_0 + \left(\frac{(\epsilon \tilde{L} T)^2}{50}  + \frac{16(\epsilon\tilde{L} T)^2}{7}\right) \max_{s\le T}\mathbb{E}[\norm{\bar{z}_s}]^2 +   2\hat{C}   + \frac{3\epsilon^2\tilde{L}^2}{K} \max_{s\le T} \mathbb{E}[\norm{\bar{z}_s}]^2   \right)  dr \nonumber \\
& \quad \overset{\eqref{eq:cond_t}, \eqref{eq:C_eps}}{\le}   T^2 \hat{C} + \frac{\epsilon \tilde{L}T^2}{8} \bar{a}_0+  \epsilon^2\tilde{L}^2T^4\left(\frac{7}{6}+\frac{3}{2KT^2}\right) \max_{s\le T}\mathbb{E}[\norm{\bar{z}_s}]^2
\nonumber 
\end{align}
where in the last step we used $s_{T-r} \le T-r$. 
Inserting this upper bound back into \eqref{vop_bar_at} yields  
\begin{align*}
 \bar{a}_T  \, \le \, c_T \bar{a}_0 + T^2 \hat{C} +  \frac{\epsilon \tilde{L}T^2}{8} \bar{a}_0+  \epsilon^2\tilde{L}^2T^4\left(\frac{7}{6}+\frac{3}{2KT^2}\right) \max_{s\le T}\mathbb{E}[\norm{\bar{z}_s}]^2 \;. 
\end{align*} 
The required estimate is then obtained by inserting the elementary inequality \begin{align}
c_T &\, \le \, 1 - (1/2) K T^2 + (1/6)  K^2 T^4 \, \overset{\eqref{eq:cond_t}}{\le} \, 1 - (1/2) K T^2 + (1/24)  K T^2 (K/L)  \le 1 - (1/2 - 1/24) K T^2 \label{eq:bound_c_T}
\end{align} which is valid since $K \le L$.  The required result then holds because $1/2 - 1/24 -1/(3\cdot 8) = 5/12$.
\end{proof}


\begin{proof}[Proof of Lemma~\ref{lem:secondmoment}] 

The proof somewhat resembles the proof of \Cref{lem:convexarea} except to bound the second moment the difference is replaced by a single component. Define 
\begin{align*}
 \Theta_t :=\int_{\mathbb{R}^d} \nabla_1 W(\bar{x}_t, y) \rmd \bar{\mu}_t(y) \;.
\end{align*}
Then, $\nabla U_{\bar{\mu}_t}(\bar{x}_t) = \nabla V(\bar{x}_t) + \epsilon \Theta_t$.  
Furthermore, let $\bar{a}_t := \norm{\bar{x}_t}^2$ and  $\bar{b}_t := 2 \langle \bar{x}_t, \bar{v}_t \rangle$.  Note that $\rho_t := \langle \bar{x}_t, \nabla V(\bar{x}_t) \rangle$ satisfies \begin{equation} \label{ieq:rhot_moment}
K \bar{a}_t + \frac{1}{L} \norm{\nabla V(\bar{x}_t)}^2 - \hat{C} \overset{\ref{ass_Vconv}}{\le}  \rho_t \overset{\ref{ass_Vlip}}{\le}  L \bar{a}_t .
\end{equation}
From \eqref{eq:hamdyn_nonl},  
\begin{equation}
    \begin{aligned}
        \frac{d}{dt} \bar{a}_t := \bar{b}_t, \quad \frac{d}{dt} \bar{b}_t = 2 |\bar{v}_t|^2 - 2 \langle \bar{x}_t, \nabla V(\bar{x}_t) \rangle - 2 \epsilon \langle \bar{x}_t, \Theta_t \rangle \;. 
    \end{aligned}
\end{equation}
By variation of parameters, for $t\le T$ \begin{equation} \label{vop_bar_at_moment} 
    \bar{a}_t = c_t \bar{a}_0 + \int_0^t s_{t-r} \left( 2 \norm{\bar{v}_r}^2 - 2 \rho_r + \delta_r \right) dr
\end{equation} 
where $\delta_r := K \bar{a}_r  - 2 \epsilon \langle \bar{x}_r, \Theta_r \rangle$, $c_t:=\cos(\sqrt{ K} t )$, and $s_t := \frac{\sin( \sqrt{ K} t )}{ \sqrt{ K}} $.  
We bound  $\delta_t$  in \eqref{vop_bar_at_moment} from above by   
\begin{equation} \label{eq:delta_moment} 
\begin{aligned}    
\delta_t  \  &\le \  K \bar{a}_t  + 2 \epsilon \norm{\bar{x}_t}\norm{\Theta_t} \ \overset{\ref{ass_Wlip}}{\le}  K \bar{a}_t + 2 \epsilon \norm{\bar{x}_t}\tilde{L}(\norm{\bar{x}_t}+ \mathbb{E}[\norm{\bar{x}_t}] +\mathbf{W}_0/\tilde{L}) \ 
\\ & \overset{\eqref{eq:C_eps}}{\le} \  2K \bar{a}_t + \epsilon\tilde{L} \max_{r\le T} \mathbb{E}[\norm{\bar{x}_r}]^2+ K/45 \bar{a}_t + \epsilon^2 45/K\mathbf{W}_0^2  \  \qquad \text{for }t\le T, 
\end{aligned}
\end{equation}
where $\mathbb{E}[\norm{\bar{x}_r}]=\int_{\mathbb{R}^{d}} |y| \rmd \bar{\mu}_r(y)$ for any $r\in[0,T]$.

To upper bound the terms involving $|\bar{v}_t|^2$ in \eqref{vop_bar_at_moment},  use \eqref{eq:hamdyn_nonl}.   Therefore, \begin{align} \label{eq:Wt2_moment}
 |\bar{v}_t|^2 &\, = \,   \norm{  -\int_0^{t}  ( \nabla V(\bar{x}_s) + \epsilon \Theta_s)  ds + \bar{v}_0}^2 \ \, \le \,  8 \norm{\bar{v}_0}^2 + \frac{8}{7} t  \int_0^{t} \norm{ \nabla V(\bar{x}_s) + \epsilon \Theta_s}^2 ds \nonumber \\
&		\, \le \, 8 \norm{\bar{v}_0}^2  + t  \int_0^{t}  \left(8\norm{\nabla V(\bar{x}_s)}^2 + \frac{4}{3}\epsilon^2 \norm{ \Theta_s  }^2\right) ds 
\end{align}
where we used the Cauchy-Schwarz inequality.  
Therefore, combining \eqref{eq:Wt2_moment} and \eqref{eq:sinus}, and by Fubini's Theorem,   
\begin{align}
& \int_0^t s_{t-r} |\bar{v}_r|^2 dr  \, \overset{\eqref{eq:Wt2_moment}}{\le} \,
  2 \int_0^t \int_0^r  r s_{t-r} \left(4\norm{\nabla V(\bar{x}_s)}^2 + \frac{2}{3}\epsilon^2 \norm{ \Theta_s  }^2\right)  ds dr + 8\int_0^t s_{t-r} |\bar{v}_0|^2 dr \nonumber
 \\ & \qquad \ 
 \overset{\eqref{eq:sinus}}{\le} \,
 2 \int_0^t \int_0^r  r s_{t-s}  \left(4\norm{\nabla V(\bar{x}_s)}^2 + \frac{2}{3}\epsilon^2 \norm{ \Theta_s  }^2\right) ds dr + 8\int_0^t s_{t-r} |\bar{v}_0|^2 dr
\nonumber \\
& \qquad \ \le \,  2 \int_0^t \int_s^t  r s_{t-s} \left(4\norm{\nabla V(\bar{x}_s)}^2 + \frac{2}{3}\epsilon^2 \norm{ \Theta_s  }^2\right)  dr ds + 8\int_0^t s_{t-r} |\bar{v}_0|^2 dr
\nonumber \\ & \qquad \, \le \,  t^2  \int_0^t   s_{t-s} \left(4\norm{\nabla V(\bar{x}_s)}^2 + \frac{2}{3}\epsilon^2 \norm{ \Theta_s  }^2\right)  ds + 8\int_0^t s_{t-r} |\bar{v}_0|^2 dr \;. \label{ieq:Wt2_moment}
\end{align} 
To upper bound the mean-field interaction force,   
\begin{align}  & \norm{\Theta_t}^2    \ = \ \norm{ \int_{\mathbb{R}^{d}} \nabla_1 W(\bar{x}_t,y)  \rmd \bar{\mu}_t(y)  }^2  \quad \le \Big( \  \int_{\mathbb{R}^{d}} \norm{ \nabla_1 W(\bar{x}_t,y) } \rmd \bar{\mu}_t(y)   \Big)^2     \nonumber  \\
& \quad \overset{\ref{ass_Wlip}}{\le} \Big( \tilde{L} ( \norm{ \bar{x}_t} + \mathbb{E}[\norm{\bar{x}_t}] ) + \mathbf{W}_0 \Big)^2      \quad \le 2 \tilde{L}^2 \bar{a}_t + 3 \tilde{L}^2 \max_{s\le t}\mathbb{E}[\norm{\bar{x}_s}]^2 + 6 \mathbf{W}_0^2.
\label{ieq:Omt_moment}
\end{align}

 Inserting these bounds into the second term of \eqref{vop_bar_at_moment} and simplifying yields 
\begin{align}
&  \int_0^t s_{t-r} \left( 2 \norm{\bar{v}_r}^2 - 2 \rho_r + \delta_r \right) dr  \nonumber  \\
&  \quad \overset{\eqref{eq:delta_moment}}{\le}  \int_0^t s_{t-r} \left( 2 \norm{\bar{v}_r}^2  - 2 \rho_r + 2K\bar{a}_r + \epsilon\tilde{L} \max_{s\le T} \mathbb{E}[\norm{\bar{x}_s}]^2+ \frac{K}{45} \bar{a}_r + \frac{\epsilon^2 45}{K}\mathbf{W}_0^2  \right) dr \nonumber \\
& \quad \overset{\eqref{ieq:Wt2_moment}}{\le} \int_0^t s_{t-r} \left( T^2 \left(8\norm{\nabla V(\bar{x}_r)}^2 + \frac{4}{3}\epsilon^2 \norm{ \Theta_r  }^2\right)    - 2 \rho_r + 2K  \bar{a}_r + \epsilon\tilde{L} \max_{s\le T} \mathbb{E}[\norm{\bar{x}_s}]^2  + \frac{K}{45} \bar{a}_r  \right)  dr \nonumber \\
&\qquad + \int_0^t s_{t-r}\Big( \frac{\epsilon^2 45}{K}\mathbf{W}_0^2+16 |\bar{v}_0|^2\Big) dr \nonumber \\
& \quad \overset{\eqref{ieq:Omt_moment}, \eqref{ieq:rhot_moment}}{\le} \int_0^t s_{t-r} \left( T^2 \left(8\norm{\nabla V(\bar{x}_r)}^2 + \frac{4}{3} \epsilon^2 \left( 2 \tilde{L}^2 \bar{a}_r + 3 \tilde{L}^2 \max_{s\le r}\mathbb{E}[\norm{\bar{x}_s}]^2 + 6 \mathbf{W}_0^2 \right) \right)  +  2\hat{C}-\frac{2}{L}\norm{\beta_r}^2    \right)  dr \nonumber \\
&\qquad + \int_0^t s_{t-r} \left(  \epsilon\tilde{L} \max_{s\le T} \mathbb{E}[\norm{\bar{x}_s}]^2 + \frac{K}{45} \bar{a}_r + \frac{ \epsilon^2 45}{K}\mathbf{W}_0^2+16 |\bar{v}_0|^2 \right) dr \nonumber \\
& \quad \overset{\eqref{eq:cond_t}}{\le}  \int_0^t s_{t-r} \left( T^2 \left( \frac{4}{3} \epsilon^2 \left( 2 \tilde{L}^2 \bar{a}_r + 3 \tilde{L}^2 \max_{s\le T}\mathbb{E}[\norm{\bar{x}_s}]^2 + 6 \mathbf{W}_0^2 \right) \right)  +  2\hat{C}  \right)  dr \nonumber \\
&\qquad + \int_0^t s_{t-r} \left(  \epsilon\tilde{L} \max_{s\le T} \mathbb{E}[\norm{\bar{x}_s}]^2 + \frac{K}{45} \bar{a}_r + \frac{\epsilon^2 45}{K}\mathbf{W}_0^2+16 |\bar{v}_0|^2 \right) dr \nonumber \\
& \quad \overset{\eqref{eq:cond_t}}{\le} \frac{t^2}{2}\left(\left(\epsilon \tilde{L} \frac{2 }{15}+\frac{K}{45}\right)\max_{s\le T}\bar{a}_s + \epsilon \tilde{L}\frac{6}{5}\max_{s\le T}\mathbb{E}[\norm{x_s}]^2+  2\hat{C}+  \frac{\epsilon^2 45}{K}\mathbf{W}_0^2+ 8\epsilon^2 T^2 \mathbf{W}_0^2 + 16 \norm{\bar{v}_0}^2 \right) \label{eq:s_T_moment}
\end{align}
where in the last step we used $s_{t-r} \le t-r$. 
Inserting this upper bound back into \eqref{vop_bar_at_moment}, using \eqref{eq:C_eps} and taking expectation yields   
\begin{align}
& \mathbb{E} [\bar{a}_t]  \, \le \, c_t \mathbb{E}[\bar{a}_0] +  \frac{t^2}{2}\left(K\frac{1 }{15}\mathbb{E}[\max_{s\le T}\bar{a}_s] + K\frac{2}{5}\max_{s\le T}\mathbb{E}[\norm{x_s}]^2+ \left( \frac{45 \epsilon^2}{K}+8 \epsilon^2 T^2 \right) \mathbf{W}_0^2  + 16 \mathbb{E}[\norm{\bar{v}_0}^2] \right) 
+  t^2\hat{C} \nonumber
\\ & \quad \, \le \, (1 - (11/24) K t^2)\mathbb{E}[\bar{a}_0] +  \frac{t^2}{2}\left(K \frac{1}{15}\mathbb{E}[\max_{s\le T}\bar{a}_s] + K\frac{2}{5}\max_{s\le T}\mathbb{E}[\norm{x_s}]^2 + 16 \mathbb{E}[\norm{\bar{v}_0}^2] \right)  \nonumber
\\ & \qquad + t^2\left( \frac{45\epsilon^2}{2K}+4\epsilon^2 T^2 \right)  \mathbf{W}_0^2 +  t^2\hat{C} \;, \label{eq:expa_moment}
\end{align} 
where in the last step we used the inequality 
\begin{align}
c_t &\, \le \, 1 - (1/2) K t^2 + (1/6)  K^2 t^4 \, \overset{\eqref{eq:cond_t}}{\le} \, 1 - (1/2) K t^2 + (1/24)  K t^2 (K/L)  \le 1 - (1/2 - 1/24) K t^2 \label{eq:bound_c_T_moment}
\end{align} which is valid since $K \le L$.  
We note that by \eqref{eq:lem_nonlin1}, \eqref{eq:lem_nonlin4} and \eqref{eq:cond_t}, it holds
\begin{align}
& \mathbb{E}[ \max_{s\le T} \bar{a}_s] \le \mathbb{E} \left[ \left( \frac{5}{4} \norm{\bar{x}_0} + \frac{5}{4}T \norm{\bar{v}_0}+ \epsilon\tilde{L}T^2 \max_{s\le T}\mathbb{E}[\norm{\bar{x}_s}]+\epsilon T^2 \mathbf{W}_0 \right)^2 \right] \nonumber
\\ & \quad {\le} \mathbb{E}\left[ \left( \frac{5}{4} \norm{\bar{x}_0} + \frac{5}{4}T \norm{\bar{v}_0}+ \frac{1}{20} \left(\frac{5}{4} \mathbb{E}[\norm{\bar{x}_0}+T\norm{\bar{v}_0}] +\epsilon T^2 \mathbf{W}_0 \right) +\epsilon T^2 \mathbf{W}_0 \right)^2 \right] \nonumber
\\ & \quad \le  \mathbb{E}\left[ 2 \left( \frac{5}{4} \norm{\bar{x}_0}+ \frac{1}{20}\frac{5}{4} \mathbb{E}[\norm{\bar{x}_0}] \right)^2 + 4\left( \frac{5}{4}T \norm{\bar{v}_0}+ \frac{1}{16}  T\mathbb{E}[\norm{\bar{v}_0}] \right)^2+4\left(\frac{21}{20}\epsilon T^2 \mathbf{W}_0 \right)^2 \right] \nonumber
\\ & \quad \le  2 \left( \frac{5}{4}\right)^2 \frac{441}{400} \mathbb{E}[\bar{a}_0] + 4 \left( \frac{5}{4}\right)^2 \frac{441}{400} T^2 \mathbb{E}[\norm{\bar{v}_0}^2]  + 4 \left(\frac{21}{20}\epsilon T^2 \mathbf{W}_0 \right)^2 \quad \text{and} \label{eq:bar_a2_moment} \\
    & \max_{s\le T}\mathbb{E}[\norm{\bar{x}_s}]^2\le \left( \frac{5}{4} \mathbb{E}[\norm{\bar{x}_0}]+\frac{5}{4} T\mathbb{E}[\norm{\bar{v}_0}] +\epsilon T^2 \mathbf{W}_0 \right)^2  \nonumber \\ 
    & \qquad \le \frac{16}{15} \left(\frac{5}{4}\right)^2 \mathbb{E}[\bar{a}_0]+ 32\left(\frac{5}{4}\right)^2T^2\mathbb{E}[\norm{\bar{v}_0}^2]+ 32\epsilon^2 T^4\mathbf{W}_0^2 \;. \label{eq:bar_Ea2_moment}
\end{align}
Inserting these estimates into \eqref{eq:expa_moment} and using \eqref{eq:C_eps} and \eqref{eq:cond_t} yields
\begin{align}
    & \mathbb{E} [\bar{a}_t]  \, \le  \, \left(1 - \frac{11}{24} K t^2\right)\mathbb{E}[\bar{a}_0] +  \frac{K t^2 }{6} \frac{1}{5} \left( 2 \left( \frac{5}{4}\right)^2 \frac{441}{400} \mathbb{E}[\bar{a}_0] + 4 \left( \frac{5}{4}\right)^2 \frac{441}{400} T^2 \mathbb{E}[\norm{\bar{v}_0}^2]  + 4 \left(\frac{21}{20}\epsilon T^2 \mathbf{W}_0 \right)^2 \right) \nonumber
    \\ & \qquad \qquad  + \frac{Kt^2}{6} \frac{6}{5}\left(\frac{16}{15} \left(\frac{5}{4}\right)^2 \mathbb{E}[\bar{a}_0]+ 32\left(\frac{5}{4}\right)^2T^2\mathbb{E}[\norm{\bar{v}_0}^2]+ 32\epsilon^2 T^4\mathbf{W}_0^2 \right) + 8t^2 \mathbb{E}[\norm{\bar{v}_0}^2]   \nonumber
\\ & \qquad \qquad + t^2\left( \frac{45\epsilon^2}{2K}+4\epsilon^2 T^2 \right)  \mathbf{W}_0^2 +  t^2\hat{C} \nonumber
\\ & \quad  = \left(1-\left(\frac{11}{24}+\frac{147}{1280} +\frac{1}{3}\right) Kt^2\right) \mathbb{E} [\bar{a}_0] + \left(\left(\frac{147}{640}+10 \right)KT^2+8\right)t^2\mathbb{E}[\norm{\bar{v}_0}^2] \nonumber
\\ & \qquad +\frac{327}{50}Kt^2\epsilon^2T^4  \mathbf{W}_0^2 + t^2\left( \frac{45\epsilon^2}{2K}+4\epsilon^2 T^2 \right)  \mathbf{W}_0^2 +  t^2\hat{C} \nonumber
\\ & \quad \le \left(1-\frac{13}{1280} Kt^2\right) \mathbb{E} [\bar{a}_0] + 11t^2d + t^2\left( \frac{45\epsilon^2}{2K}+6\epsilon^2 T^2 \right)  \mathbf{W}_0^2 +  t^2\hat{C}\;. \label{eq:bound_aT_moment} 
\end{align}

By Gr\"{o}nwall's Inequality, there exists a uniform-in-time second moment bound $\mathbf{B}_1$ for both the nonlinear Hamiltonian dynamics $\mathbb{E}[\norm{\bar{x}_t}^2]$ and the second moment of the measure after arbitrarily many nHMC steps. 
In particular, $\mathbf{B}_1$ depends on $T$, $K$, $L$,  $\epsilon$, $\mathbb{E}[|\bar{x}_0|^2]$, $d$, $\mathbf{W}_0$ and $\mathcal{R}$ and can be chosen of the form given in \eqref{eq:C_1}.
\end{proof}

\section{Global Contractivity of uHMC with Time Integrator Randomization}

In this section, \cite[Theorem 3.1]{BouRabeeSchuh2023} is adapted to the uHMC transition step in \eqref{eq:transstep_meanf}.  The key difference is that the transition step now uses a randomized time integrator in place of the standard Verlet integrator.

\begin{theorem} \label{thm:contr_uHMC}
Suppose \Cref{ass} \ref{ass_Vmin}, \ref{ass_Vlip}, \ref{ass_Vconv} and \ref{ass_Wlip} hold.   Let $\epsilon \in [0, \infty)$ satisfy \begin{align}
    \epsilon \tilde{L} \ \le \  \min\left( \frac{K}{3},\frac{125 K }{624 \sqrt{7+1/(KT^2)}}\exp\left(-\frac{5\tilde{R}}{2T}-5\right) \right) \;. \label{eq:Cepsi}
\end{align} Let $T \in (0,\infty)$   satisfy \begin{align}
( L + 2 \epsilon \tilde{L} ) T^2 & \, \le \,  \min\left(\frac{1}{9},\frac{1}{36^2 L\tilde{R}^2}\right) \;, \label{eq:CT} \end{align}
and $h \ge 0$ satisfy $T/h \in \mathbb{Z}$ if $h>0$. 
Then for all $x, y \in \mathbb{R}^{N d}$ \begin{align}
    \mathbb{E}[ \rho_N( \mathbf{X}_h(x,y) , \mathbf{Y}_h(x,y) ) ] \ &\le \  (1 - c) \rho_N(x,y) \quad  \text{where} \quad  c \  = \ \frac{K T^2}{156} \exp( - \frac{\tilde{R}}{T} ) \;. \label{eq:c_meanf}  
\end{align}
\end{theorem}

\begin{proof}
The proof is somewhat similar to the proof of \Cref{thm:contr_nonlinear} and the proof of \cite[Theorem 3.1]{BouRabeeSchuh2023}, except now the underlying flow is discretized and the underlying potential force is evaluated at a random midpoint. We write $R^i=|\mathbf{X}_h(x,y)^i-\mathbf{Y}_h(x,y)^i|$ and $r^i=|x^i-y^i|$. For $r^i\ge \tilde{R}$, the initial velocities in the $i$th component are synchronized, i.e., $\xi^i=\eta^i$. By \Cref{lem:convexarea_meanf}, \eqref{eq:R1} and concavity of $f$, we obtain
\begin{align}
    \mathbb{E}[f(R^i)-f(r^i)] &\le f(r^i)\mathbb{E}[R^i-r^i] \nonumber
    \\ & \le f'(r^i)\left(-\frac{1}{16}\right)KT^2 r^i+f'(r^i)\epsilon\tilde{L}T^2 \sqrt{7+\frac{1}{KT^2}}\mathbb{E}\left[ \frac{1}{N}\sum_{l=1}^N\sup_{s\le T}|Z_s^l|+h\sup_{s\le T}|W_s^l|\right]. \label{eq:largedist_meanf}
\end{align}
For $r^i<\tilde{R}$, we split
\begin{align} \label{eq:split_meanf}
    \mathbb{E}[f(R^i)-f(r^i)]& =\mathbb{E}[f(R^i)-f(r^i); \{w^i=-\gamma z^i\}]+\mathbb{E}[f(R^i\wedge R_1)-f(r^i); \{w^i\neq-\gamma z^i\}] \nonumber
    \\ & +\mathbb{E}[f(R^i)-f(R^i\wedge R_1); \{w^i\neq -\gamma z^i\}] 
\end{align}
where $z^i=x^i-y^i$ and $w^i=\xi^i-\eta^i$
and bound each term separately. Note that as in \eqref{eq:TV_bound}, it holds
\begin{align}
    \mathbb{P}[w^i\neq -\gamma z^i]=\|\mathcal{N}(0,I_d)-\mathcal{N}(\gamma z^i,I_d)\|_{\mathrm{TV}}< 1/10.
\end{align}
For the first term in \eqref{eq:split_meanf}, we observe by \eqref{eq:gamma}, \eqref{eq:CT} and \eqref{apriori:ZiWidev} 
\begin{align*}
    R^i& \le |z^i-Tw^i| + \frac{9}{5}(L+\epsilon \tilde{L}) T^2 r^i+\frac{9}{5}(L+\epsilon \tilde{L}) T^2 T\gamma r^i 
    \\ & \qquad + \frac{9\epsilon\tilde{L} T^2}{5N}\sum_{l=1}^N\left(|z^l|+2 T|w^l|+\sup_{s\le T}|Z_s^l-z^l-sw^l|+h\sup_{s\le T}|W_s^l-w^l|\right)
    \\ & \le (1-\gamma T) r^i + \frac{2}{5}\gamma T r^i + \frac{9\epsilon\tilde{L} T^2}{5N}\sum_{l=1}^N\left(|z^l|+2 T|w^l| +\sup_{s\le T}|Z_s^l-z^l-sw^l|+h\sup_{s\le T}|W_s^l-w^l|\right)
\end{align*}
where $Z_s^l=Q_s(x,\xi)-Q_s(y,\eta)$ and $W_s^l=P_s(x,\xi)-P_s(y,\eta)$.
By concavity of $f$,
\begin{align}
    & \mathbb{E}[f(R^i)-f(r^i); \{w^i=-\gamma z^i\}]\le -f'(r^i)\frac{9}{10}\frac{3}{5}\gamma Tr^i \nonumber
    \\ &  \qquad + f'(r^i)\frac{9\epsilon\tilde{L} T^2}{5N}\mathbb{E}\left[\sum_{l=1}^N\left(|z^l|+2 T|w^l| +\sup_{s\le T}|Z_s^l-z^l-sw^l|+h\sup_{s\le T}|W_s^l-w^l|\right); \{w^i= -\gamma z^i\}\right]. \nonumber
\end{align}
Analogously to \eqref{eq:boundII}, we obtain
\begin{align*}
    \mathbb{E}[f(R^i\wedge R_1)-f(r^i); \{w^i\neq-\gamma z^i\}]\le T f'(r^i) \mathbb{P}[w^i\neq -\gamma z^i]\le f'(r^i)(2/5)\gamma T r^i.
\end{align*}
For $\mathbb{E}[f(R^i)-f(R^i\wedge R_1); \{w^i\neq -\gamma z^i\}]$, we observe by \eqref{apriori:ZiWidev} and since $w^i=2(e^i\cdot\xi^i)e^i$
\begin{align*}
    R^i& \le (1+\frac{9}{5}(L+\epsilon \tilde{L})T^2)|z^i|+(1+\frac{18}{5}(L+\epsilon \tilde{L})T^2)2|e^i\cdot\xi^i|
    \\ & \quad  + \frac{9\epsilon\tilde{L} T^2}{5N}\sum_{l=1}^N\left(|z^l|+2 T|w^l| +\sup_{s\le T}|Z_s^l-z^l-sw^l|+h\sup_{s\le T}|W_s^l-w^l|\right)
\end{align*}
and hence by \eqref{eq:CT} and concavity of $f$
\begin{align*}
    &\mathbb{E}[(f(R^i)- f(R_1))^+; \{w^i\neq -\gamma z^i\}] \le f'(R_1)\mathbb{E}\left[\left(\frac{6}{5}r^i+ \frac{7}{5}T|\xi^i\cdot e^i| -R_1 \right)^+ ;\{w^i\neq -\gamma z^i\}\right]
     \\ & \quad  + f'(R_1) \frac{9\epsilon\tilde{L} T^2}{5N}\mathbb{E}\left[\sum_{l=1}^N\left(|z^l|+2 T|w^l| +\sup_{s\le T}|Z_s^l-z^l-sw^l|+h\sup_{s\le T}|W_s^l-w^l|\right); \{w^i\neq -\gamma z^i\}\right].
\end{align*}
For the third term in \eqref{eq:split_meanf}, it holds similarly to \eqref{eq:boundIII}
\begin{align}
    & \mathbb{E}\left[\left(\frac{6}{5}r^i+ \frac{7}{5}T|\xi^i\cdot e^i| -R_1 \right)^+ ;\{w^i\neq -\gamma z^i\}\right]  = \int_{-\gamma r^i/2}^{\gamma r^i/2} \left(\frac{6}{5}r^i+ \frac{7}{5}T|u| -R_1 \right)^+ \frac{e^{-\frac{u^2}{2}}}{\sqrt{2\pi}}\rmd u \nonumber
    \\ & \qquad +  \int_{\gamma r^i/2}^{\infty} \left(\left(\frac{6}{5}r^i+ \frac{7}{5}T|u| -R_1 \right)^+- \left(\frac{6}{5}r^i+ \frac{7}{5}T|u-\gamma r^i| -R_1 \right)^+ \right)\frac{e^{-\frac{u^2}{2}}}{\sqrt{2\pi}}\rmd u \nonumber
    \\ & \quad \le \int_{\gamma r^i/2}^{\infty} \frac{7}{5} T\gamma r^i \frac{e^{-\frac{u^2}{2}}}{\sqrt{2\pi}}\rmd u \le \frac{7}{10}T\gamma r^i \nonumber
\end{align}
where \eqref{eq:R1} is used in the last step. 
Hence, we obtain 
\begin{align*}
    &\mathbb{E}[(f(R^i)- f(R_1))^+; \{w^i\neq -\gamma z^i\}] \le f'(R_1) \frac{7}{10} T\gamma r^i
     \\ & \quad  + f'(R_1) \frac{9\epsilon\tilde{L} T^2}{5N}\mathbb{E}\left[\sum_{l=1}^N\left(|z^l|+2 T|w^l| +\sup_{s\le T}|Z_s^l-z^l-sw^l|+h\sup_{s\le T}|W_s^l-w^l|\right); \{w^i\neq -\gamma z^i\}\right].
\end{align*}
Since $f'(R_1)/f'(\tilde{R})\le 1/12$, it holds
\begin{align*}
    &\mathbb{E}[(f(R^i)- f(R_1))^+; \{w^i\neq -\gamma z^i\}] \le f'(r^i) \frac{7}{120} T\gamma r^i
     \\ & \quad  + f'(r^i) \frac{3\epsilon\tilde{L} T^2}{20N}\mathbb{E}\left[\sum_{l=1}^N\left(|z^l|+2 T|w^l| +\sup_{s\le T}|Z_s^l-z^l-sw^l|+h\sup_{s\le T}|W_s^l-w^l|\right); \{w^i\neq -\gamma z^i\}\right].
\end{align*}
Combining the estimates for the three terms in \eqref{eq:split_meanf}, we obtain
\begin{align}
    & \mathbb{E}[f(R^i)-f(r^i)]\le -f'(r^i)\frac{9}{10}\frac{3}{5}\gamma Tr^i + f'(r^i)\frac{2}{5}\gamma T r^i+ f'(r^i) \frac{7}{120} T\gamma r^i \nonumber
    \\ &  \qquad + f'(r^i)\frac{9\epsilon\tilde{L} T^2}{5N}\mathbb{E}\left[\sum_{l=1}^N\left(|z^l|+2 T|w^l| +\sup_{s\le T}|Z_s^l-z^l-sw^l|+h\sup_{s\le T}|W_s^l-w^l|\right)\right] \nonumber
    \\ & \quad \le -f(r^i) \frac{49}{600} \gamma T r^i + \frac{9\epsilon\tilde{L} T^2}{5N}\mathbb{E}\left[\sum_{l=1}^N\left(|z^l|+2 T|w^l| +\sup_{s\le T}|Z_s^l-z^l-sw^l|+h\sup_{s\le T}|W_s^l-w^l|\right)\right] . \label{eq:smalldist_meanf}
\end{align}
We observe that by \eqref{eq:CT} and \eqref{apriori:ZiWidev}, the last term in \eqref{eq:largedist_meanf} and \eqref{eq:smalldist_meanf} are both bounded by 
\begin{align*}
& \mathbb{E}\left[ \frac{1}{N}\sum_{l=1}^N\sup_{s\le T}|Z_s^l|+h\sup_{s\le T}|W_s^l|\right] 
\\ & \qquad \le \mathbb{E}\left[\sum_{l=1}^N\left(|z^l|+2 T|w^l|) +\sup_{s\le T}|Z_s^l-z^l-sw^l|+h\sup_{s\le T}|W_s^l-w^l|\right)\right] \le \frac{6}{5}\mathbb{E}\left[\sum_{l=1}^N (|z^l|+2T|w^l|)\right].
\end{align*}
Using this estimate and combining \eqref{eq:largedist_meanf} and \eqref{eq:smalldist_meanf},
\begin{align}
    &\mathbb{E}[f(R^i)-f(r^i)]\le -f'(r^i) \min\left(\frac{1}{16}KT^2,\frac{49}{600} \gamma T\right) r^i \nonumber
    \\ & \qquad + \epsilon\tilde{L}T^2\max\left( \sqrt{7+\frac{1}{KT^2}} ,3\right)\frac{6}{5N}\mathbb{E}\left[\sum_{l=1}^N (|z^l|+2T|w^l|)\right]. \nonumber
\end{align}
By \eqref{eq:gamma} and \eqref{eq:CT}, the minimum is attained at $\frac{1}{16}KT^2$ and the maximum at $\sqrt{7.5+\frac{1}{KT^2}}$. 
We observe
\begin{align*}
    &\mathbb{E}[\norm{w^l}]=\mathbb{E}[\norm{w^l};\{r^l<\tilde{R}\}]= \1_{r^l<\tilde{R}}(|z^l|\gamma \mathbb{P}[w^l=-\gamma z^l]+\mathbb{E}[|w^l|; w^l\neq -\gamma z^l])
    \\ & \qquad \le \norm{z^l}\gamma + \1_{r^l<\tilde{R}}\int_{-\gamma r^l/2}^{\infty} \frac{2|u|}{\sqrt{2\pi}}\left(\exp\left(-\frac{u^2}{2}\right)-\exp\left(-\frac{(u+\gamma r^l)^2}{2}\right) \right) 
    \rmd u
    \\ & \qquad \le \norm{z^l}\gamma + \1_{r^l<\tilde{R}}\left(\int_{-\gamma r^l/2}^{\gamma r^l/2} \frac{2|u|}{\sqrt{2\pi}}\exp\left(-\frac{u^2}{2}\right)\rmd u +\int_{\gamma r^l/2}^{\infty} \frac{2\gamma r^l}{\sqrt{2\pi}}\exp\left(-\frac{u^2}{2}\right) 
    \rmd u \right)
    \\ & \qquad \le \norm{z^l}\gamma + \1_{r^l<\tilde{R}}\left( \frac{(\gamma r^l)^2}{\sqrt{2\pi}} + \gamma r^l\right) \overset{\eqref{eq:gamma}}{\le}  \frac{21}{10}r^l\gamma.
\end{align*}
Then, by \eqref{eq:estimate_f}, it holds
\begin{align}
   &\mathbb{E}[f(R^i)-f(r^i)]\le -\frac{5}{4}\left(\tilde{R}/T+2\right)\exp\left(-\frac{5\tilde{R}}{4T}-\frac{5}{2}\right) \frac{1}{16}KT^2 f(r^i) \nonumber
    \\ & \qquad + \epsilon\tilde{L}T^2 \sqrt{7+\frac{1}{KT^2}} \exp\left(\frac{5\tilde{R}}{4T}+\frac{5}{2}\right) \frac{6}{5}\frac{26}{5}\mathbb{E}\left[\sum_{l=1}^N f(r^l)\right]. \nonumber
\end{align}
Taking the sum over all particles, we obtain for $\epsilon$ satisfying \eqref{eq:C_eps}
\begin{align}
    \mathbb{E}\left[\frac{1}{N}\sum_{i=1}^N f(R^i)\right]\le (1-c)\mathbb{E}\left[\frac{1}{N}\sum_{i=1}^N f(r^i)\right] \nonumber
\end{align}
for $c$ given in \eqref{eq:c_meanf}.
\end{proof}

\subsection{Preliminaries}

From \eqref{eq:U_meanf} and \Cref{ass}, note that for all $x, y \in \mathbb{R}^{N d}$ and $i \in \{1, \dots, N \}$, \begin{align}
\norm{\nabla_i U(x) - \nabla_i U(y)}  \, &\overset{\ref{ass_Vlip},\ref{ass_Wlip}}{\le} \,  (L + \epsilon \tilde{L}) \norm{x^i - y^i} + \frac{\epsilon \tilde{L} }{N} \sum_{\ell = 1}^N  \norm{x^{\ell} - y^{\ell}}  , \label{eq:L_i_mf} \\
\norm{\nabla_i U(x)} \  &\overset{\eqref{eq:L_i_mf},\ref{ass_Vmin}}{\le} \  (L + \epsilon \tilde{L}) \norm{x^i} + \frac{\epsilon \tilde{L} }{N} \sum_{\ell = 1}^N  \norm{x^{\ell}} + \epsilon \mathbf{W}_0  ,  \label{eq:G_i_mf} \\
\langle x^1 - y^1, \nabla V(x^1) - \nabla V(y^1) \rangle  \, &\overset{\ref{ass_Vconv},\ref{ass_Vlip}}{\ge} \,  K \norm{x^1 - y^1}^2 + \frac{1}{L} \norm{\nabla V(x^1) - \nabla V(y^1)}^2 - \hat{C} \;.
\label{eq:K_i}
\end{align}

\subsection{Deviation from Free Flow}

Here we develop estimates on the numerical flow $(Q_t(x,v), P_t(x,v))$ that solves \eqref{eq:hamdyn_meanf_num}.  

\begin{lemma} \label{lem:aprioriQV}
Suppose \Cref{ass} \ref{ass_Vmin}, \ref{ass_Vlip},  and \ref{ass_Wlip} hold. 
Let $\epsilon \in [0, \infty)$ and $T \in (0,\infty)$   satisfy \eqref{eq:Cepsi} and \eqref{eq:CT}.  Let $L_e = L+2 \epsilon \tilde{L}$.  Let $h \ge 0$ satisfy $T/h \in \mathbb{Z}$ if $h>0$. 
For any $x,v\in \mathbb{R}^{N d}$ and $i \in \{1, \dots, N\}$,
	\begin{align}
	\begin{split} \label{apriori:QiVidev}
& \sup_{s \le T} \norm{Q_s^i(x,v) - x^i - s v^i}
+ L_e^{-\frac{1}{2}}\sup_{s \le T} \norm{P_s^i(x,v) - v^i} \, \le \,  3 (L + \epsilon \tilde{L} ) T L_e^{-\frac{1}{2}} \left( \norm{x^i} + 2 T \norm{v^i} \right) +  L_e^{-1} \epsilon \mathbf{W}_0  \\ 
&  + \frac{3 \epsilon \tilde{L} T}{N} L_e^{-\frac{1}{2}}\sum_{\ell=1}^N \big( \norm{x^{\ell}} + 2 T \norm{v^{\ell}} +  \sup_{s \le T} \norm{Q_s^{\ell}(x,v) - x^{\ell} - s v^{\ell}}
+ L_e^{-\frac{1}{2}} \sup_{s \le T} \norm{P_s^{\ell}(x,v) - v^{\ell}} \big) 
\end{split} \\
  \begin{split} 
&  \sum_{\ell=1}^N  \big( \sup_{s \le T} \norm{Q_s^{\ell}(x,v) - x^{\ell} - s v^{\ell}} + L_e^{-\frac{1}{2}}\sup_{s \le T} \norm{P_s^{\ell}(x,v)-v^{\ell}} \big) 
\\
& \qquad \le   N  L_e^{-1}   \epsilon  \mathbf{W}_0  +  6 L_e^{\frac{1}{2}} T \sum_{\ell=1}^N \big( \norm{x^{\ell}} + L_e^{-\frac{1}{2}} \norm{v^{\ell}} \big). \label{apriori:QVdev} 
 	\end{split} 
 	\end{align}

\end{lemma}

\begin{proof}
Note that \eqref{apriori:QVdev} follows from summing  \eqref{apriori:QiVidev} over $i$ from $1$ to $N$ and invoking \eqref{eq:CT}.   Therefore, it suffices to prove \eqref{apriori:QiVidev}.  Let $t \in [0,T]$ and $Q^{\star}_t := Q_{\lfloor t \rfloor }  + h \mathcal{U}_{\lfloor t \rfloor/h} P_{\lfloor t \rfloor }$.  From \eqref{eq:hamdyn_meanf_num}, note that \begin{align}
& \norm{Q_t^i(x,v) - x^i - t v^i}
+ L_e^{-1/2}  \norm{P_t^i(x,v) - v^i} \ = \ \norm{ \int_0^t \int_0^s \nabla_i U(Q_r^{\star}) dr ds } + L_e^{-1/2} \norm{ \int_0^t \nabla_i U(Q_s^{\star}) ds }  \nonumber \\
& ~  \le  \int_0^t \int_0^s \norm{\nabla_i U(Q_r^{\star})} dr ds  + L_e^{-1/2}  \int_0^t \norm{\nabla_i U(Q_s^{\star})} ds \nonumber \\
& ~ \overset{\eqref{eq:G_i_mf}}{\le} (L+\epsilon \tilde{L} ) \left( \int_0^t \int_0^s \norm{Q_r^{\star,i}} dr ds + L_e^{-1/2} \int_0^t \norm{Q_s^{\star,i}} ds   \right)  + (\frac{T^2}{2} + L_e^{-1/2} T)  \epsilon  \mathbf{W}_0 \nonumber \\
& ~ \quad + \frac{\epsilon \tilde{L}}{N} \sum_{\ell=1}^N \left( \int_0^t \int_0^s \norm{Q_r^{\star,\ell}} dr ds +L_e^{-1/2} \int_0^t \norm{Q_s^{\star,\ell}} ds   \right) \nonumber \\
& ~ \le 
(L+\epsilon \tilde{L} ) \left( \int_0^t \int_0^s ( \norm{Q^i_{\lfloor r \rfloor}} + h \norm{P^i_{\lfloor r \rfloor}} ) dr ds +  L_e^{-1/2}\int_0^t ( \norm{Q^i_{\lfloor s \rfloor}} + h \norm{P^i_{\lfloor s \rfloor}} ) ds   \right)  + (\frac{T^2}{2} + L_e^{-1/2} T) \mathbf{W}_0  \nonumber \\
& ~ \quad + \frac{\epsilon \tilde{L}}{N} \sum_{\ell=1}^N \left( \int_0^t \int_0^s ( \norm{Q_{\lfloor r \rfloor}^{\ell}} +  h \norm{P_{\lfloor r \rfloor}^{\ell}} ) dr ds +  L_e^{-1/2} \int_0^t (\norm{Q_{\lfloor s \rfloor}^{\ell}} + h \norm{P_{\lfloor s \rfloor}^{\ell}}) ds   \right) \nonumber \\
& ~ \le \frac{3}{2} (L+\epsilon \tilde{L})  T L_e^{-1/2} \left(  \sup_{s \le T} \norm{Q_s^{i}(x,v) - x^{i} - s v^{i}}
+ L_e^{-1/2}  \sup_{s \le T} \norm{P_s^{i}(x,v) - v^{i}}  + \norm{x^i} + 2 T \norm{v^i} \right)       \nonumber \\
& ~ \qquad + \frac{3}{2} \frac{\epsilon \tilde{L}  T}{N} L_e^{-1/2}   \sum_{\ell=1}^N \left( \sup_{s \le T} \norm{Q_s^{\ell}(x,v) - x^{\ell} - s v^{\ell}}
+ L_e^{-1/2} \sup_{s \le T} \norm{P_s^{\ell}(x,v) - v^{\ell}} \right)  \nonumber \\
& ~ \qquad +  \frac{3}{2} \frac{\epsilon \tilde{L} T}{N} L_e^{-1/2} \sum_{\ell=1}^N \left( \norm{x^{\ell}} + 2 T \norm{v^{\ell}} \right) + (\frac{T^2}{2} + L_e^{-1/2} T)  \epsilon \mathbf{W}_0 \;. 
\nonumber 
\end{align}
Taking the supremum of the left-hand-side, invoking  \eqref{eq:CT}, and simplifying yields \eqref{apriori:QiVidev}. \end{proof}

To state the next lemma, it helps to introduce the shorthand notation: \begin{equation} \label{eq:zwZtWt}
z \ := \ x-y, \quad w \ := \ v-u, \quad Z_t \ := \ Q_t(x,v)-Q_t(y,u), \quad \text{and} \quad W_t \ := \ P_t(x,v)-P_t(y,u) \;.
\end{equation}

\begin{lemma} \label{lem:aprioriZW}
Suppose \Cref{ass} \ref{ass_Vmin}, \ref{ass_Vlip},  and \ref{ass_Wlip} hold. 
Let $\epsilon \in [0, \infty)$ and $T \in (0,\infty)$   satisfy \eqref{eq:Cepsi} and \eqref{eq:CT}. Let $h \ge 0$ satisfy $T/h \in \mathbb{Z}$ if $h>0$. 
For any $x,y,u,v\in \mathbb{R}^{N d}$, \begin{align}
	\begin{split} \label{apriori:ZiWidev}
& \sup_{s \le T} \norm{Z_s^i - z^i - s w^i}
+ h \sup_{s \le T} \norm{W_s^i - w^i} \, \le \, \frac{9}{5} (L+ \epsilon \tilde{L} ) T^2 ( \max(\norm{z^i},\norm{z^i+Tw^i}) + T \norm{w^i} ) \\
& \qquad + \frac{9 \epsilon \tilde{L} T^2}{5N}  \sum_{\ell=1}^N \left( \norm{z^{\ell}} + 2 T \norm{w^{\ell}}+  \sup_{s \le T} \norm{Z_s^{\ell} - z^{\ell} - s w^{\ell}}
+ h \sup_{s \le T} \norm{W_s^{\ell} - w^{\ell}} \right) \;,
\\ & \le \, \frac{9}{5} (L+ \epsilon \tilde{L} ) T^2 ( \norm{z^i} + 2 T \norm{w^i} )  + \frac{9 \epsilon \tilde{L} T^2}{5N}  \sum_{\ell=1}^N \left( \norm{z^{\ell}} + 2 T \norm{w^{\ell}} + \sup_{s \le T} \norm{Z_s^{\ell} - z^{\ell} - s w^{\ell}}
+ h \sup_{s \le T} \norm{W_s^{\ell} - w^{\ell}} \right) \;,
\end{split} \\
  \begin{split} \label{apriori:ZWdev} 
& \sum_{\ell=1}^N  \left( \sup_{s \le T} \norm{Z_s^{\ell} - z^{\ell} - s w^{\ell}} + h \sup_{s \le T} \norm{W_s^{\ell}-w^{\ell}} \right) 
\, \le \, 	\frac{9}{5} (L+ 2 \epsilon \tilde{L}) T^2 \sum_{\ell=1}^N \left( \norm{z^{\ell}} + 2 T \norm{w^{\ell}} \right) \;. 
 	\end{split} 
 	\end{align}

\end{lemma}

\begin{proof}
The proof of Lemma~\ref{lem:aprioriZW} is similar to the proof of Lemma~\ref{lem:aprioriQV} except with \eqref{eq:L_i_mf} taking the place of \eqref{eq:G_i_mf}.
\end{proof}

\subsection{Asymptotic Contractivity}

The following Lemma shows that the asymptotic strong co-coercivity assumption on $\nabla V$ in \ref{ass_Vconv} implies that the numerical flow of the $N$-particle system $(Q_t(x,v), P_t(x,v))$ that solves \eqref{eq:hamdyn_meanf_num}  is asymptotically contractive in the $i$-th position component if the $i$-th initial velocities are synchronized.  

\begin{lemma} \label{lem:convexarea_meanf}
Suppose \Cref{ass} \ref{ass_Vmin}, \ref{ass_Vlip}, \ref{ass_Vconv} and \ref{ass_Wlip} hold.    Let $\epsilon \in [0, \infty)$ and $T \in (0,\infty)$   satisfy \eqref{eq:Cepsi} and \eqref{eq:CT}. Let $h \ge 0$ satisfy $T/h \in \mathbb{Z}$ if $h>0$.  Then for any $t \in [0,T]$, $i \in \{1, \cdots, N \}$ and for all $x, y, u, v \in \mathbb{R}^{N d}$ such that $v^i = u^i$, it holds a.s. \begin{equation}
   \begin{aligned}
    & | Q_t^i(x,v) - Q_t^i(y,u) |^2 \le  ( 1- K T^2 / 3) |x^i - y^i|^2  \\
    & \qquad+ \frac{5}{4} T^2 \hat{C} + \epsilon^2 \tilde{L}^2   T^4 \left(7 + \frac{1}{K T^2} \right)   \left( \frac{1}{N} \sum_{\ell=1}^N \left( \sup_{s \le T} \norm{Z_s^{\ell}} + h \sup_{s \le T} \norm{W_s^{\ell}} \right)  \right)^2 \;. 
    \end{aligned}
\end{equation}
\end{lemma}

\begin{proof}
This proof carefully generalizes the proof of \cite[Lemma 5]{BouRabeeMarsden2022} to an $N$-particle system where $\nabla V$ is only asymptotically strongly co-coercive.
To this end, it is notationally convenient to define \begin{align*}
(X_t, V_t) &:= (Q_t(x,v), P_t(x,v)),  \quad  
(Y_t, U_t) := (Q_t(y,u), P_t(y,u)), \quad
Z_t := X_t - Y_t, \quad W_t := V_t - U_t, \\
X_t^{\star} &:= X_{\lfloor t \rfloor} + h \mathcal{U}_{\lfloor t \rfloor/h} V_{\lfloor t \rfloor},  \quad Y_t^{\star} := Y_{\lfloor t \rfloor} + h \mathcal{U}_{\lfloor t \rfloor/h} U_{\lfloor t \rfloor},  \\
\beta_t^{\star,i} &:= \nabla V(X_t^{\star,i}) - \nabla V(Y_t^{\star,i}),  \quad \Omega_t^{\star,i} := \frac{1}{N} \sum_{\ell=1}^N \left( \nabla_1 W(X_t^{\star,i}, X_t^{\star,\ell}) - \nabla_1 W(Y_t^{\star,i}, Y_t^{\star,\ell} ) \right) \;.
\end{align*}
Therefore, $\nabla_i U(X_t^{\star}) - \nabla_i U(Y_t^{\star}) = \beta_t^{\star,i} + \epsilon \Omega_t^{\star,i}$.  
Furthermore, let $A_t^i := \norm{Z_t^i}^2$, $B_t^i := 2 \langle Z_t^i, W_t^i \rangle$, and $A_t^{\star,i} := \norm{Z_t^i}^2$.  Note that $\rho_t^{\star,i} := \langle Z_t^{\star,i}, \beta_t^{\star,i} \rangle$ satisfies \begin{equation} \label{ieq:rhoit}
K A_t^{\star,i} + \frac{1}{L} \norm{\beta_t^{\star,i}}^2 - \hat{C} \overset{\eqref{eq:K_i}}{\le}  \rho_t^{\star,i} \overset{\ref{ass_Vlip}}{\le}  L A_t^{\star,i} \;.
\end{equation}
From \eqref{eq:hamdyn_meanf_num}, it follows that \begin{align} 
		\frac{d}{dt} Z_t^i \, &= W_t^i ,
		\qquad \frac{d}{dt} W_t^i \, = \, -  \beta_t^{\star,i} - \epsilon \Omega_t^{\star,i} \;, \quad \text{and hence,  } \label{eq:dotz_dotw} \\
        \frac{d}{dt} A_t^i &= B_t^i, \quad \frac{d}{dt} B_t^i = 2 |W_t^i|^2 - 2 \langle Z_t^i, \beta_t^{\star,i} \rangle - 2 \epsilon \langle Z_t^i, \Omega_t^{\star,i} \rangle \;. 
    \end{align}
    Let $c_t:=\cos(\sqrt{ K} t )$ and $s_t := \sin( \sqrt{ K} t )/\sqrt{ K} $. 
By variation of parameters, \begin{equation} \label{vop_at} 
    A_T^i = c_T A_0^i + \int_0^T s_{T-r} \left( 2 \norm{W_r^i}^2 - 2 \rho_r^{\star,i} + \epsilon_r^i \right) dr
\end{equation} 
where $\epsilon_t^i := K A_t^i - 2 \langle Z_t^i - Z_t^{\star,i}, \beta_t^{\star, i} \rangle - 2 \epsilon \langle Z_t^i, \Omega_t^{\star,i} \rangle$. 
To upper bound  $\epsilon_r^i$  in \eqref{vop_at},   \begin{align}    & \epsilon_t^i  \ \le \ -2 \langle (t- \lfloor t \rfloor - h \mathcal{U}_{\lfloor t \rfloor/h}) W_{\lfloor t \rfloor}^i - \frac{1}{2} (t - \lfloor t \rfloor)^2 (\beta_t^{\star,i} + \epsilon \Omega_{t}^{\star,i}), \beta_t^{\star,i} \rangle \nonumber + \frac{\epsilon^2}{K} \norm{\Omega_t^{\star,i}}^2 + 2 K A_t^i  \nonumber \\
& \quad \le \norm{W_{\lfloor t \rfloor}^i}^2 + \frac{5}{2} h^2 \norm{\beta_t^{\star,i}}^2 + \epsilon^2 \left( \frac{h^2}{2} + \frac{1}{K} \right) \norm{\Omega_t^{\star,i}}^2  + 2 K A_t^i \;. \label{ieq:epsi} 
\end{align}
To upper bound the term $A_t^i$ in \eqref{ieq:epsi},   \begin{align}    & A_t^i  \ = \ \norm{ Z_t^{\star,i} + (t- \lfloor t \rfloor - h \mathcal{U}_{\lfloor t \rfloor/h}) W_{\lfloor t \rfloor}^i - \frac{1}{2} (t - \lfloor t \rfloor)^2 (\beta_t^{\star,i} + \epsilon \Omega_{t}^{\star,i}) }^2 \nonumber \\
& \quad \le A_t^{\star,i} + 2 h | \langle Z_t^{\star,i}, W_{\lfloor t \rfloor}^i \rangle | - h^2 \rho_t^{\star,i} + \epsilon h^2 | \langle Z_t^{\star,i}, \Omega_{t}^{\star,i} \rangle | \nonumber \\
& \qquad + h^3 | \langle  W_{\lfloor t \rfloor}^i, \beta_t^{\star,i} \rangle | + \epsilon h^3 | \langle  W_{\lfloor t \rfloor}^i , \Omega_t^{\star,i} \rangle | + \frac{h^4}{2} \left(\norm{\beta_t^{\star,i}}^2  + \epsilon^2  \norm{\Omega_t^{\star,i}}^2   \right) \nonumber \\
& \quad \le - h^2 \rho_t^{\star,i} + \left(  1+ \frac{3}{2} K h^2 \right)  A_t^{\star,i} + \left( \frac{1}{K} + h^2 \right) \norm{ W_{\lfloor t \rfloor}^i}^2 + \epsilon^2 \left( \frac{h^2}{2 K} + h^4 \right) \norm{\Omega_t^{\star,i}}^2  + h^4 \norm{\beta_t^{\star,i}}^2 
\label{ieq:Ait}
\end{align}
To upper bound the terms involving $|W_t^i|^2$ in \eqref{vop_at},  use \eqref{eq:dotz_dotw} and note that $W_0^i=0$, since the $i$-th initial velocities in the two copies are synchronized.  Therefore, \begin{align} \label{eq:Wit2}
 |W_t^i|^2 \, = \,   \norm{  \int_0^{t}   \beta_s^{\star,i} + \epsilon \Omega_s^{\star,i}  ds }^2 \ \, \le \,  t  \int_0^{t} \norm{\beta_s^{\star,i} + \epsilon \Omega_s^{\star,i}  }^2 ds
		\, \le \,  2 t  \int_0^{t} \left(\norm{\beta_s^{\star,i}}^2 + \epsilon^2 \norm{ \Omega_s^{\star,i}  }^2\right) ds 
\end{align}
where we used the Cauchy-Schwarz inequality.  By \eqref{eq:CT}, note that $s_{t-r}$ is monotonically decreasing with $r$ and satisfies \eqref{eq:sinus}. 
Therefore, combining \eqref{eq:Wit2} and \eqref{eq:sinus}, and by Fubini's Theorem,   \begin{align}
& \int_0^t s_{t-r} |W_r^i|^2 dr  \, \overset{\eqref{eq:Wit2}}{\le} \,
 2 \int_0^t \int_0^r  r s_{t-r} \left(\norm{\beta_s^{\star,i}}^2 + \epsilon^2 \norm{ \Omega_s^{\star,i}  }^2\right)  ds dr
 \overset{\eqref{eq:sinus}}{\le} 
 2 \int_0^t \int_0^r  r s_{t-s}  \left(\norm{\beta_s^{\star,i}}^2 + \epsilon^2 \norm{ \Omega_s^{\star,i}  }^2\right) ds dr
\nonumber \\
& \qquad \ \le \,  2 \int_0^t \int_s^t  r s_{t-s}  \left(\norm{\beta_s^{\star,i}}^2 + \epsilon^2 \norm{ \Omega_s^{\star,i}  }^2\right)  dr ds  \, \le \,  t^2  \int_0^t   s_{t-s} \left(\norm{\beta_s^{\star,i}}^2 + \epsilon^2 \norm{ \Omega_s^{\star,i}  }^2\right)  ds \;. \nonumber
\end{align} In fact, as a byproduct of this calculation, observe that \begin{equation}
(\int_0^t s_{t-r} \norm{W^i_{\lfloor r \rfloor}}^2 dr ) \vee ( \int_0^t s_{t-r} \norm{W^i_r}^2 dr) \, \le \,  t^2 \int_0^t   s_{t-s} \left(\norm{\beta_s^{\star,i}}^2 + \epsilon^2 \norm{ \Omega_s^{\star,i}  }^2\right)   ds \;.  \label{ieq:Wit2}
\end{equation}
To upper bound the mean-field interaction force,   \begin{align}  & \norm{\Omega_t^{\star,i}}^2    \ = \ \norm{ \frac{1}{N} \sum_{\ell=1}^N \nabla_1 W(X^{\star,i}_t,X^{\star,\ell}_t) - \nabla_1 W(Y^{\star,i}_t,Y^{\star,\ell}_t) }^2  \nonumber \\
& \quad \le \left( \frac{1}{N} \sum_{\ell=1}^N \norm{ \nabla_1 W(X^{\star,i}_t,X^{\star,\ell}_t) - \nabla_1 W(Y^{\star,i}_t,Y^{\star,\ell}_t)} \right)^2     \nonumber  \\
& \quad \overset{\ref{ass_Wlip}}{\le} \left( \frac{\tilde{L}}{N} \sum_{\ell=1}^N ( \norm{ Z_t^{\star,i}} + \norm{Z_t^{\star,\ell}} ) \right)^2   =  \left( \tilde{L} \norm{Z_t^{\star,i}} + \frac{\tilde{L}}{N} \sum_{\ell=1}^N \norm{Z_t^{\star,\ell}} \right)^2   \nonumber \\
& \quad \le 2 \tilde{L}^2 \norm{Z_t^{\star,i}}^2 + 2 \tilde{L}^2 \left( \frac{1}{N} \sum_{\ell=1}^N \sup_{s \le T} \norm{Z_s^{\star,\ell}} \right)^2 \le 2 \tilde{L}^2 A_t^{\star,i} + 2 \tilde{L}^2 \left( \frac{1}{N} \sum_{\ell=1}^N \left( \sup_{s \le T} \norm{Z_s^{\ell}} + h \sup_{s \le T} \norm{W_s^{\ell}} \right)  \right)^2
\label{ieq:Omit}
\end{align}

 Inserting these bounds into the second term of \eqref{vop_at} and simplifying yields 
\begin{align}
&  \int_0^T s_{T-r} \left( 2 \norm{W_r^i}^2 - 2 \rho_r^{\star,i} + \epsilon_r^i \right) dr  \nonumber  \\
&  \quad \overset{\eqref{ieq:epsi}}{\le}  \int_0^T s_{T-r} \left( 2 \norm{W_r^i}^2 +   \norm{ W_{\lfloor r \rfloor}^i}^2 - 2 \rho_r^{\star,i} + 2 K A_r^{i} + \frac{5}{2} h^2 \norm{\beta_r^{\star,i}}^2 + \epsilon^2 \left( \frac{h^2}{2} + \frac{1}{K} \right) \norm{\Omega_r^{\star,i}}^2 \right) dr \nonumber \\
&  \quad \overset{\eqref{ieq:Ait}}{\le}  \int_0^T s_{T-r} \left( 2 \norm{W_r^i}^2 +  (3 + 2 K h^2 )  \norm{ W_{\lfloor r \rfloor}^i}^2 - 2 (1+K h^2) \rho_r^{\star,i} + K (1+3 K h^2) A_r^{\star,i} \right. \nonumber \\ 
&\qquad \qquad \qquad \qquad \left. + \left( \frac{5}{2} h^2 + 2 K h^4 \right) \norm{\beta_r^{\star,i}}^2 + \epsilon^2 \left(\frac{3}{2} h^2 + 2 K h^4 + \frac{1}{K} \right) \norm{\Omega_r^{\star,i}}^2 \right)  dr \nonumber \\
& \quad \overset{\eqref{ieq:Wit2}}{\le} \int_0^T s_{T-r} \left(   - 2 (1+K h^2) \rho_r^{\star,i} + K (1+3 K h^2) A_r^{\star,i} + \left(  (5 + 2 K h^2 ) T^2 +  \frac{5}{2} h^2 + 2 K h^4 \right) \norm{\beta_r^{\star,i}}^2  \right. \nonumber \\ 
&\qquad \qquad \qquad \qquad \left. +  \epsilon^2 \left((5 + 2 K h^2 ) T^2 + \frac{3}{2} h^2 + 2 K h^4 + \frac{1}{K} \right) \norm{\Omega_r^{\star,i}}^2 \right)  dr \nonumber \\
& \quad \overset{\eqref{ieq:rhoit},\eqref{eq:CT}}{\le}  \int_0^T s_{T-r} \left(    2 (1+K h^2) \hat{C} + \left( - 2 K (1+K h^2) +   K (1+3 K h^2) \right) A_r^{\star,i}  \right. \nonumber \\ 
&\qquad \qquad \qquad \qquad 
+ \frac{1}{L} \left( - 2 (1+K h^2) + (5 + 2 K h^2 ) L T^2 +  \frac{5}{2} L h^2 + 2 K L h^4 \right) \norm{\beta_r^{\star,i}}^2 \nonumber \\
&\qquad \qquad \qquad \qquad  \left. +  \epsilon^2 T^2 \left(7 + \frac{1}{K T^2} \right) \norm{\Omega_r^{\star,i}}^2 \right)  dr \nonumber \\
& \quad \overset{\eqref{ieq:Omit}}{\le}   \int_0^T s_{T-r} \Big(    2 (1+K h^2) \hat{C}  + 2 \epsilon^2 \tilde{L}^2  T^2 \left(7 + \frac{1}{K T^2} \right)   \left( \frac{1}{N} \sum_{\ell=1}^N \left( \sup_{s \le T} \norm{Z_s^{\ell}} + h \sup_{s \le T} \norm{W_s^{\ell}} \right)  \right)^2 \nonumber\\
&\qquad \qquad \qquad \qquad 
 + K \left( - (1-K h^2)  +  \frac{1}{K} 2 \epsilon^2 \tilde{L}^2  T^2 \left(7 + \frac{1}{K T^2} \right)   \right) A_r^{\star,i}   \nonumber \\ 
&\qquad \qquad \qquad \qquad 
  + \frac{1}{L} \left( - 2 (1+K h^2) + (5 + 2 K h^2 ) L T^2 +  \frac{5}{2} L h^2 + 2 K L h^4 \right) \norm{\beta_r^{\star,i}}^2 \Big)  dr \nonumber \\
& \quad \overset{\eqref{eq:CT}, \eqref{eq:Cepsi}}{\le}   \frac{5}{4} T^2 \hat{C} + \epsilon^2 \tilde{L}^2   T^4 \left(7 + \frac{1}{K T^2} \right)   \left( \frac{1}{N} \sum_{\ell=1}^N \left( \sup_{s \le T} \norm{Z_s^{\ell}} + h \sup_{s \le T} \norm{W_s^{\ell}} \right)  \right)^2
\nonumber 
\end{align}
where in the last step we used $s_{T-r} \le T-r$. 
Inserting this upper bound back into \eqref{vop_at} yields  \[ A_T^i  \, \le \, c_T  \frac{5}{4} T^2 \hat{C} A_0^i + \epsilon^2 \tilde{L}^2   T^4 \left(7 + \frac{1}{K T^2} \right)   \left( \frac{1}{N} \sum_{\ell=1}^N \left( \sup_{s \le T} \norm{Z_s^{\ell}} + h \sup_{s \le T} \norm{W_s^{\ell}} \right)  \right)^2 \;. 
\] The required estimate is then obtained by inserting the inequality \eqref{eq:bound_c_T} 
which is valid since $L T^2 \le 1/4$ and $K \le L$.  The required result then holds a.s.~because $1/2 - 1/24 = 11/24 > 1/3$.
\end{proof}

In retrospect,  asymptotic strong co-coercivity of $\nabla V$ plays a key role in this proof, and in fact, allows one to obtain an optimal parameter dependence in the contraction rate when $\mathcal{R}=0$ in \ref{ass_Vconv}.

\section{Strong Accuracy of Randomized Time Integrator for Mean-Field \texorpdfstring{$U$}{}} 

\label{sec:strong_accuracy}

The main result of this section gives strong accuracy  bounds for the randomized time integrator defined in \eqref{eq:hamdyn_meanf_num}; in turn, this  allows to quantify the asymptotic bias of the corresponding inexact uHMC chain \cite{DuEb21}. Towards this end, the following weighted norm is useful to define: \begin{equation}
\wnorm{(x,v)}^2 \ = \ |x|^2 + L_e^{-1} |v|^2 \quad \forall~(x,v) \in \mathbb{R}^{2d} \label{eq:wnorm}
\end{equation} 
where for notational brevity it helps to work with $L_e = L+2 \epsilon \tilde{L}$.  

\begin{theorem} \label{thm:strong_accuracy}
Suppose \Cref{ass} \ref{ass_Vmin}, \ref{ass_Vlip}, and \ref{ass_Wlip} hold.  Let $\epsilon \in [0, \infty)$ and $T \in (0,\infty)$   satisfy \eqref{eq:Cepsi} and \eqref{eq:CT}. Let $h \ge 0$ satisfy $T/h \in \mathbb{Z}$ if $h>0$.  Then for all $x,v \in \mathbb{R}^{N d}$ and $k\in \mathbb{N}$ with $k\le T/h$ \begin{equation}
\label{eq:strong_accuracy}
\begin{aligned}
    &\mathbb{E}  \frac{1}{N}\sum_{\ell=1}^N \wnorm{(Q_{k h}^{\ell}(x,v), P_{k h}^{\ell}(x,v)) - (q_{k h}^{\ell}(x,v), p_{k h}^{\ell}(x,v))}  \\
    & \qquad \ \le \ 5 (L_e^{1/2} h)^{3/2}  \left( 5\sqrt{ \frac{1}{N}\sum_{\ell=1}^N  \norm{x^{\ell}}^2 }+10T\sqrt{ \frac{1}{N}\sum_{\ell=1}^N \norm{v^{\ell}}^2 } + 3 L_{e}^{-1} \epsilon \mathbf{W}_0 \right).
\end{aligned}
\end{equation}
\end{theorem}

\begin{proof}
Let $\{ t_i := i h \}_{i \in \mathbb{N}_0}$ be an evenly spaced time grid.  Let $\mathcal{F}_{t_k}$ denote the sigma-algebra of events generated by $\{ \mathcal{U}_k \}_{k \in \mathbb{N}_0}$ in \eqref{eq:hamdyn_meanf_num} up to the $k$-th step. Let $\Delta^{\ell}_{t_{k}} := \mathbb{E} \left[  \wnorm{(Q_{t_k}^{\ell}(x,v), P_{t_k}^{\ell}(x,v)) - (q_{t_k}^{\ell}(x,v), p_{t_k}^{\ell}(x,v))}^2 \right]$ denote the squared $L^2$-error in the $\ell$-th component for $k \in \mathbb{N}_0$.   For any $(x,v) \in \mathbb{R}^{N d}$, let $\varphi_h(x,v) := (q_h(x,v),p_h(x,v))$ be the exact flow map of \eqref{eq:hamdyn_exact} from $(x,v)$ for a duration $h$. 
Applying Young's inequality yields: 
\begin{align}
     \sum_{\ell=1}^N  \Delta^{\ell}_{t_{k+1}} \ &\le \         \rn{1}_{t_{k+1}}+\rn{2}_{t_{k+1}}+\rn{3}_{t_{k+1}} \quad \text{where} 
     \label{eq:L2_expansion}
     \\
     \rn{1}_{t_{k+1}} \ &:= \  (1+ L_{e}^{1/2} h) \sum_{\ell=1}^N \mathbb{E} \left[  \wnorm{\varphi_h^{\ell}(Q_{t_{k}}(x,v), P_{t_{k}}(x,v)) - \varphi_h^{\ell}(q_{t_{k}}(x,v), p_{t_{k}}(x,v))}^2 \right]  \label{eq:rn1} \\
     \rn{2}_{t_{k+1}} \ &:= \ \frac{1}{L_{e}^{1/2} h}  \sum_{\ell=1}^N \mathbb{E} \left[  \wnorm{\mathbb{E} \left[ (Q_{t_{k+1}}^{\ell}(x,v), P_{t_{k+1}}^{\ell}(x,v)) \mid \mathcal{F}_{t_k} \right] - \varphi_h^{\ell}(Q_{t_{k}}(x,v), P_{t_{k}}(x,v))}^2 \right]
     \label{eq:rn2} \\
     \rn{3}_{t_{k+1}} \ &:= \  \sum_{\ell=1}^N \mathbb{E} \left[  \wnorm{(Q_{t_{k+1}}^{\ell}(x,v), P_{t_{k+1}}^{\ell}(x,v)) - \varphi_h^{\ell}(Q_{t_{k}}(x,v), P_{t_{k}}(x,v))}^2 \right]
     \label{eq:rn3}
\end{align}
Inserting \eqref{eq:lip_exact} from Lemma~\ref{lem:lipschitz_exact_flow} into $\rn{1}_{t_{k+1}}$ yields \begin{equation} \label{eq:sa_rn1}
\rn{1}_{t_{k+1}} \ \le \ (1+ L_{e}^{1/2} h) (1+ 7 L_{e}^{1/2} h) \sum_{\ell=1}^N  \Delta^{\ell}_{t_{k}} \ \le \ (1+ \frac{28}{3} L_{e}^{1/2} h) \sum_{\ell=1}^N  \Delta^{\ell}_{t_{k}} 
\end{equation} Inserting \eqref{eq:local_mean_error} from Lemma~\ref{lem:rmm_local_mean_error} into $ \rn{2}_{t_{k+1}}$ yields \begin{equation} \label{eq:sa_rn2}
\rn{2}_{t_{k+1}} \ \le \  7 (L_e^{1/2} h)^5  \sum_{\ell=1}^N \left(     \norm{Q_{t_k}^{\ell}(x,v)}^2 +  L_e^{-1} \norm{P_{t_k}^{\ell}(x,v)}^2
+ \frac{1}{3} L_e^{-2} \epsilon^2  \mathbf{W}_0^2 \right)
\end{equation}
Inserting \eqref{eq:local_L2_error} from  Lemma~\ref{lem:rmm_local_L2_error} into $\rn{3}_{t_{k+1}}$ \begin{equation} \label{eq:sa_rn3}
\rn{3}_{t_{k+1}} \ \le \  6 (L_e^{1/2} h)^4  \sum_{\ell=1}^N \left(     \norm{Q_{t_k}^{\ell}(x,v)}^2 +  L_e^{-1} \norm{P_{t_k}^{\ell}(x,v)}^2
+ \frac{1}{3} L_e^{-2} \epsilon^2  \mathbf{W}_0^2 \right)
\end{equation}
Combining \eqref{eq:sa_rn1}, \eqref{eq:sa_rn2}, and \eqref{eq:sa_rn3} and using $L_e h^2 \le 1/9$ yields \begin{align*}
  \sum_{\ell=1}^N  \Delta^{\ell}_{t_{k+1}} \ &\le \ (1+ \frac{28}{3} L_{e}^{1/2} h) \sum_{\ell=1}^N  \Delta^{\ell}_{t_{k}} + \frac{25}{3} (L_e^{1/2} h)^4  \sum_{\ell=1}^N \left(     \norm{Q_{t_k}^{\ell}(x,v)}^2 +  L_e^{-1} \norm{P_{t_k}^{\ell}(x,v)}^2
+ \frac{1}{3} \epsilon^2  \mathbf{W}_0^2 \right) \;.
\end{align*}
By discrete Gr\"{o}nwall's inequality (Lemma~\ref{lem:groenwall}) and Jensen's inequality \begin{align*}
& \mathbb{E} \frac{1}{N} \sum_{\ell=1}^N \wnorm{(Q_{t_{k+1}}^{\ell}(x,v), P_{t_{k+1}}^{\ell}(x,v)) - (q_{t_{k+1}}^{\ell}(x,v), p_{t_{k+1}}^{\ell}(x,v))}  \ \le \ \sqrt{ \frac{1}{N} \sum_{\ell=1}^N  \Delta^{\ell}_{t_{k+1}} } \\
& \qquad \ \le \  \sqrt{ \frac{25}{28} (L_e^{1/2} h)^3  e^{\frac{28}{3} L_e^{1/2} T} \frac{1}{N}\sum_{\ell=1}^N \left(     \norm{Q_{t_k}^{\ell}(x,v)}^2 +  L_e^{-1} \norm{P_{t_k}^{\ell}(x,v)}^2
+ \frac{1}{3} L_e^{-2} \epsilon^2  \mathbf{W}_0^2 \right) } \\
& \qquad \ \le \   5 (L_e^{1/2} h)^{3/2}  \sqrt{ \frac{1}{N}\sum_{\ell=1}^N      \left( \norm{Q_{t_k}^{\ell}(x,v)}^2 +  L_e^{-1} \norm{P_{t_k}^{\ell}(x,v)}^2 \right) 
+ \frac{1}{3}  L_e^{-2} \epsilon^2  \mathbf{W}_0^2 } \; . 
\end{align*}
By applying \eqref{apriori:QiVidev}, \eqref{apriori:QVdev} and \eqref{eq:CT} in succession and using Young's inequality
\begin{align*}
    &\frac{1}{N}\sum_{\ell=1}^N      \left( \norm{Q_{t_k}^{\ell}(x,v)}^2 +  L_e^{-1} \norm{P_{t_k}^{\ell}(x,v)}^2 \right)
    \\&  \qquad \ \le \ \frac{1}{N}\sum_{\ell=1}^N  2 \left( \frac{4}{3}\left(\norm{x^{\ell}}+2T \norm{v^{\ell}}\right) + L_e^{-1} \epsilon \frac{5}{3}\mathbf{W}_0+\frac{6\epsilon\tilde{L}T}{N}L_e^{-1/2} \sum_{k=1}^N\frac{4}{3}\left(\norm{x^{k}}+2T \norm{v^{k}}\right) \right)^2 
    \\ & \qquad \ \le \  \frac{1}{N}\sum_{\ell=1}^N  \left(\frac{16}{3}(\norm{x^{\ell}}+2T\norm{v^{\ell}})^2+\frac{20}{3}L_{e}^{-2} \epsilon^2 \mathbf{W}_0^2 + 32\left(\frac{\epsilon\tilde{L}T}{N}L_{e}^{-1/2} \sum_{k=1}^N  (\norm{x^{k}}+2T\norm{v^{k}})\right)^2  \right)
    \\ & \qquad \ \le \ \frac{1}{N}\sum_{\ell=1}^N \left( \frac{160}{5}\left(\norm{x^{\ell}}^2+ 4T^2\norm{v^{\ell}}^2\right)+ \frac{20}{3}L_{e}^{-2} \epsilon^2 \mathbf{W}_0^2 \right).
\end{align*}
Hence, 
\begin{align*}
    & \mathbb{E} \frac{1}{N} \sum_{\ell=1}^N \wnorm{(Q_{t_{k+1}}^{\ell}(x,v), P_{t_{k+1}}^{\ell}(x,v)) - (q_{t_{k+1}}^{\ell}(x,v), p_{t_{k+1}}^{\ell}(x,v))}  
    \\ & \qquad \ \le \ 5 (L_e^{1/2} h)^{3/2}  \left( \frac{4}{3}\sqrt{ \frac{1}{N}\sum_{\ell=1}^N  10\norm{x^{\ell}}^2 }+\frac{8}{3}\sqrt{ \frac{1}{N}\sum_{\ell=1}^N  10T^2\norm{v^{\ell}}^2 } + \sqrt{7}L_{e}^{-1} \epsilon \mathbf{W}_0 \right).
\end{align*}
Simplifying this bound gives the required result.
\end{proof}

The proof of Theorem~\ref{thm:strong_accuracy} uses a discrete Gr\"{o}nwall's inequality, which we include here for the reader's convenience.  
\begin{lemma}[Discrete Gr\"{o}nwall's inequality]
\label{lem:groenwall}
Let $\lambda, h \in \mathbb{R}$ be such that $1+\lambda h > 0$.  Suppose that $(g_k)_{k \in \mathbb{N}_0}$ is a non-decreasing sequence, and $(a_k)_{k \in \mathbb{N}_0}$ satisfies $a_{k+1} \le (1 + \lambda h) a_{k} + g_{k}$ for $k \in \mathbb{N}_0$.  Then it holds \[
a_k \ \le \ (1+\lambda h)^k a_0 + \frac{1}{\lambda h} ( (1+\lambda h)^k - 1) g_{k-1} , \quad \text{for all $k \in \mathbb{N}$} \;.
\]
\end{lemma}

\subsection{A Priori Bounds for Exact Flow}

\begin{lemma}[Growth Condition] 
\label{lem:growth_exact_flow}
Let $h > 0$ satisfy $L_e h^2 \le 1/9$ where $L_e = L+2 \epsilon \tilde{L}$.  Then for all $(x,v) \in \mathbb{R}^{N d}$  
\begin{equation}
\label{eq:growth_exact_flow}
    \begin{aligned}
        \sum_{\ell=1}^N \sup_{0 \le s \le h} \norm{q_s^{\ell}(x,v) - x^{\ell} - s v^{\ell}}^2  \ \le \  \frac{16}{5} (L_e h^2)^2  \sum_{\ell=1}^N \left(     \norm{x^{\ell}}^2 + h^2 \norm{v^{\ell}}^2
+ \frac{1}{3} L_e^{-2} \epsilon^2  \mathbf{W}_0^2 \right) \;.
    \end{aligned}
\end{equation}
\end{lemma}

\begin{proof}
By \eqref{eq:hamdyn_exact} and the Cauchy-Schwarz inequality, \begin{align}
& \sum_{\ell=1}^N \sup_{0 \le s \le h} \norm{q_s^{\ell} - x^{\ell} - s v^{\ell}}^2 =  \sum_{\ell=1}^N \sup_{0 \le s \le h} \norm{\int_0^s (s-r) \nabla_{\ell} U(q_r) dr}^2 \nonumber \\
& \quad   \le   \frac{h^3}{3} \sum_{\ell=1}^N  \sup_{0 \le s \le h} \int_0^s  \norm{\nabla_{\ell} U(q_r)}^2 dr  \overset{\eqref{eq:G_i_mf}}{\le}  h^3 \sum_{\ell=1}^N  \sup_{0 \le s \le h} \int_0^s   \left[ (L+\epsilon \tilde{L})^2 \norm{q_r^{\ell}}^2 + \frac{(\epsilon \tilde{L})^2}{N} \sum_{k=1}^N  \norm{q_r^{k}}^2 + L_e^{-2}\epsilon^2  \mathbf{W}_0^2  \right] dr  \nonumber \\
& \quad   \le   L_e^2 h^4  \sum_{\ell=1}^N \left( \sup_{0 \le s \le h}   \norm{q_s^{\ell}}^2  
+ \epsilon^2  \mathbf{W}_0^2 \right)   \le  3 L_e^2 h^4  \sum_{\ell=1}^N \big( \sup_{0 \le s \le h}   \norm{q_s^{\ell} - x^{\ell} - s v^{\ell}}^2   +  \norm{x^{\ell}}^2 + h^2 \norm{v^{\ell}}^2 
+ \frac{1}{3} L_e^{-2} \epsilon^2  \mathbf{W}_0^2 \big) . \nonumber 
\end{align}    
The required result follows by using $L_e h^2 \le 1/9$ and simplifying this expression.  
\end{proof}

\begin{lemma}[Lipschitz Continuity] 
\label{lem:lipschitz_exact_flow}
Let $h > 0$ satisfy $L_e h^2 \le 1/9$ where $L_e = L+2 \epsilon \tilde{L}$.  Then for all $(x,v), (y,u) \in \mathbb{R}^{N d}$ \begin{equation} \label{eq:lip_exact}
    \begin{aligned}
        \sum_{\ell=1}^N \sup_{0 \le s \le h} \wnorm{(q_s^{\ell}(x,v),p_s^{\ell}(x,v)) - (q_s^{\ell}(y,u),p_s^{\ell}(y,u))}^2 \  \le \ (1 + 7 L_e^{1/2} h ) \sum_{\ell=1}^N   \wnorm{(z_0^{\ell},w_0^{\ell})}^2  \;.
    \end{aligned}
\end{equation}
\end{lemma}

\begin{proof}
 For $t \in [0,h]$, it is  convenient to define $(x_t,v_t) := (q_t(x,v),p_t(x,v)) $, $(y_t,u_t) := (q_t(y,u),p_t(y,u))$ and \begin{align*}
(z_t,w_t) \ := \ (x_t - y_t,v_t-u_t), \quad  \beta_t^{i} \ := \ \nabla V(x_t^{i}) - \nabla V(y_t^{i}),  \quad \Omega_t^{i} \ := \ \frac{1}{N} \sum_{\ell=1}^N \left( \nabla_1 W(x_t^{i}, x_t^{\ell}) - \nabla_1 W(y_t^{i}, y_t^{\ell}) \right) \;.
\end{align*}
Therefore, $\nabla_i U(x_t) - \nabla_i U(y_t) = \beta_t^{i} + \epsilon \Omega_t^{i}$.  By expanding \eqref{eq:wnorm} and using Cauchy-Schwarz inequality, 
\begin{align}
&        \sum_{\ell=1}^N \sup_{0 \le s \le h} \wnorm{(z_s^{\ell},w_s^{\ell}) - (z_0^{\ell}+s w_0^{\ell},w_0^{\ell})}^2 \ \le \  \sum_{\ell=1}^N \left( \sup_{0 \le s \le h}  \norm{z_s^{\ell}-z_0^{\ell}-s w_0^{\ell}}^2 + L_e^{-1} \sup_{0 \le s \le h}  \norm{w_s^{\ell}-w_0^{\ell}}^2 \right) \nonumber \\
& \quad \ \le \ \sum_{\ell=1}^N \left( \sup_{0 \le s \le h}  \norm{\int_0^s (s-r) [\beta_r^{\ell} + \epsilon \Omega_r^{\ell}] dr}^2 + L_e^{-1} \sup_{0 \le s \le h}  \norm{\int_0^s [\beta_r^{\ell} + \epsilon \Omega_r^{\ell}] dr}^2  \right) \nonumber \\
& \quad \ \le \ \sum_{\ell=1}^N \left( \frac{h^3}{3} + L_e^{-1} h \right)  \sup_{0 \le s \le h} \int_0^s  \norm{\beta_r^{\ell} + \epsilon \Omega_r^{\ell}}^2 dr \ \overset{\eqref{eq:L_i_mf}}{\le} \   \sum_{\ell=1}^N 2 \left( \frac{L_e^2 h^4}{3} + L_e h^2 \right)  \sup_{0 \le s \le h}  \norm{z_s^{\ell}}^2 \nonumber  \\
& \quad \ \le \ \sum_{\ell=1}^N 6 \left( \frac{L_e^2 h^4}{3} + L_e h^2 \right) \left( \sup_{0 \le s \le h} \norm{z_s^{\ell}-z_0^{\ell}-s w_0^{\ell}}^2  + \norm{z_0^{\ell}} + h^2 \norm{w_0^{\ell}}^2 \right) \nonumber 
\end{align}
The required result follows by using $L_e h^2 \le 1/9$ and simplifying this expression.  
\end{proof}

\subsection{Single-step Discretization Error Bounds}

\begin{lemma}[Local Mean Error of Randomized Time Integrator]  
\label{lem:rmm_local_mean_error} Let $h > 0$ satisfy $L_e h^2 \le 1/9$ where $L_e = L+2 \epsilon \tilde{L}$.  Then for all $(x,v) \in \mathbb{R}^{N d}$ it holds 
\begin{equation} \label{eq:local_mean_error}
    \begin{aligned}
 \sum_{\ell=1}^N       \wnorm{\mathbb{E} \left[ (Q_{h}^{\ell}(x,v), P_{h}^{\ell}(x,v)) \right] - (q_{h}^{\ell}(x,v), p_{h}^{\ell}(x,v))}^2 \ \le \  7 (L_e h^2)^3  \sum_{\ell=1}^N \left(     \wnorm{(x^{\ell},v^{\ell})}^2
+ \frac{1}{3} L_e^{-2} \epsilon^2  \mathbf{W}_0^2 \right) \;.
    \end{aligned}
\end{equation}

\end{lemma}

\begin{proof}
 As a first step, expand the squared, local mean error according to \eqref{eq:wnorm} \begin{align}
 \qquad \qquad & \sum_{\ell=1}^N  \wnorm{\mathbb{E} \left[ (Q_{h}^{\ell}(x,v), P_{h}^{\ell}(x,v)) \right] - (q_{h}^{\ell}(x,v), p_{h}^{\ell}(x,v))}^2 \le \rn{1} + \rn{2} \quad \text{where} \\
 \rn{1} \ &:= \  \sum_{\ell=1}^N \left| \int_0^h \int_0^h \frac{(h-s)}{h} \left[ \nabla_{\ell} U(x+r v) - \nabla_{\ell} U(q_s) \right] dr ds \right|^2 \;, \\
  \rn{2} \ &:= \ L_e^{-1} \sum_{\ell=1}^N \left|  \int_0^h \left[ \nabla_{\ell} U(x+s v) - \nabla_{\ell} U(q_s) \right]  ds \right|^2  \;.
\end{align}
    By applying Cauchy-Schwarz inequality twice and then inserting \eqref{eq:L_i_mf}, 
\begin{align}
& \rn{1} \ = \ \sum_{\ell=1}^N \left| \int_0^h \frac{(h-s)}{h} \int_0^h  \left[ \nabla_{\ell} U(x+r v) - \nabla_{\ell} U(q_s) \right] dr ds \right|^2  \nonumber \\
& \quad \ \le \ \int_0^h \frac{(h-s)^2}{h^2} ds   \sum_{\ell=1}^N   \int_0^h  \left| \int_0^h  \left[ \nabla_{\ell} U(x+r v) - \nabla_{\ell} U(q_s) \right] dr  \right|^2 ds  \nonumber \\
& \quad \ \le \ \frac{h^2}{3} \sum_{\ell=1}^N  \int_0^h   \int_0^h \left|  \nabla_{\ell} U(x+r v) - \nabla_{\ell} U(q_s) \right|^2 dr   ds \nonumber \\
& \quad  \overset{\eqref{eq:L_i_mf}}{\le}    \frac{2 L_e^2 h^2}{3} \sum_{\ell=1}^N   \int_0^h   \int_0^h \left|  (r-s) v^{\ell} - (q_s^{\ell} - x^{\ell} - s v^{\ell} ) \right|^2 dr   ds  \nonumber \\
& \quad \ \le \  \frac{4 L_e^2 h^2}{3}  \sum_{\ell=1}^N \left( \frac{h^4}{12} |v^{\ell}|^2 +  h^2 \sup_{0 \le s \le h} \left| q_s^{\ell} - x^{\ell} - s v^{\ell}  \right|^2 \right)  \label{eq:lme:rn1}
\end{align}
where in the last two steps Young's product inequality was used.  Similarly, 
\begin{align}
& \rn{2}  \ \le \ L_e^{-1} h \sum_{\ell=1}^N   \int_0^h \left| \nabla_{\ell} U(x+s v) - \nabla_{\ell} U(q_s) \right|^2  ds  \overset{\eqref{eq:L_i_mf}}{\le}    2 L_e h^2 \sum_{\ell=1}^N  \sup_{0 \le s \le h} \left| q_s^{\ell} - x^{\ell} - s v^{\ell}  \right|^2    \;. \label{eq:lme:rn2}
\end{align}
Combining \eqref{eq:lme:rn1} and \eqref{eq:lme:rn2}   yields \begin{align}
& \rn{1} + \rn{2} \ \le \ \frac{ L_e^2 h^4}{9}  \sum_{\ell=1}^N  h^2 |v^{\ell}|^2 + \left(  \frac{4 L_e^2 h^4}{3}  + 2 L_e h^2 \right) \sum_{\ell=1}^N  \sup_{0 \le s \le h} \left| q_s^{\ell} - x^{\ell} - s v^{\ell}  \right|^2    \nonumber \\
& \quad \ \overset{\eqref{eq:growth_exact_flow}}{\le} \   (L_e h^2)^2   \sum_{\ell=1}^N   h^2 |v^{\ell}|^2  + 7 (L_e h^2)^3  \sum_{\ell=1}^N \left(     \norm{x^{\ell}}^2 
+ \frac{1}{3} L_e^{-2} \epsilon^2  \mathbf{W}_0^2 \right) \;.
\end{align}
Simplifying this expression by using $L_e h^2 \le 1/9$ gives the required bound.  \end{proof}

\begin{lemma}[Local Mean Squared Error of Randomized Time Integrator]  
\label{lem:rmm_local_L2_error}
Let $h > 0$ satisfy $L_e h^2 \le 1/9$ where $L_e = L+2 \epsilon \tilde{L}$.  Then for all $(x,v) \in \mathbb{R}^{N d}$ it holds
\begin{equation} \label{eq:local_L2_error}
    \begin{aligned}
\sum_{\ell=1}^N \mathbb{E} \left[  \wnorm{(Q_{h}^{\ell}(x,v), P_{h}^{\ell}(x,v)) - q_{h}^{\ell}(x,v), p_{h}^{\ell}(x,v))}^2 \right] \ \le \ 6 (L_e h^2)^2  \sum_{\ell=1}^N \left(     \wnorm{(x^{\ell},v^{\ell})}^2
+ \frac{1}{3} L_e^{-2} \epsilon^2  \mathbf{W}_0^2 \right) \;.
    \end{aligned}
\end{equation}
\end{lemma}

\begin{proof}
Let $\rn{10} : = \sum_{\ell=1}^N  \mathbb{E} \wnorm{ (Q_{h}^{\ell}(x,v), P_{h}^{\ell}(x,v))  - (q_{h}^{\ell}(x,v), p_{h}^{\ell}(x,v))}^2$
By applying Cauchy-Schwarz inequality, \begin{align}
& \rn{10} \overset{\eqref{eq:wnorm}}{=}   \mathbb{E}  \sum_{\ell=1}^N \left| \int_0^h (h-s) \left[ \nabla_{\ell} U(x+ h \mathcal{U}_0 v) - \nabla_{\ell} U(q_s) \right]  ds \right|^2 + L_e^{-1} \mathbb{E}  \sum_{\ell=1}^N \left|  \int_0^h \left[ \nabla_{\ell} U(x+ h \mathcal{U}_0 v) - \nabla_{\ell} U(q_s) \right]  ds \right|^2 \nonumber \\
&~~ \ \le \ \left( \int_0^h (h-s)^2 ds + L_e^{-1}  h  \right) \mathbb{E}  \sum_{\ell=1}^N  \int_0^h  \left| \nabla_{\ell} U(x+ h \mathcal{U}_0 v) - \nabla_{\ell} U(q_s) \right|^2  ds  \nonumber \\
&~~ \overset{\eqref{eq:L_i_mf}}{\le} 2 \left( \frac{L_e^2 h^3}{3} + L_e h \right) \mathbb{E}  \sum_{\ell=1}^N   \int_0^h \left|  (h \mathcal{U}_0 - s) v^{\ell} - (q_s^{\ell} - x^{\ell} - s v^{\ell})  \right|^2  ds  \nonumber \\
&~~ \ \le \ 4 \left( \frac{L_e^2 h^3}{3} + L_e h \right) \left(h^3 \sum_{\ell=1}^N  |v^{\ell}|^2 +  h  \sum_{\ell=1}^N \sup_{0 \le s \le h} \left| q_s^{\ell} - x^{\ell} - s v^{\ell}  \right|^2 \right) \nonumber 
\end{align}
Inserting \eqref{eq:growth_exact_flow} and using $L_e h^2 \le 1/9$ gives the required upper bound.
\end{proof}

\newpage 

\appendix

\section{Uniform-in-steps Second Moment Bound and Contractivity for xHMC} 

\label{appendix}

 In the following, we state a uniform-in-steps moment bound for xHMC for the mean-field particle system and an extension of a contraction result for xHMC  \cite[Theorem 3]{BouRabeeSchuh2023} with slightly less restrictive conditions on $\epsilon$ and $T$.
	\begin{lemma} [Uniform-in-steps second moment bound of xHMC for mean-field particle models] \label{lem:secondmoment_meanf}
		 Suppose \Cref{ass}, \eqref{eq:cond_t} and \eqref{eq:C_eps} hold. Suppose $\mathbb{E}[N^{-1}\sum_{i=1}^N|x_0^i|^2]<\infty$.
		Then there exists a constant $\mathbf{B}_2$ such that both the second moment of the Hamiltonian dynamics for $t\leq T$ is uniformly bounded, i.e.,
		\begin{align*}
		\sup_{0 \le t \le T}\mathbb{E}\Big[\frac{1}{N}\sum_{i=1}^N|{x}_t^ i|^2\Big]\leq \mathbf{B}_2<\infty,
		\end{align*}
		and the second moment of the Markov chain generated by nHMC is uniformly bounded, i.e.,
		\begin{align*}
		\sup_{m\in\mathbb{N}}\mathbb{E}\Big[\frac{1}{N}\sum_{i=1}^N|\bar{\mathbf{X}}_m^ i|^2\Big]\leq \mathbf{B}_2<\infty,
		\end{align*}
		Moreover, $\mathbf{B}_2$ is given by 
  \begin{align}\label{eq:mathbfC_2}
      \mathbf{B}_2=\mathbb{E}\left[N^{-1}\sum_{i=1}^N|x_0^i|^2\right]+\frac{13}{1280K}\left( 11d+\left(\frac{15\epsilon}{2\tilde{L}}+6\epsilon^2 T^2\right)\mathbf{W}_0^2+\mathcal{R}^2(2L+K) \right).
  \end{align}
  depends on $K$, $\epsilon$, $L$, $\mathcal{R}$, $T$, $d$, $\mathbf{W}_0$ and $\mathbb{E}[N^{-1}\sum_{i=1}^N|{x}_0^ i|^2]$.
	\end{lemma}

	\begin{proof} 
Though the proof is similar to the proof of \Cref{lem:secondmoment}, a complete proof is given for the reader's convenience. For $i \in \{ 1,\ldots,N \}$, define 
\begin{align*}
 \Theta_t^i :=\sum_{j=1}^N \nabla_1 W(x^i_t, x^j_t) \;.
\end{align*}
Then, $\nabla_i U(x_t) = \nabla V(x^i_t) + \epsilon \Theta_t^i$. Let $a^i_t := \norm{x^i_t}^2$ and  $b^i_t := 2 \langle x^i_t, v^i_t \rangle$.  Note that $\rho_t^i := \langle x^i_t, \nabla V(x^i_t) \rangle$ satisfies 
\begin{equation} \label{ieq:rhot_momentx}
K a^i_t + \frac{1}{L} \norm{\nabla V(x^i_t)}^2 - \hat{C} \overset{\ref{ass_Vconv}}{\le}  \rho_t^i \overset{\ref{ass_Vlip}}{\le}  L a^i_t.
\end{equation}
From \eqref{eq:hamdyn_exact},  
\begin{equation}
    \begin{aligned}
        \frac{d}{dt} a^i_t := b^i_t, \quad \frac{d}{dt} b^i_t = 2 |v^i_t|^2 - 2 \langle x^i_t, \nabla V(x^i_t) \rangle - 2 \epsilon \langle x^i_t, \Theta_t^i \rangle  \;. 
    \end{aligned}
\end{equation}
By variation of parameters, for $t\le T$
\begin{equation} \label{vop_bar_at_momentx} 
    a^i_t = c_t a^i_0 + \int_0^t s_{t-r} \left( 2 \norm{v^i_r}^2 - 2 \rho^i_r + \delta^i_r \right) dr
\end{equation} 
where $\delta^i_t := K a^i_t  - 2 \epsilon \langle x^i_t, \Theta_t^i \rangle$, $c_t:=\cos(\sqrt{ K} t )$, and $s_t := \frac{\sin( \sqrt{ K} t )}{ \sqrt{ K}} $.  
We bound  $\delta_t^i$  in \eqref{vop_bar_at_momentx} from above by   
\begin{equation} \label{eq:delta_i} 
\begin{aligned}    
\delta_t^i  \  &\le \  K a^i_t  + 2 \epsilon \norm{x^i_t}\norm{\Theta^i_t} \ \overset{\ref{ass_Wlip}}{\le}  K a^i_t + 2 \epsilon \norm{x^i_t}\tilde{L}(\norm{x^i_t}+ \mathbb{E}[\norm{x^i_t}] +\mathbf{W}_0/\tilde{L}) \ 
\\ & \overset{\eqref{eq:C_eps}}{\le} \  2K a^i_t + \epsilon\tilde{L} \max_{r\le T} \Big(\frac{1}{N}\sum_{j=1}^N \norm{x^j_r}\Big)^2+ K/45 a^i_t + \epsilon^2 45/K\mathbf{W}_0^2  \  \qquad \text{for }t\le T. 
\end{aligned}
\end{equation}

To upper bound the terms involving $|v^i_t|^2$ in \eqref{vop_bar_at_momentx},  use \eqref{eq:hamdyn_exact}.   Therefore, 
\begin{align} \label{eq:Wt2_momentx}
 |v^i_t|^2 &\, = \,   \norm{  -\int_0^{t}  ( \nabla V(x^i_s) + \epsilon \Theta^i_s)  ds + v^i_0}^2 \ \, \le \,  8 \norm{v^i_0}^2 + \frac{8}{7} t  \int_0^{t} \norm{ \nabla V(x^i_s) + \epsilon \Theta^i_s}^2 ds \nonumber \\
&		\, \le \, 8 \norm{v^i_0}^2  + t  \int_0^{t}  \left(8\norm{\nabla V(x^i_s)}^2 + \frac{4}{3}\epsilon^2 \norm{ \Theta^i_s  }^2\right) ds 
\end{align}
where we used the Cauchy-Schwarz inequality.  
Therefore, combining \eqref{eq:Wt2_momentx} and \eqref{eq:sinus}, and by Fubini's Theorem,  we obtain as in \eqref{ieq:Wt2_moment} 
\begin{align}
& \int_0^t s_{t-r} |v^i_r|^2 dr \, \le \,  t^2  \int_0^t   s_{t-s} \left(4\norm{\nabla V(x^i_s)}^2 + \frac{2}{3}\epsilon^2 \norm{ \Theta^i_s  }^2\right)  ds + 8\int_0^t s_{t-r} |v^i_0|^2 dr \;. \label{ieq:Wt2_momentx}
\end{align} 
To upper bound the mean-field interaction force,   
\begin{align}  & \norm{\Theta_t^i}^2   \le \Big( \ \frac{1}{N}\sum_{j=1}^N \norm{ \nabla_1 W(x^i_t,x_t^j) }  \Big)^2  \overset{\ref{ass_Wlip}}{\le} \Big( \tilde{L} ( \norm{ x^i_t} + \frac{1}{N}\sum_{j=1}^N \norm{x^j_t} ) + \mathbf{W}_0 \Big)^2    
\\ & \quad \le 2 \tilde{L}^2 a^i_t + 3 \tilde{L}^2 \max_{s\le t}\Big(\frac{1}{N}\sum_{j=1}^N \norm{x^j_s}\Big)^2 + 6 \mathbf{W}_0^2.
\label{ieq:Omt_momentx}
\end{align}

 Inserting these bounds into the second term of \eqref{vop_bar_at_momentx} and simplifying yields similarly to \eqref{eq:s_T_moment}
\begin{align}
&  \int_0^t s_{t-r} \left( 2 \norm{v^i_r}^2 - 2 \rho^i_r + \delta^i_r \right) dr  \nonumber  \\
&  \le \frac{t^2}{2}\left(\left(\epsilon \tilde{L}  \frac{2 }{15}+\frac{K}{45}\right)\max_{s\le T}a^i_s + \epsilon \tilde{L}\frac{6}{5}\max_{s\le T}\Big(\frac{1}{N}\sum_{j=1}^N \norm{x^j_s}\Big)^2+  2\hat{C}+  \frac{45\epsilon^2}{K}\mathbf{W}_0^2+ 8\epsilon^2 T^2 \mathbf{W}_0^2 + 16 \norm{v^i_0}^2 \right).
\end{align}
Inserting this upper bound back into \eqref{vop_bar_at_momentx}, summing over the number of particles and using \eqref{eq:bound_c_T_moment} yields  
\begin{equation}
\begin{aligned}
& \frac{1}{N}\sum_{i=1}^N a_t^i  
\, \le \, (1 - (11/24) K t^2)\frac{1}{N}\sum_{i=1}^N a_0^i +  \frac{t^2}{2}\left( \frac{K}{15N}\sum_{i=1}^N \max_{s\le T}a^i_s + \epsilon \tilde{L}\frac{6}{5}\max_{s\le T}\Big(\frac{1}{N}\sum_{j=1}^N\norm{x^j_s}\Big)^2 \right)  
\\ & \qquad + 8t^2 \frac{1}{N}\sum_{i=1}^N\norm{v^i_0}^2 + t^2\left( \frac{45\epsilon^2}{2K}+4\epsilon^2 T^2 \right)  \mathbf{W}_0^2 +  t^2\hat{C} \;. 
\end{aligned} 
\label{eq:expa_momentx}
\end{equation}
Analogously to  \eqref{eq:lem_nonlin1} and \eqref{eq:lem_nonlin4} it holds for $t\le T$,
\begin{align*}
    &\max_{s\le t} |x_s^i|\le (1+(L+\epsilon\tilde{L})t^2)\max(|x^i|,|x^i+tv^i|)+t^2\epsilon\tilde{L} \max_{s\le t}\frac{1}{N}\sum_{j=1}^N |x^j| + \epsilon t^2 \mathbf{W}_0 \qquad \text{and }
    \\ & \max_{s\le t} \sum_{i=1}^N|x_s^i|\le (1+(L+2\epsilon\tilde{L})t^2)\max(\sum_{i=1}^N|x^i|,\sum_{i=1}^N|x^i+tv^i|)+ \epsilon t^2 \mathbf{W}_0.
\end{align*}
Hence by \eqref{eq:cond_t}, it holds analogously to \eqref{eq:bar_a2_moment} and \eqref{eq:bar_Ea2_moment}
\begin{align}
& \frac{1}{N}\sum_{i=1}^N \max_{s\le T} a^i_s 
\le  2 \left( \frac{5}{4}\right)^2 \frac{441}{400}\frac{1}{N}\sum_{i=1}^N a^i_0 + 4 \left( \frac{5}{4}\right)^2 \frac{441}{400} T^2 \frac{1}{N}\sum_{i=1}^N \norm{v^i_0}^2  + 4 \left(\frac{21}{20}\epsilon T^2 \mathbf{W}_0 \right)^2  \quad \text{and} \label{eq:bar_a2_momentx} \\
    & \max_{s\le T}\Big(\frac{1}{N}\sum_{i=1}^N\norm{x^i_s}\Big)^2 
    \le \frac{16}{15} \left(\frac{5}{4}\right)^2 \frac{1}{N}\sum_{i=1}^Na^i_0+ 32\left(\frac{5}{4}\right)^2T^2\frac{1}{N}\sum_{i=1}^N\norm{v^i_0}^2+ 32\epsilon^2 T^4\mathbf{W}_0^2. \nonumber 
\end{align}
Inserting these estimates into \eqref{eq:expa_momentx}, using \eqref{eq:C_eps} and \eqref{eq:cond_t} and taking expectation yields analogously to \eqref{eq:bound_aT_moment}
\begin{align}
    & \mathbb{E} \left[\frac{1}{N}\sum_{i=1}^N a^i_t\right]  
\le \left(1-\frac{13}{1280} Kt^2\right) \mathbb{E} \left[\frac{1}{N}\sum_{i=1}^N a^i_0\right] + 11t^2d + t^2\left( \frac{45\epsilon^2}{2K}+6\epsilon^2 T^2 \right)  \mathbf{W}_0^2 +  t^2\hat{C} \nonumber\;, 
\end{align}

By Grönwall's Inequality, there exists a uniform-in-time second moment bound $\mathbf{B}_2$ for both the second moment of the exact Hamiltonian dynamics $\mathbb{E}[(1/N)\sum_{i=1}^N |x^i|^2]$ and the second moment of the measure after arbitrary many xHMC steps. 
In particular, $\mathbf{B}_2$ depends on $T$, $K$, $L$ and $\epsilon$, $\mathbb{E}[N^{-1} \sum_{i=1}^N|x^i_0|^2]$, $d$, $\mathbf{W}_0$ and $\mathcal{R}$ and can be chosen of the form given in \eqref{eq:mathbfC_2}.\end{proof}

To prove contractivity for xHMC for mean-field models a \emph{particle-wise} coupling of two copies of the xHMC transition step on $\mathbb{R}^{N d}$ is used; cf.~\cite[\S 2.3]{BouRabeeSchuh2023} which  goes back to Eberle \cite[\S 3]{Eb2016A}. In particular, the coupling $(\mathbf{X}(x,x'),\mathbf{X}'(x,x'))$ is given by 
$\mathbf{X}^i(x,x') =q_t^i(x, \xi)$ and $\mathbf{X}'^i(x,x')=q_t^i(x', \eta)$, where $\xi$ and $\eta$ are defined on the same probability space such that $\xi\sim \mathcal{N}(0,I_{Nd})$ and $\eta$ is defined as follows.  Let $\{\mathcal{U}_i\}_{i=1}^N$ be $N$ i.i.d.~standard uniform random variables that are independent of $\xi$.
If $|x^i-{x}'^i| \ge \tilde{R}$, then the random initial velocities of the  $i$th-particles are coupled synchronously, i.e., $\eta^i \ = \  \xi^i$; and otherwise,
\begin{align*}
\eta^i \ = \  \begin{cases} \xi^i+\gamma z^i & \text{if  } \mathcal{U}_i\leq \dfrac{\varphi_{0,1}(e^i\cdot\xi^i+\gamma|z^i|)}{\varphi_{0,1}(e^i\cdot\xi^i)}, \\
\xi^i-2(e^i\cdot\xi^i)e^i & \text{else,}
\end{cases}
\end{align*}
where $z^i=x^i-x'^i$ and $e^i=z^i/|z^i|$ if $|z^i|\neq 0$, and else, $e^i$ is an arbitrary unit vector.
	
	\begin{theorem} [Contraction of xHMC for mean-field particle models] \label{thm:contr_meanf}
		Suppose \Cref{ass} holds. Assume $T \in (0,\infty)$ and $\epsilon\in[0,\infty)$ satisfying \eqref{eq:cond_T} and \eqref{eq:cond_eps}.
		Then for all $x,x'\in\mathbb{R}^{Nd}$,
		\begin{align*} 
		\mathbb{E} \Big[\frac{1}{N}\sum_{i=1}^N \rho({\mathbf{X}}^{i}(x,x'),{\mathbf{X}}'^{i}(x,x'))\Big]\leq (1-c)\mathbb{E}\Big[\frac{1}{N}\sum_{i=1}^N\rho(x,x')\Big],
		\end{align*}
		where $\rho$ is given in \eqref{eq:rho} and the contraction rate $c$ is given in \eqref{eq:c}.
	\end{theorem}
	
	\begin{proof}
    The proof is a combination of the proof \cite[Theorem 3.1]{BouRabeeSchuh2023} whose idea of the proof is also applied in the proof of \Cref{thm:contr_nonlinear} and the analogous result of \Cref{lem:convexarea} for xHMC for mean-field particle systems.
		To avoid confusion, note that the interaction parameter $\epsilon$ in the potential $U$ in \cite{BouRabeeSchuh2023} differs by a factor $2$.
		Let $z_t^i=x_t^i-x_t'^i$ be the difference of the two positions of the $i$th particles at time $t\le T$.
		Following the proof of \Cref{lem:convexarea}, we can bound $a^i(T)=|z_T^i|^2$ by
		\begin{align*}
			a^i(T)& \le (1-(5/12)K T^2)|z_0^i|^2 +\epsilon^2\tilde{L}^2 T^4\Big( \frac{7}{6}+\frac{3}{2KT^2}\Big)\max_{s\le t} \Big(\frac{1}{N}\sum_{i=1}^N|z_s^j|\Big)^2+\hat{C}t^2
			\\ & \le (1-(1/4)K T^2)|z_0^i|^2+\epsilon^2\tilde{L}^2 T^4\Big( \frac{7}{6}+\frac{3}{2KT^2}\Big)\max_{s\le T} \Big(\frac{1}{N}\sum_{i=1}^N|z_s^j|\Big)^2
		\end{align*} 
		for $|z_0^i|\ge \tilde{R}$ with $\tilde{R}$ given in \eqref{eq:R1}.
    Hence, analogous to \eqref{eq:contr_largedist}, for $r^i=|z_0^i|\ge \tilde{R}$ and $R^i=|z_T^i|$,
    \begin{align*}
        \mathbb{E}[f(R^i)-f(r^i)]\leq f'(r^i)\mathbb{E}[R^i-r^i]\leq f'(r^i)(-1/8) K T^2r+f'(r^i)+\epsilon\tilde{L} T^2\sqrt{ \frac{7}{6}+\frac{3}{2KT^2}}\max_{s\leq T}\frac{1}{N}\sum_{j=1}^N|z^j_s|.
    \end{align*}
    For $r^i<\tilde{R}$, we obtain by following \cite[Theorem 3.1]{BouRabeeSchuh2023} or the proof of \Cref{thm:contr_nonlinear} for nonlinear HMC, respectively,
    \begin{align*}
        \mathbb{E}[f(R^i)-f(r^i)]\le -f'(r^i)\frac{21}{240}\gamma T r^i+f'(r^i)\frac{13}{12}\epsilon\tilde{L}T^2\max_{s\leq T}N^{-1}\sum_{j=1}^N|{z}_s^j|.
    \end{align*}
    Combining the two cases as in \cite[Theorem 3.1]{BouRabeeSchuh2023} and \Cref{thm:contr_nonlinear}, respectively, concludes the proof.
	\end{proof}

\section{Shallow Neural Networks}

\label{sec:shallow}

To better motivate the sampling of nonlinear probability measures, and as a supplement to this work, here we make concrete the connection between nonlinear probability measures and the training of a shallow neural network  \cite{MeMoNg18,chizat2018global,sirignano2020mean,RoVa2022, HuReSiSz21}.
The following is based largely on   \cite{HuReSiSz21}; see also \cite[\S 1.6.3]{Sc22Phd}.

In training neural networks, one is interested in finding a function $ f: \mathbb{R}^{d-1}\to \mathbb{R}$ such that for given input data $\mathbf{z} = (z_1,\ldots , z_{d-1}) \in \mathbb{R}^{d-1}$ and output data $y \in \mathbb{R}$, $f (\mathbf{z})$ provides a good approximation of
$y$. 
In a shallow neural network, we consider $f$ of the form $f (\mathbf{z}) =
\sum_{i=1}^N \beta^{i,N} \phi(\alpha^{i,N} \cdot \mathbf{z} )$, where $N$ represents the number of neurons and $\phi : \mathbb{R}\to \mathbb{R}$
is a bounded, continuous, non-constant activation function; a typical example being the
sigmoid function $\phi(r) = 1/(1 + e^{-r})$. The goal  is to find  $\alpha^{i,N} \in \mathbb{R}^{d-1}$
and  $\beta^{i,N} \in \mathbb{R}$ for $ i \in \{ 1, \ldots , N \}$ that solve the optimization problem
\begin{align*}
\min_{\alpha^{i,N}, \beta^{i,N}} \Big\{\int_{\mathbb{R}\times\mathbb{R}^{d-1}} \Big| y - \frac{1}{N}\sum_{i=1}^N  \beta^{i,N} \phi(\alpha^{i,N} \cdot \mathbf{z} )\Big|^2 \nu(\rmd y  \rmd \mathbf{z})  \Big\} \;,
\end{align*}
where $\nu$ is a measure with compact support over the data $(\mathbf{z},y)$.  Equivalently, this optimization problem can be formulated as an optimization problem over the set of empirical probability measures on $\mathbb{R}^{Nd}$  describing the parameters of the $N$-neuron neural network, i.e.,  $\mu_N(\rmd \bar{\beta} \rmd \bar{\alpha})=\sum_{i=1}^N\delta_{\beta^{i,N}}(\rmd \bar{\beta}^i)\delta_{\alpha^{i,N}}(\rmd \bar{\alpha}^i)$.

Besides being high-dimensional, this optimization problem is in general non-convex, and hence, it is hard, if not impossible, to solve.  Remarkably, in the infinite neuron limit $N \nearrow \infty$,  the problem of finding the optimal parameters in $\mathbb{R}^{N d}$ of the shallow neural network turns into a \emph{convex} optimization problem over the space of probability measures on $\mathbb{R}^d$ \cite{MeMoNg18,chizat2018global,sirignano2020mean,RoVa2022, HuReSiSz21}.
In particular, 
 with a suitable regularisation term given e.g.~by the relative entropy with respect to the $d$-dimensional standard normal distribution, the optimization problem  turns into 
\begin{align*}
\min_{\mu\in\mathcal{P}(\mathbb{R}^d) } \Big\{\int_{\mathbb{R}\times\mathbb{R}^{d-1}} | y - \mathbb{E}_{(\beta,\alpha)\sim \mu} [\beta \phi(\alpha \cdot \mathbf{z} )]|^2 \nu(\rmd y  \rmd \mathbf{z}) + H(\mu|\mathcal{N} (0, I_d)) \Big\} \;.
\end{align*}
The minimizer is now a nonlinear probability measure on $\mathbb{R}^d$.
To see the precise form of this measure, it helps to introduce $x=(\beta, \alpha)$ and $\varphi(x, \mathbf{z})=\beta\phi(\alpha \cdot \mathbf{z} )$. Then, as shown e.g.~in \cite{HuReSiSz21}, the minimizer is   \begin{align*}
\bar{\mu}_*(\rmd x) \propto \exp\Big(-\frac{|x|^2}{2}+ \int_{\mathbb{R}^d} 2\varphi(x, \mathbf{z})\Big(-y+\int_{\mathbb{R}^d} \varphi(\tilde{x},\mathbf{z})\bar{\mu}(\rmd \tilde{x})\Big) \nu(\rmd y \rmd \mathbf{z})\Big) \rmd x \;.
\end{align*}
In fact, $\bar{\mu}_*$ is  of the form \eqref{eq:invmeas_nonl} with 
\begin{align*}
& V(x)=\frac{|x|^2}{2}+2\int_{\mathbb{R}^d}y\varphi(x,\mathbf{z})\nu(\rmd y \rmd \mathbf{z}), \quad \text{and} \quad
 W(x,\tilde{x}) =2\int_{\mathbb{R}^d} \varphi(x,\mathbf{z})\varphi(\tilde{x},\mathbf{z})\nu(\rmd y \rmd \mathbf{z}) \;.
\end{align*}
The corresponding gradients are
\begin{align*}
&\nabla V(x)=x+2\int_{\mathbb{R}^d}y\nabla_x\varphi(x,\mathbf{z})\nu(\rmd y \rmd \mathbf{z}),\quad \text{and}
\quad \nabla_1 W(x,\tilde{x}) =2\int_{\mathbb{R}^d} \nabla_x\varphi(x,\mathbf{z})\varphi(\tilde{x},\mathbf{z})\nu(\rmd y \rmd \mathbf{z}) \;.
\end{align*}
Note that these functions and their gradients depend strongly on the data $(\mathbf{z},y)$.

The regularity and convexity assumptions in \Cref{ass} impose restrictions on the data and activation function.  In particular,
\ref{ass_Vconv} requires  the regularisation term to be sufficiently large such that $V$ is asymptotically  strongly convex. 
Moreover, the global gradient-Lipschitz continuity of $V$, $W$ in \ref{ass_Vlip} and \ref{ass_Wlip} clearly depends on global Lipschitz continuity of the underlying activation function $\phi$, which may not be satisfied in practice.  For example, if  $\phi$ is the sigmoid function, we can not guarantee  these assumptions hold unless we consider a measure $\bar{\mu}_*$ whose marginal distribution in the first component is confined to a finite interval.

\printbibliography

\end{document}